\documentclass[11pt]{article}

\usepackage{listings}
\usepackage[colorlinks=false, backref]{hyperref}
\usepackage{amsmath,amsthm,amsfonts,amssymb,amscd}
\usepackage{float}
\usepackage{graphicx}

\usepackage[T1]{fontenc}
\usepackage{lmodern}

\usepackage[mathscr]{euscript}
\usepackage{stmaryrd}
\usepackage{microtype}
\usepackage{booktabs}
\usepackage{cleveref}
\usepackage{bookmark}
\usepackage{mathrsfs}

\usepackage{emptypage}
\usepackage{tikz}
\usepackage{subfigure}
\usepackage[utf8]{inputenc}
\usepackage[T1]{fontenc}

\usepackage{mathabx}
\usepackage{extarrows}
\usepackage{fancyhdr}
\usepackage{textgreek}

\theoremstyle{plain}
\newtheorem{theorem}{\scshape Theorem}

\newtheorem{lemma}[theorem]{\scshape Lemma}

\theoremstyle{definition}

\newtheorem{remark}{\scshape Remark}

\theoremstyle{definition}

\DeclareMathAlphabet{\mathpzc}{OT1}{pzc}{m}{it}
\DeclareMathAlphabet\mathbfcal{OMS}{cmsy}{b}{n}

\DeclareMathOperator*{\esssup}{ess\,sup}

\newcommand{\la}{\langle}

\newcommand{\pa}{\partial}
\newcommand{\veps}{\varepsilon}
\newcommand{\eps}{\epsilon}
\newcommand{\opn}{\operatorname}

\newcommand{\im}{\opn{Im}}

\newcommand{\mbb}[1]{\mathbb{#1}}

\newcommand{\mc}[1]{\mathcal{#1}}
\newcommand{\jd}{\displaystyle}

\newcommand{\der}[2]{\frac{\partial #1}{\partial #2}}

\newcommand{\htil}{\widetilde{h}}
\newcommand{\phitil}{\widetilde{\varphi}}
\newcommand{\mat}[1]{\begin{pmatrix} #1 \end{pmatrix}}
\newcommand{\diag}{\opn{diag}}

\def\bbN{{\mathbb N}}

\def\bbR{{\mathbb R}}

\def\bbZ{{\mathbb Z}}

\def\D{{\mathcal D}}

\def\O{{\mathcal O}}

\def\Q{{\mathcal Q}}

\def\V{{\mathcal V}}

\def\bN{{\mathbf N}}

\def\bfn{{\mathbf n}}

\def\rf{{\rm f}}
\def\rg{{\rm g}}

\def\e{\text{\bf\emph{e}}}

\def\bfn{\text{\bf\emph{n}}}

\def\p{\text{\bf\emph{p}}}

\def\u{\text{\bf\emph{u}}}
\def\v{\text{\bf\emph{v}}}

\definecolor{grey}{rgb}{0.5,0.5,0.5}
\definecolor{lightgrey}{rgb}{0.9,0.9,0.9}
\definecolor{darkgreen}{rgb}{0,0.6,0}
\definecolor{orange}{rgb}{1,0.5,0}
\definecolor{lightpink}{rgb}{1,0.714,0.757}
\definecolor{lightorange}{rgb}{1,0.855,0.725}

\DeclareMathOperator{\sgn}{sgn}

\def\div{{\operatorname{div}}}

\def\curl{{\operatorname{curl}}}

\def\ft #1{{\widehat{#1}}}

\def\bp{{\partial_1}}
\def\p{{\partial\hspace{1pt}}}

\def\Forall{\forall\hspace{2pt}}

\def\comm#1#2{{\llbracket#1,#2\rrbracket}}

\def\({{(\hspace{-2pt}(}}
\def\){{)\hspace{-2pt})}}

\def\smallexp#1{{\text{\small #1}}}
\def\footnoteexp#1{{\text{\footnotesize #1}}}
\def\smallint#1#2{{\text{\small$\displaystyle{}\int_{#1}^{#2}$}}}

\def\dfrac#1#2{\smallexp{$\displaystyle{}\frac{#1}{#2}$}}
\def\ddfrac#1#2{\footnoteexp{$\displaystyle{}\frac{#1}{#2}$}}

\def\XXint#1#2#3{{\setbox0=\hbox{$#1{#2#3}{\int}$}
\vcenter{\hbox{$#2#3$}}\kern-.5\wd0}}

\def\nn{{\scriptstyle \mathbfcal{N} }}

\def\tt{{\scriptstyle \mathbfcal{T} }}

\def\hinit{h_{\operatorname{init}}}
\def\fhinit{\widehat{{ h}}_{\operatorname{init}}}
\def\htinit{{\dot h}_{\operatorname{init}}}
\def\fhtinit{\widehat{{\dot h}}_{\operatorname{init}}}

\title{{\bf Rigorous Asymptotic Models of Water Waves }}

\author{
{\small {\bf Arthur Cheng}}
\vspace{-.05 in}
\\{\footnotesize Department of Mathematics}
\vspace{-.05 in}
\\{\footnotesize National Central University}
\vspace{-.05 in}
\\{\footnotesize Jhongli, Taoyuan 32001 Taiwan}
\vspace{-.05 in}
\\{\footnotesize email: {\it cchsiao@math.ncu.edu.tw}}
\and
{\small {\bf Rafael Granero-Belinch\'on}}
\vspace{-.05 in}
\\{\footnotesize Departamento de Matem\'aticas, Estad\'istica y Computaci\'on}
\vspace{-.05 in}
\\{\footnotesize Universidad de Cantabria}
\vspace{-.05 in}
\\{\footnotesize Santander, Espa\~na}
\vspace{-.05 in}
\\{\footnotesize email: {\it rafael.granero@unican.es}}
\and
 {\small {\bf Steve Shkoller}}
 \vspace{-.05 in}
\\{\footnotesize Department of Mathematics}
\vspace{-.05 in}
\\{\footnotesize University of California}
\vspace{-.05 in}
\\{\footnotesize Davis, CA 95616 USA}
\vspace{-.05 in}
\\{\footnotesize email: {\it shkoller@math.ucdavis.edu}}
\and
{\small {\bf Jon Wilkening}}
\vspace{-.05 in}
\\{\footnotesize Department of Mathematics}
\vspace{-.05 in}
\\{\footnotesize University of California}
\vspace{-.05 in}
\\{\footnotesize Berkeley, CA 94720 USA}
\vspace{-.05 in}
\\{\footnotesize email: {\it wilken@math.berkeley.edu}}
}
\date{\today}

\begin{document}

\maketitle

{\footnotesize
  {\bf Abstract.}
We develop a rigorous asymptotic derivation of two mathematical models
of water waves that capture the full nonlinearity of the Euler
equations up to quadratic and cubic interactions, respectively.
Specifically, letting $ \epsilon $ denote an asymptotic parameter
denoting the steepness of the water wave, we use a Stokes expansion in
$ \epsilon $ to derive a set of linear recursion relations for the
tangential component of velocity, the stream function, and the water
wave parameterization.  The solution of the water waves system is
obtained as an infinite sum of solutions to linear problems at each
$O( \epsilon ^k)$ level, and truncation of this series leads to our
two asymptotic models, which we call the quadratic and cubic
$h$-models.

Using the growth rate of the Catalan numbers (from number theory), we
prove well-posedness of the $h$-models in spaces of analytic
functions, and prove error bounds for solutions of the $h$-models
compared against solutions of the water waves system.  We also show
that the Craig-Sulem models of water waves can be obtained from our
asymptotic procedure and that their WW2 model is well-posed in our
functional framework.

We then develop a novel numerical algorithm to solve the quadratic and
cubic $h$-models as well as the full water waves system. For three
very different examples, we show that the agreement between the model
equations and the water waves solution is excellent, even when the
wave steepness is quite large.  We also present a numerical example of
corner formation for water waves.

}

{\small
\tableofcontents}

\section{Introduction}

Both gravity and capillary 
water waves are modeled by the free-surface incompressible Euler equations of fluid dynamics, and for many applications,  the fluid is
additionally assumed to be  irrotational.  Well-posedness, stability, and  singularity formation have been well studied with many results;
 see, for example,  \cite{nalimov1974cauchy,yosihara1982gravity, Craig:existence-theory-water-waves, Beale-Hou-Lowengrub:convergence-boundary-integral, 
 Ambrose-Masmoudi:zero-surface-tension-2d-waterwaves,Wu:well-posedness-water-waves-2d, Wu:almost-global-wellposedness-2d,
  Wu:global-wellposedness-3d, AlvarezSamaniego-Lannes:large-time-existence-water-waves, Lannes1, Lindblad:well-posedness-motion, 
  Coutand-Shkoller:well-posedness-free-surface-incompressible, shatah2011local, ChCoSh2008, ChCoSh2010, alazard2014cauchy, hunter2016two, 
  Germain-Masmoudi-Shatah:global-solutions-gravity-water-waves-annals, ionescu2015global, ifrim2014two, alazard2015global, 
  castro2012splashannals, coutand2014finite, FeIoLi2016,CoSh2016, deng2016global}.
  However, the Euler equations are sufficiently complicated that for many physical scenarios, a precise 
  understanding of the dynamics of the solutions to the full water waves problem is not (at this time) known.    Consequently, 
since the pioneering works of Airy, Boussinesq and Stokes \cite{airy1841tides, Bo1872, Bo1877,stokes1847theory},  there has been a sustained effort to find
suitable  approximations of the Euler equations,  specific  to certain asymptotic regimes.  Such approximate asymptotic models have closely
related dynamics and can be significantly easier to analyze.  Herein, we develop an asymptotic procedure that yields approximate model
equations for the water waves problem to various orders of approximation of the nonlinearity.   In particular, we present two models that 
respectively capture the nonlinearity up to 
 quadratic and cubic interactions.

We derive two asymptotic models for the evolution of both gravity and gravity-capillary waves in deep water, using an asymptotic expansion in the steepness of the wave $ \epsilon $,  which we view as a small parameter, equivalent  to the ratio of the amplitude to the wavelength.  Such an expansion  has been used extensively since it was introduced by Stokes \cite{stokes1847theory}; see, for example, \cite{AkNi2010, NiRe2005, NiRe2006, AkNi2012, AkNi2014}).
Starting with the case of gravity water waves,  we employ such a  Stokes expansion and obtain  linear recursion relations for the stream function,
the tangential  component of velocity, and the free-surface parameterization. Truncating  this expansion to $O( \epsilon ^3)$ yields a quadratic model equation for  gravity water waves. We refer to this PDE as the \emph{quadratic $h$-model}
\begin{equation}
\p_{\!t}^2 h + g \Lambda h = - \Lambda (|H \p_{\!t}h|^2) + g \Lambda (h \Lambda h) + g \p_{\!1} (h \p_{\!1} h)\,.
\end{equation} 
Keeping
all terms in the recursion relation to  $O( \epsilon ^4)$ yields  the \emph{cubic $h$-model}, a new model of  water wave dynamics that accurately captures the cubic interactions of the Euler equations and is given by
\begin{equation}
\p_{\!t}^2 h + g \Lambda h = -\Lambda \big[(H\p_{\!t} h)^2\big] + g \p_{\!1} \big(h \p_{\!1} h \big) +g\Lambda \big(h \Lambda h \big)+ \mathcal{Q}(h)
\end{equation}
where the cubic nonlinearity $\mathcal{Q}(h)$ is defined  in \eqref{R(h)}.   Asymptotic models for  gravity-capillary waves are derived in the same fashion (in section \ref{sectionST}) when
 gravity and surface tension forces are of the same order.

The same expansion procedure that we used for  the one-fluid problem can be used to derive 
 models for two-fluid internal waves and
 the Rayleigh-Taylor instability (as noted in Remark \ref{remark}). Furthermore, out approach can be applied to the case of finite depth fluids as well.

We note that the quadratic $h$-model was  first derived by Akers \& Milewski \cite{AkMi2010} using a formal asymptotic procedure in which they
assumed that due to very small amplitudes of the water wave, it could be assumed that all elliptic problems are set on the time-independent half-space.
  A Stokes expansion procedure was then used by Akers \& Nicholls  \cite{AkNi2010} with a time-dependent fluid domain,
but for the case of the traveling wave ansatz.

We prove that both the quadratic and cubic $h$-models  are well-posed in  spaces of analytic functions that are similar to the
Wiener algebra but with a (growing) exponential weight (used to guarantee analyticity).   As we noted above, our methodology relies upon a sequence of
linear recursion relations obtained via the Stokes expansion, and thus it is possible (in principle) to solve a nonlinear PDE by a summation of an infinite series, each term in the series coming from a 
solution to a linear problem.   The objective, then, is to prove summability of the infinite series; however, classical contraction mapping techniques fail due to the
growth of the norm of the  $k$th linear solution.   Remarkably, the growth of these norms can be quantified and estimated in terms of the 
Catalan numbers \cite{St2015} from number theory (which are often used in combinatorics), and the well-known bounds on the Catalan numbers then permit the convergence of the infinite
series to the solution of the nonlinear problem.\footnote{It may be possible to establish existence of solutions to the $h$-models using some type of 
Cauchy-Kovalevsky theorem, but we are not aware of a particular form of the theorem that would be directly applicable.}
We also establish  rigorous error bounds for the difference between solutions of the $h$-models and the full water waves system.
We thus conclude that both the quadratic and cubic  $h$-models are accurate asymptotic models of water waves in  the small $ \epsilon $-regime.

The asymptotic procedure that we shall describe below allows us to derive a large class of asymptotic models of water waves, including the well-known
hierarchy of models obtained by Craig \& Sulem \cite{CrSu1993}; in particular, we show that their most studied model, WW2 (or water waves 2),  is obtained from
our approach, and we prove that it too is well-posed  in  spaces of analytic functions.   Moreover, we write the WW2 model as a second-order wave equation and
explain its connection with the quadratic $h$-model.

Finally, we present an arbitrary-order exponential time differencing
scheme
\cite{cox-matthews:etd-2002,kassam-trefethen:etd-2006,chen-wilkening:setd}
for solving the quadratic and cubic $h$-models accurately and
efficiently and compare those solutions against numerical solutions of
the Euler equations.  We show that the $h$-models converge as
expected: with $\eps$ denoting the maximum slope of the initial
condition, the quadratic and cubic $h$-models converge in $L^2$ to
solutions of the full water waves problem with rates $O(\epsilon^2)$
and $O(\epsilon^3)$, respectively, where the $L^2$ error is scaled by
$\eps^{-1}$ to account for the decreasing (as a function of $ \epsilon $) 
norm of the exact solution.
We give three examples of initial data that show excellent agreement
between the $h$-models and the full water waves solution all the way
up to $\eps=O(1)$. The first example is a multi-hump initial condition
in which a jet forms in each trough as the solution drops from rest;
the second example is a localized disturbance over a flat surface that
propagates outward as time evolves; and the third example is a family
of standing water waves. In all three cases, the quadratic and cubic
models are much better than linear theory at capturing features of the
dynamics.  For large $\epsilon$, the quadratic model has a tendency to
form a corner singularity while the cubic model tracks the Euler
solution quite well. We also present a continuation of the first
example for the Euler equations to show that the wave eventually
overturns and appears to form a corner singularity before
self-intersecting, with $dP/dn\rightarrow0$ at the tip of one of the
overturning waves.

{\bf Paper Outline.} 
In Section \ref{sec:notation}, we introduce the notation and some important definitions used throughout the paper.  In Section \ref{sec:ww-equations}, we introduce
the water waves equations, and the three fundamental variables that shall be evolved: the tangential component of velocity, the stream function, and the free-surface
parameterization.   Section \ref{sec:stokes_expansion} is devoted the Stokes expansion of the water waves system and the derivation of the linear recursion 
relations.  In Section \ref{sec:hmodels}, we derive the quadratic and cubic $h$-models, and in Section \ref{Catalansection}, we prove that these models are well-posed.    
In Section \ref{sec:CSWW2}, we derive the Craig-Sulem WW2 model, and prove that it too is well-posed. 
Section \ref{sec:error} establishes the error estimates for solutions of the $h$-models compared to the full water waves system. Then, in Section \ref{sec:numerics}, we perform a number
of numerical experiments that compare the quadratic and cubic $h$-models with a highly accurate numerical solution of the full water waves system.

\section{Some notation and definitions}\label{sec:notation}

\subsection{Matrix indexing} Let $A$ be a matrix, and $b$ be a column vector. Then, we write $A^i_j$ for the component of $A$, located on row $i$ and column $j$;  consequently, using the Einstein summation convention, we write
$$
(Ab)^k=A^k_ib^i\text{ and }(A^Tb)^k=A^i_k b^i.
$$

\subsection{Power series summation}
We adopt the convention that independent of the summand $s_j$,
\begin{equation}\label{summation}
\sum_{j=0}^{k-\ell} s_j=0 \ \ \text{ whenever } k<\ell \,.
\end{equation}

\subsection{The water wave   parameterization}
We identify $ \mathbb{S}  ^1$ with the interval $[-\pi, \pi]$.   We shall denote a general parameterization of the free-surface
of the fluid by the diffeomorphism
$z( \cdot , t) : \mathbb{S}^1 \to \mathbb{R}^2  $.  This free-surface of the fluid is the water wave, which we denote by $\Gamma(t)$.  Hence, the water wave is given by
$$\Gamma(t)= \big\{\big(z_1(x_1,t), z_2(x_1,t)\big) \ : \ -\pi \le x_1 \le \pi\,,  \ t\in[0,T] \big\}  \,.$$

For the majority of our analysis, we shall assume that the water wave evolves as a graph over the horizontal $x_1$-axis.   in particular,  $(z_1,z_2)=(x_1, h(x_1,t))$ and
\begin{equation}\label{Gamma}
\Gamma(t)= \big\{\big(x_1, h(x_1,t)\big) \ : \ -\pi \le x_1 \le \pi\,,  \ t\in[0,T] \big\}  \,.
\end{equation}
The one-to-one function $h(x_1,t)$ is often called the signed {\it height} function.

\subsection{The fluid domain and some geometric quantities}
The time-dependent fluid domain is defined as
\begin{equation}\label{Omega}
\Omega(t) = \big\{ (x_1, x_2) \ : \ -\pi \le x_1 \le \pi\,, -\infty \le x_2 \le z_2(x_1,t)\,, \ t\in[0,T] \big\},
\end{equation}
\emph{i.e.} for the sake of simplicity, we assume that the depth of the fluid is much larger than the amplitude of the wave.

We define the reference domain $\D$ as
\begin{equation}\label{D}
\D = \mathbb{S}^1 \times (-\infty, 0) \,.
\end{equation}
We let $\bN =\e_2$ denote the outward unit normal to $\p \D $, and we let $\tt(\cdot ,t)$ and $\nn( \cdot , t)$ denote, respectively, the unit tangent and normal vectors to
the water wave $\Gamma(t)$, where $\nn( \cdot ,t)$ points outward to the set $\Omega(t)$.
We then set
$$
\bfn = \nn \circ z\text{ and } \boldsymbol{\tau}=\tt \circ z \,.
$$
When the water wave is defined by  graph parameterization \eqref{Gamma}, the induced  metric $\rg$ is given by
\begin{equation}\label{metric}
\rg =  1+ (\p_{\!1} h)^2\,.
\end{equation}

\subsection{Derivatives} We write
$$
\partial_k f=\frac{\partial f}{\partial x_k}\ \text{ for } k=1,2\,, \ \  \partial_t f=\frac{\partial f}{\partial t} \,, \ \
\nabla=(\p_{\!1}, \p_2)\,,    \ \ \nabla^\perp=(- \p_2, \p_{\!1}) \,,
$$
and for a vector $F$,
$$
\operatorname{div} F=  \nabla \cdot F  \ \text{ and } \ \operatorname{curl} F = \nabla ^\perp\cdot F \,.
$$
The Laplace operator is defined as $ \Delta = \p_{\!1}^2 + \p_2^2$.

\subsection{Fourier series}\label{sec::Fourier}
If  $f:\mathbb{S}^1\to\bbR$ is a square-integrable $2\pi$-periodic function,  then it has the Fourier series representation
$f(x_1) = \sum\limits_{k=- \infty }^ \infty  \ft{f}(k) e^{ikx_1}$ for all $x_1 \in \mathbb{S}^1$, where the complex Fourier coefficients
are defined by $\hat f(k) = {\dfrac{1}{2\pi}} \smallint{\mathbb{S}^1}{}\, f(x_1)
e^{-ikx_1}dx_1$.  We shall sometimes write $\hat f_k$ for $\hat f(k)$.
Functions  $g:\D \to \bbR$ (which are square-integrable in $x_1$ can be expanded as
$
g(x_1,x_2) = \sum\limits_{k\in \bbZ} \ft{g}(x_2,k) e^{ikx_1}$ for all $(x_1,x_2) \in \D$, where
$\ft{g}(x_2,k) = {\dfrac{1}{2\pi}} \smallint{\mathbb{S}^1}{}\, g(x_1, x_2)
e^{-ikx_1}dx_1$.

\subsection{Singular integral operators}   Let $f(x_1)$ denote a $2\pi$ periodic function on $\mathbb{S}^1$.
Using the Fourier representation, we define the
Hilbert transform $H$ and the Dirichlet-to-Neumann operator $ \Lambda $ as
\begin{align}\label{Hilbert}
\ft{Hf}(k)=-i\text{sgn}(k) \ft{f}(k) \,, \ \
\ft{\Lambda f}(k)&=|k|\ft{f}(k)\,.
\end{align}
In particular, we note that
$$
\partial_1 H=\Lambda,\quad H^2=-1.
$$
Equivalently, suppose that $f:\mathbb{S}^1 \to \mathbb{R}  $ is a $2\pi$ periodic function and that $\Phi$ is its harmonic extension to $\D$.
Then,
\begin{equation}\label{Lambda_def}
\Lambda f= \p_2 \Phi \text{ \ on \ } \mathbb{S}^1\times\{0\} \,.
\end{equation}
Finally, we denote the commutator between $f$ and the Hilbert transform acting on $g$ as
$$
\comm{H}{f}g=H(fg)-fHg.
$$
Let us observe that $\comm{f}{H}g=-\comm{H}{f}g.$
\subsubsection{Function spaces}  For $1 < p \le \infty $, we denote by $L^p(\mathbb{S}^1)$ the set of Lebesgue measurable $2\pi$-periodic functions such that
$\|u\|_{L^p} < \infty $ where $\|u\|_{L^p} = \left(\smallint{\mathbb{S}^1}{}\, |u(x)|^p\,dx \right)^ {\frac{1}{p}} $ if $1< p< \infty $ and $\|u\|_{L^ \infty } = \esssup_{x\in \mathbb{S}^1} |u(x)|$.
For integers $k \ge 0  $,  we let $H^k(\mathbb{S}^1)=
\Big\{ u: L^2(\mathbb{S}^1)  \ \Big| \  \|u\|^2_{H^k}:= \sum\limits_{j=0}^k \|\p_{\!1}^j u\|_{L^2} < \infty \Big\}$.    For $s \in \mathbb{R}  $,  we then define the space $H^s(\mathbb{S}^1)$ to be
the  $2\pi$-periodic distributions such that
$  \|u\|^2_{H^s}:= \sum\limits_{m=- \infty}^ \infty (1+ m^2)^s |\hat u_m|^2 < \infty $.

For a given $\tau>0$, we define the following Banach scale of analytic functions as
\begin{equation}\label{X_tau}
X_\tau=\Big\{u: \mathbb{S}^1 \to \mathbb{R}  \,\Big|\,\|f\|_{X_\tau} = \smallexp{$\sum\limits_{m\in\bbZ}$} e^{\tau|m|}|\hat{u}_m|<\infty\Big\}.
\end{equation}

\section{Water waves equations}\label{sec:ww-equations}
Water waves are modeled by the incompressible and irrotational free-surface Euler equations, written as
\begin{subequations}\label{ww}
\begin{alignat}{2}
\partial_t\u + (\u \cdot \nabla) \u + \nabla p &= {\bf 0} &&\text{in}\quad \Omega(t)\,,\\
\operatorname{curl} \u=\div \u &= 0 &&\text{in}\quad \Omega(t)\,,\\
p&=  -\lambda \frac{\partial_1^2h}{(1+(\bp h)^2)^{3/2}}  \qquad&&\text{on}\quad \Gamma(t)\,,\\
\u &= \u_0 &&\text{on}\quad\Omega \times \{t=0\}\,,\\
\V(\Gamma(t)) &= \u \cdot \nn  \,,
\end{alignat}
\end{subequations}
where $t  \in [0,T]$, $\Omega(t)$ is defined in \eqref{Omega},  $\Gamma(t)$ is defined in \eqref{Gamma}, $0\leq\lambda$ is the surface tension parameter and $\V(\Gamma(t))=\u \cdot \nn$ means  that the free-surface $\Gamma(t)$ moves with normal velocity $\u \cdot \nn$.
 We shall assume that all functions are $2\pi$-periodic in $x_1$.
\begin{center}
\begin{tikzpicture}[scale=0.5]
    \draw (21,1.) node { $\Gamma(t)$};
    \draw (15.2,-1.5) node { $ \Omega(t)$ };
    \draw[color=blue,ultra thick] plot[smooth,tension=.6] coordinates{( 10,0) (11,-2) (13,-1) (15,1) (17, 0) (19, 2) (21,-1) };
\end{tikzpicture}
\end{center}

\subsection{The Bernoulli equation}
Since $\curl \u = 0$ in $\Omega(t)$, $\u = \nabla \phi$ for some scalar potential $\phi$. Then,
 (\ref{ww}a) can be written as
\begin{equation}\label{Bernoulli_1phase}
 \partial_t\phi(x,t) + \frac{1}{2} |\nabla \phi(x,t)|^2 + p(x,t) + \rho gx_2 =  f(t) \quad\Forall x \in \Omega(t) \text{ and } t>0\,.
\end{equation}
where $f$ is a function independent of $x$.

\subsection{The evolution of the tangential velocity}
On  $\mathbb{S}^1$,  we define the following
quantities:
$$
\v = \u \circ z\,,\quad \Psi = \phi \circ z\,,\quad\text{and}\quad \bfn = \nn \circ z\,.
$$
We shall make use of the {\it tangential velocity}
$$
\omega = \u \cdot \tt  \ \text{ and } \
\widebar{\omega}(x_1, t) = \v(x_1, t) \cdot \partial_1 z( x_1, t) \  \text{ on } \ \Gamma(t) \,.
$$
From (\ref{ww}e),  $\partial_t z \cdot \bfn = \v \cdot \bfn$ so that
\begin{equation}\label{ss2}
\partial_t z(x_1,t) = \v(x_1,t) + c(x_1,t) \p_{\!1} z(x_1,t)
\end{equation}
for an arbitrary scalar function $c$.   (Note that $\p_{\!1} z$ is a tangent vector and that the water waves problem has a tangential reparameterization symmetry.)

By the chain rule,
\begin{align}
\p_{\!1} \Psi(x_1,t) &= \p_{\!1} z(x_1,t) \cdot (\nabla \phi)\circ z(x_1,t)
 = (\v \cdot \p_{\!1} z)(x_1,t)
 = \widebar{\omega}(x_1,t)\,,\label{ss1}
\end{align}
and
\begin{align*}
\partial_t \Psi(x_1,t) &= \partial_t \phi(z(x_1,t),t) + \partial_t z(x_1,t) \cdot (\nabla \phi) \circ z(x_1,t)\\
&= (\partial_t \phi \circ z + \partial_t z \cdot \v)(x_1,t)\,.
\end{align*}
From \eqref{ss1}, $\p_{\!1} \partial_t \Psi(x_1,t) = \partial_t \widebar{\omega}(x_1,t)$, and (\ref{ss2}) shows that
$$
\partial_t \Psi(x_1,t) = (\partial_t \phi \circ z)(x_1,t) + |\v(x_1,t)|^2 + c(x_1,t) \,\widebar{\omega}(x_1,t)\,.
$$
Therefore, we find that $\widebar{\omega}$ satisfies
\begin{equation}\label{omega_evolution_eq_1phase_temp}
\partial_t \widebar{\omega} = \p_{\!1} \big(\partial_t \phi \circ z + |\v|^2 + c \,\widebar{\omega} \big)\,.
\end{equation}
From (\ref{Bernoulli_1phase}),
$$
(\partial_t \phi \circ z)(x_1,t) + \frac{1}{2} |\v(x_1,t)|^2 + (p\circ z)(x_1,t) + gz_2(x_1,t) = f(t) \,,
$$
so that
\begin{align*}
\p_{\!1} (\partial_t \phi \circ z)& = - \p_{\!1} \big(\dfrac{1}{2} |\v|^2 + (p\circ z)+ gz_2\big)
 \,,
\end{align*}
where we have used the boundary condition (\ref{ww}c) in the last equality.
Using  (\ref{omega_evolution_eq_1phase_temp}), we find that
\begin{equation}\label{omega_evolution_eq_1phase}
\partial_t \widebar{\omega} = \p_{\!1} \big(\dfrac{1}{2} |\v|^2 + c \,\widebar{\omega}  + \lambda \frac{\partial_1^2h}{(1+(\bp h)^2)^{3/2}}  - gz_2\big)  \,.
\end{equation}

We now suppose that the interface $\Gamma(t)$  remains a graph and is given by \eqref{Gamma}.
With the definition of the metric \eqref{metric}, we write the unit normal and tangent vectors, respectively, to $\Gamma(t)$ as
$$\bfn = \rg^{-\frac{1}{2}} (-\p_{\!1} h,1) \,, \text{ and } \boldsymbol{\tau} = \rg^{-\frac{1}{2}} (1, \p_{\!1} h) \,.$$
Using (\ref{ww}e), we decompose $\v$ as follows:
\begin{align}
\v&= (\v \cdot \bfn) \bfn + (\v\cdot {\boldsymbol\tau}) {\boldsymbol\tau}
 = \rg^{-\frac{1}{2}}  \partial_t h \, \bfn \ + \rg^{-\frac{1}{2}} \widebar{\omega} \, {\boldsymbol\tau} \,, \label{v_decomposition}
\end{align}
and hence
\begin{align}
|\v|^2 &= \rg ^{-1} (|\partial_t h|^2  + |\bar \omega|^2)  \,. \label{v-squared}
\end{align}
Equations (\ref{ss2}) and (\ref{v_decomposition}) then provide us with the identity
\begin{align*}
c &= \p_t z_1 - \v_1 = - \v_1
= \rg^{-1}\left( \partial_t h \p_{\!1} h -  \widebar{\omega}\right)\,,
\end{align*}
so that (\ref{omega_evolution_eq_1phase}) can be written as
\begin{align}
\partial_t \widebar{\omega} &= - g \p_{\!1} h + \p_{\!1} \big[\dfrac{1}{2} \rg ^{-1} (|\partial_t h|^2  + |\bar \omega|^2) - \rg^{-1} |\widebar{\omega}|^2 + \rg^{-1} \widebar{\omega} \partial_t h \p_{\!1} h \big]+\lambda \partial_1\left(\frac{\partial_1^2h}{(1+(\bp h)^2)^{3/2}}\right) \nonumber\\
&= - g \p_{\!1} h + \frac{1}{2} \p_{\!1} \big[\rg ^{-1} \big(|\partial_t h|^2  - |\bar \omega|^2 + 2 \widebar{\omega} \partial_t h \p_{\!1} h\big) \big]+\lambda \partial_1\left(\frac{\partial_1^2h}{(1+(\bp h)^2)^{3/2}}\right)\label{omega_evolution_eq_1phase_h}\,.
\end{align}

\subsection{The equation for the stream function}
Since $\Omega(t)$ is simply connected, by classical Hodge theory, we can uniquely determine the velocity vector $\u$ by solving the following elliptic system:
\begin{equation} \label{onephaseelliptic}
\operatorname{curl}  \u = 0  \text{ and }
\operatorname{div} \u  = 0 \ \text{ in }\Omega(t)\,, \ \ \text{ and } \ \
\u \cdot \tt = \omega \ \text{ on }\ \Gamma(t)\,.
\end{equation}

Solutions of  \eqref{onephaseelliptic} have the form $\u = \nabla^\perp \vartheta$ for some stream function $\vartheta$ which satisfies the scalar Neumann problem
\begin{equation}\label{theta_eq_1phase}
\Delta \vartheta = 0 \ \text{ in } \ \Omega(t)\,, \ \ \text{ and } \ \
\dfrac{\p \vartheta}{\p \nn} = -\omega \ \text{ on }\  \Gamma(t)\,.
\end{equation}
Existence, uniqueness, and regularity of solutions to \eqref{theta_eq_1phase} is classical when $\Gamma(t)$ is sufficiently smooth and $\smallint{\Gamma(t)}{}\, \omega dS(t)=0$;
see \cite{ChSh2017} for the case that $\Gamma(t)$ is of Sobolev class.

\subsection{The evolution equation for the free-surface}
We extend the parameterization \eqref{Gamma} to a diffeomorphism $\psi$ of $\D$ as
\begin{equation}\label{psi_ww}
\psi(x_1,x_2) = (x_1, x_2 + h(x_1,t))\qquad \Forall (x_1,x_2) \in \D\,,
\end{equation}
and set
\begin{equation}\label{A_matrix}
\nabla \psi = \left[\begin{matrix}
1 & 0 \\
\partial_1 h & 1
\end{matrix}\right],\quad
A = (\nabla \psi)^{-1} = \left[\begin{matrix}
1 & 0 \\
-\partial_1 h & 1
\end{matrix}\right]\,.
\end{equation}
We then define the stream function  on the reference domain $\D$ as $\varphi = \vartheta \circ \psi$.   We then compute that
$$
\v = \u \circ \psi
= (\nabla^\perp \vartheta)\circ \psi
= (-A^k_2 \partial_k \varphi, A^k_1 \partial_k \varphi)
= (-\partial_2 \varphi, \partial_1 \varphi - \bp h \partial_2 \varphi)\,.
$$
From (\ref{ww}e),
$\partial_t h(x_1,t) = \v \cdot (-\bp h,1) $ on $\mathbb{S}^1$, so that
\begin{equation}\label{h_eq_1phase}
\partial_t h = \bp\varphi \qquad\text{on}\quad \mathbb{S}^1\,.
\end{equation}

\section{Stokes expansion and linear recursion for the time-dependent water waves}\label{sec:stokes_expansion}
\subsection{Stokes expansion}
Letting $0<\epsilon < 1 $ denote the steepness parameter (which can be
  viewed as the ratio of the amplitude to characteristic wavelength),
we consider the following Stokes expansion \emph{ansatz}:
\begin{equation}
\label{ansatz}
h(x_1,t)   =  \epsilon  \widetilde{h}(x_1, t) \,, \  \
\varphi(x,t)  = \epsilon  \widetilde{\varphi}( x, t)\,, \ \
\widebar{\omega}(x_1,t)  = \epsilon  \widetilde{\omega}( x_1, t) \,,
\end{equation}
where
\begin{subequations}\label{expansions}
\begin{align}
\widetilde{h}(x_1,t) &= h_0(x_1,t) + \epsilon h_1(x_1,t) + \epsilon^2 h_2(x_1,t) + \cdots\,,\\
\widetilde{\varphi}(x,t) &= \varphi_0(x,t) + \epsilon \varphi_1(x,t) + \epsilon^2 \varphi_2(x,t) + \cdots\,,\\
\widetilde{\omega}(x_1,t) &= \omega_0(x_1,t) + \epsilon \omega_1(x_1,t) + \epsilon^2 \omega_2(x_1,t) + \cdots\,.
\end{align}
\end{subequations}
In particular, at the initial time $t=0$, we have
\begin{equation}\label{initialdata}
(h(x_1,0),\partial_t h(x_1,0)\big) = \left(\hinit(x_1),\htinit(x_1)\right)\text{ for all $x_1\in \mathbb{S}^1$}.
\end{equation}
or, equivalently,
\begin{equation}\label{initialdata2}
(\widetilde{h}(x_1,0),\partial_t \widetilde{h}(x_1,0)\big) = \left(\frac{\hinit(x_1)}{\epsilon}, \frac{\htinit(x_1)}{\epsilon}\right)\text{ for all $x_1\in \mathbb{S}^1$}.
\end{equation}

\subsection{Linear recursion for the stream function}
Using \eqref{A_matrix}, the scalar Neumann problem  \eqref{theta_eq_1phase} can be written as
\begin{subequations}\label{theta_eq_1phase_ALE3}
\begin{alignat}{2}
\partial_k(A^k_\ell A^j_\ell \partial_j \varphi) &= 0 &&\text{in}\quad \D\,,\\
A^k_\ell \p_k\varphi \bfn_\ell &= -\omega \circ \psi \qquad&&\text{on}\quad \mathbb{S}^1\,.
\end{alignat}
\end{subequations}
and in  expanded form,
\begin{subequations}\label{theta_eq_1phase_ALE}
\begin{alignat}{2}
\Delta \varphi &= 2 \bp h \p_{\!12} \varphi + \partial_1^2  h \p_{\!2}\varphi - (\bp h)^2 \p_{\!2}^2 \varphi \qquad\qquad\quad&&\text{in}\quad \D \,,\\
\dfrac{\p \varphi}{\p \bN} &= -\dfrac{\widebar{\omega}}{1+(\bp h)^2}  +\frac{\p_{\!1} h}{1+(\bp h)^2}\p_{\!1}\varphi \qquad&&\text{on}\quad \mathbb{S}^1 \,.
\end{alignat}
\end{subequations}

Substitution of our Stokes expansion \eqref{ansatz} shows that
\eqref{theta_eq_1phase_ALE} is equivalent to the following linear recursion relation for $k \ge 0$:

\begin{subequations}
\label{phi_recursion}
\begin{alignat}{2}
\Delta \varphi_k &= \p_{\!2} \left[\sum_{j=0}^{k-1}  \left( 2\p_{\!1} h_j \p_{\!1} \varphi_{k-1-j} + \p_{\!1}^2 h_j  \varphi_{k-1-j} \right) - \sum_{j=0}^{k-2} \sum_{r=0}^j \p_{\!1} h_r \p_{\!1} h_{j-r} \p_{\!2} \varphi_{k-2-j} \right] &&\ \text{ in }\ \D\,,\\
\dfrac{\p \varphi_k}{\p \bN} &= -\omega_k
+ \sum_{j = 0}^{k-1} \p_{\!1} h_j \p_{\!1} \varphi_{k-1-j} - \sum_{j=0}^{k-2} \sum_{r=0}^j \p_{\!1} h_r \p_{\!1} h_{j-r} \p_{\!2} \varphi_{k-2-j}  && \ \text{ on }\ \mathbb{S}^1\,.
\end{alignat}
\end{subequations}

\subsection{Linear recursion for the height function}
Substitution of \eqref{ansatz} into \eqref{h_eq_1phase} shows that
\begin{equation}\label{hk_ww}
\partial_{t} h_k = \p_{\!1} \varphi_k \qquad\text{on}\quad \mathbb{S}^1\,.
\end{equation}

\subsection{Linear recursion for  the  tangential velocity}
In absence of surface tension effects ($\lambda=0$), we have that \eqref{omega_evolution_eq_1phase_h} is equivalent to
\begin{align}
\partial_t \widebar{\omega} &= - g \p_{\!1} h + \frac{1}{2} \p_{\!1} \big[(1+(\p_{\!1} h)^2) ^{-1} \big(|\partial_t h|^2  - |\bar \omega|^2 + 2 \widebar{\omega} \partial_t h \p_{\!1} h\big) \big]\nonumber\\
&= - g \p_{\!1} h + (1+(\p_{\!1} h)^2)\frac{\p_{\!1} \big[ |\partial_t h|^2  - |\bar \omega|^2 + 2 \widebar{\omega} \partial_t h \p_{\!1} h \big]}{2}\nonumber\\
&\quad-\p_{\!1} h\p_{\!1}^2h\big[ |\partial_t h|^2  - |\bar \omega|^2 + 2 \widebar{\omega} \partial_t h \p_{\!1} h \big]  -\partial_t \widebar{\omega}(2(\partial_1 h)^2+(\partial_1 h)^4)\nonumber\\
&\quad -g\partial_1 h(2(\partial_1 h)^2+(\partial_1 h)^4)\label{omega_evolution_eq_1phase_h2}\,.
\end{align}

Substitution of the asymptotic expansion (\ref{expansions}) into \eqref{omega_evolution_eq_1phase_h} shows that
\begin{align}
\partial_t \omega_k &= - g \p_{\!1} h_k + \sum_{\ell=0}^{k-1}\frac{\p_{\!1} \left[\partial_t h_{k-1-\ell}\partial_t h_{\ell}  - \omega_\ell\omega_{k-1-\ell}\right]}{2}+ \sum_{\ell=0}^{k-2}\sum_{n=0}^{\ell} \p_{\!1} \left[ \omega_n \partial_t h_{\ell-n} \p_{\!1} h_{k-2-\ell} \right]\nonumber\\
&\quad +\sum_{\ell=0}^{k-3}\sum_{n=0}^{\ell}\sum_{j=0}^n\p_{\!1} h_j\p_{\!1} h_{n-j}\frac{\p_{\!1} \big[ \partial_t h_{\ell-n}\partial_t h_{k-3-\ell}  - \omega_{\ell-n}\omega_{k-3-\ell} \big]}{2}\nonumber\\
&\quad +\sum_{\ell=0}^{k-4}\sum_{n=0}^{\ell}\sum_{j=0}^n\sum_{m=0}^{j}\p_{\!1} h_m\p_{\!1} h_{j-m}\p_{\!1} \big[ \omega_{n-j} \partial_t h_{\ell-n} \p_{\!1} h_{k-4-\ell} \big]\nonumber\\
&\quad -\sum_{\ell=0}^{k-3}\sum_{n=0}^{\ell}\sum_{j=0}^n\p_{\!1} h_j\p_{\!1}^2 h_{n-j}\left[\partial_t h_{\ell-n}\partial_t h_{k-\ell} -\omega_{\ell-n}\omega_{k-3-\ell} \right]\nonumber\\
&\quad -2\sum_{\ell=0}^{k-4}\sum_{n=0}^{\ell}\sum_{j=0}^n\sum_{m=0}^j\p_{\!1} h_m\p_{\!1}^2 h_{j-m}\omega_{n-j} \partial_t h_{\ell-n} \p_{\!1} h_{k-4-\ell}\nonumber\\
&\quad -2\sum_{\ell=0}^{k-2}\sum_{n=0}^{\ell}\partial_t\omega_n\partial_1 h_{\ell-n}\partial_1 h_{k-2-\ell} -2g\sum_{\ell=0}^{k-2}\sum_{n=0}^{\ell}\partial_1h_n\partial_1 h_{\ell-n}\partial_1 h_{k-2-\ell}\nonumber\\
&\quad -\sum_{\ell=0}^{k-4}\sum_{n=0}^{\ell}\sum_{j=0}^n\sum_{m=0}^j\partial_t\omega_m\partial_1 h_{j-m}\partial_1 h_{n-j}\partial_1 h_{\ell-n}\partial_1 h_{k-4-\ell}\nonumber\\
&\quad -g\sum_{\ell=0}^{k-4}\sum_{n=0}^{\ell}\sum_{j=0}^n\sum_{m=0}^j\partial_1h_m\partial_1 h_{j-m}\partial_1 h_{n-j}\partial_1 h_{\ell-n}\partial_1 h_{k-4-\ell}\,.\label{wk_ww}
\end{align}

\section{Derivation of the quadratic and cubic $h$-models}\label{sec:hmodels}

\subsection{Preliminary lemmas}
The linear recursion for the stream function $\varphi_k$ given in (\ref{phi_recursion}) can be decomposed into simpler elliptic equations.  Thus,
given certain forcing functions $h, \varphi$ and $g$, we shall focus on the following two elliptic equations
\begin{equation}\label{elliptic1}
\Delta X = \p_{\!2}\big[2 (\p_{\!1} h) (\p_{\!1} \varphi) + (\p_{\!1}^2 h) \varphi \big] \ \text{ in } \ \D \,, \ \text{ and } \
\p_{\!2} X = (\p_{\!1} h) (\p_{\!1} \varphi) \ \text{on } \ \mathbb{S}^1\,,
\end{equation}
and
\begin{equation}\label{elliptic2}
\Delta Y = \p_{\!2} g \ \text{ in }\ D\,, \ \text{ and } \
\p_{\!2} Y = g \ \text{ on }\ \mathbb{S}^1\,.
\end{equation}
We shall make use of two lemmas that show that the restriction of the solutions of \eqref{elliptic1} and  \eqref{elliptic1} to $\mathbb{S}^1$ can be expressed
in terms of  the functions on $\mathbb{S}^1$: $h_0,\cdots,h_{k-1}$ and $\omega_0,\cdots,\omega_{k-1}$.

Following our discussion in Section \ref{sec::Fourier}, a harmonic function $f(x_1,x_2)$  in $\D$ can be expanded as
$$
f(x_1,x_2) = \sum_{n = - \infty }^ \infty  \ft{f}(n) e^{inx_1 + |n|x_2}\quad\text{for all $(x_1,x_2) \in \D$}
$$
for some complex coefficients $\ft{f}(n)$ that do not depend on $x_2$.   For example, the stream function
$\varphi_0$, solving (\ref{phi_recursion}) with $k=0$, is harmonic and hence
$\varphi_0(x_1,x_2) = \sum\limits_{n\in \bbZ} \ft{\varphi_0}(n) e^{inx_1 + |n|x_2}$.
For $k=1$, the right-hand side of (\ref{phi_recursion}a)   is given by
\begin{align*}
& \p_{\!2}\big[2 (\p_{\!1} h_0) (\p_{\!1} \varphi_0) + (\p_{\!1}^2 h_0) \varphi_0 \big](x_1,x_2) 
= \sum_{n=- \infty }^\infty  \sum_{m=- \infty }^\infty |m| (m^2 - n^2) \ft{h_0}(n-m) \ft{\varphi_0}(m) e^{inx_1 + |m|x_2}\,,
\end{align*}
and the right-hand side of (\ref{phi_recursion}b) is
\begin{align*}
& (\p_{\!1} h_0) (\p_{\!1} \varphi_0) 
= -  \sum_{n=- \infty }^\infty  \sum_{m=- \infty }^\infty m (n-m) \ft{h_0}(n-m) \ft{\varphi_0}(m) e^{in x_1 + |m|x_2}\,.
\end{align*}
It follows that the solution $\varphi_1$ can be written via the expansion
$$
\varphi_1(x_1,x_2) = \sum_{n,m\in\bbZ} \ft{P_1}_{n,m} e^{in x_1 + |m|x_2} \,,
$$
where $ \sum_{n,m\in\bbZ}$ denotes the double sum  $\sum_{n=- \infty }^\infty  \sum_{m=- \infty }^\infty$ and
 $\{\ft{P_1}_{n,m}\}_{n,m\in\bbZ}$ is a (double) sequence of complex numbers.   Using the recursion formula \eqref{phi_recursion},
an induction argument then shows that for all $j\in \bbN$, the stream function $\varphi_j$ can be written as the expansion
$$
\varphi_j(x_1,x_2) = \sum_{n,m\in\bbZ} \ft{P_j}_{n,m}(x_2) e^{in x_1 + |m|x_2}\,,
$$
where for each fixed $j,n,m$,  $\ft{P_j}_{n,m}(x_2)$ is a polynomial  (of degree $j-1$) function of $x_2$.
This motivates the following two lemmas.

\begin{lemma}\label{lem:p1X}
Let $h:\mathbb{S}^1\to\bbR$ and $\varphi:\D\to \bbR$ denote $2\pi$-periodic functions of $x_1$, such that
$$
h(x_1) = \sum_{k\in\bbZ, k\ne 0} \ft{h}_k e^{ikx_1}\,,\quad \varphi(x_1,x_2) =
\sum_{k,m\in\bbZ} \ft{P}_{k,m}(x_2) e^{ikx_1 + |m|x_2} \,,
$$
where $x_2 \mapsto \ft{P}_{k,m}(x_2)$ is a polynomial function. If $X$ is the unique solution to \eqref{elliptic1}, then
\begin{equation}\label{p1X}
(\p_{\!1} X)(x_1,0) =  - H \big[(\p_{\!1} h)(\p_{\!1} \varphi)\big] - \smallexp{$\sum_{k,\ell,m\in \bbZ}$} i\text{\rm sgn}(k) |m| (\ell^2 - k^2) \ft{h}_{k-\ell} \smallexp{$\sum_{j=0}^\infty \ddfrac{(-1)^j \ft{P}^{(j)}_{\ell,m}(0)}{(|m|+|k|)^{j+1}}$} e^{ikx_1},
\end{equation}
where $\ft{P}^{(j)}_{\ell,m}(0)$ denotes $\p_2^j\ft{P}_{\ell,m}(x_2)$  evaluated at $x_2=0$.
Moreover,
if $\varphi$ is harmonic in $\D$ so that $\varphi(x_1,x_2) = \sum\limits_{k\in\bbZ} \ft{\varphi}_k e^{ikx_1 + |k|x_2}$, then
    \begin{equation}\label{p1X2}
    \p_{\!1} X = - \Lambda [h \p_{\!1}\varphi] + \p_{\!1} (h\Lambda \varphi) = \p_{\!1} \big(\comm{h}{H} \p_{\!1}\varphi\big)\quad\text{on }\ \mathbb{S}^1\,,
    \end{equation}
where $\comm{\cdot }{\cdot }$ denotes the commutator.
\end{lemma}
\begin{proof}
With $X(x_1,x_2) = \sum\limits_{k\in\bbZ} \ft{X}_k(x_2) e^{ikx_1}$,
 $\ft{X}_k(x_2)$ satisfies the differential equation
\begin{align*}
\p_{\!2}^2 \ft{X}_k(x_2) - k^2 \ft{X}_k(x_2) &=
\sum_{\ell,m\in \bbZ} |m| (\ell^2 - k^2) \ft{h}_{k-\ell} \ft{P}_{\ell,m} e^{|m|x_2} \ \   \text{ for } x_2 < 0\,,\\
(\p_{\!2} \ft{X}_k)(0) &= 
\sum_{\ell,m\in\bbZ} (\ell-k)\ell \ft{h}_{k-\ell} \ft{P}_{\ell,m} \,,
\end{align*}
whose solution is given by the variation-of-parameters formula: for $k\ne 0$,
\begin{align*}
\ft{X}_k(x_2) &= \frac{1}{|k|} \sum_{\ell,m\in\bbZ} (\ell-k)\ell \ft{h}_{k-\ell} \ft{P}_{\ell,m}(0) e^{|k|x_2} \\
&\quad - \sum_{\ell,m\in \bbZ} \dfrac{|m| (\ell^2 - k^2)}{2|k|} \ft{h}_{k-\ell} \big(e^{|k| x_2} + e^{-|k| x_2}\big) \sum_{j=0}^\infty \dfrac{(-1)^j \ft{P}^{(j)}_{\ell,m}(0)}{(|m|+|k|)^{j+1}} \nonumber\\
&\quad + \sum_{\ell,m\in \bbZ} \dfrac{|m| (\ell^2 - k^2)}{2k} \ft{h}_{k-\ell} \int_0^{x_2} \ft{P}_{\ell,m}(y_2) \big[e^{(|m|-k)y_2 + kx_2} - e^{(|m|+k)y_2-kx_2}\big]dy_2 
\,.
\end{align*}
Therefore,
\begin{equation}\label{ikX}
i k \ft{X}_k(0) = i \text{\rm sgn}(k) \sum_{\ell,m\in\bbZ} \Big[(\ell-k)\ell \ft{h}_{k-\ell} \ft{P}_{\ell,m} - |m| (\ell^2 - k^2) \ft{h}_{k-\ell} \sum_{j=0}^\infty \ddfrac{(-1)^j \ft{P}^{(j)}_{\ell,m}(0)}{(|m|+|k|)^{j+1}}\Big] \,,
\end{equation}
and (\ref{p1X}) follows from the Fourier inversion formula.

In the case that $\varphi$ is harmonic in $\D$ or equivalently, the Fourier coefficients are given as $\ft{P}_{\ell,m} = \ft{\varphi}_\ell$ if $\ell = m$
and $\ft{P}_{\ell,m} = 0$ if $\ell \ne m$, then the identity \eqref{ikX} shows that
\begin{align*}
i k \ft{X}_k(0) &=i \text{\rm sgn}(k) \sum_{\ell\in\bbZ} (\ell-k)\ell \ft{h}_{k-\ell} \sum_{m\in\bbZ} \ft{P}_{\ell,m} - i\text{sgn}(k) \sum_{\ell\in \bbZ}
\dfrac{|m| (\ell^2 - k^2)}{|m|+|k|} \ft{h}_{k-\ell} \sum_{m\in\bbZ} \ft{P}_{\ell m} \\
&= i \text{\rm sgn}(k) \sum_{\ell\in\bbZ} (\ell-k)\ell \ft{h}_{k-\ell} \ft{\varphi}_\ell - i\text{sgn}(k) \sum_{\ell\in \bbZ} \dfrac{|\ell| (\ell^2 - k^2)}{|\ell|+|k|} \ft{h}_{k-\ell} \ft{\varphi}_{\ell} \\
&= i \text{\rm sgn}(k) \sum_{\ell\in\bbZ} \big[i (k-\ell)\ft{h}_{k-\ell}\big] (i \ell \ft{\varphi}_\ell) - i\text{sgn}(k) \sum_{\ell\in\bbZ} (|\ell| - |k|) \ft{h}_{k-\ell} \ft{\Lambda \varphi}_\ell
\end{align*}
and hence on $\mathbb{S}^1$,
\begin{align*}
\p_{\!1} X &= - H \big[(\p_{\!1} h) (\p_{\!1} \varphi)\big] + H (h \Lambda \varphi) + \p_{\!1} (h \Lambda \varphi) \\
&= - H \big[h \p_{\!1}^2 \varphi + (\p_{\!1} h)(\p_{\!1} \varphi)\big] + \p_{\!1} (h \Lambda \varphi) = - \Lambda (h \p_{\!1} \varphi) + \p_{\!1} (h \Lambda \varphi)
\end{align*}
from which (\ref{p1X2}) follows.
\end{proof}

\begin{lemma}\label{lem:p1Y}
Let $g:\D\to \bbR$ be a $2\pi$-periodic function of $x_1$, such that
$$
g(x_1,x_2) = \sum_{k,m \in\bbZ} \ft{g}_{k,m} e^{ikx_1 + |m|x_2}\,,
$$
and let $Y$ denote the unique solution to \eqref{elliptic2}.
Then
\begin{equation}\label{p1Y}
\p_{\!1} Y = - H g - i \text{\rm sgn}(k) \sum_{k,m\in\bbZ} \frac{|m|}{|m|+|k|} \ft{g}_{k,m} e^{ikx_1 } \quad\text{on }\ \mathbb{S}^1\,.
\end{equation}
\end{lemma}
\begin{proof}
Letting $Y(x_1,x_2) = \sum\limits_{k\in\bbZ} \ft{Y}_k(x_2) e^{ikx_1}$, we find that
\begin{align*}
\p_{\!2}^2 \ft{Y}_k(x_2) - k^2 \ft{Y}_k(x_2) &= \sum_{m\in\bbZ} |m| \ft{g}_{k,m} e^{|m|x_2} \ \ \text{ for } x_2 < 0\,,\\
\p_{\!2} \ft{Y}_k(0) &= \sum_{m \in\bbZ} \ft{g}_{k,m}\,.
\end{align*}
The solution for the case $k\ne 0$ is obtained via the variation-of-parameters formula as
\begin{align*}
\ft{Y}_k(x_2) &= \frac{1}{|k|} \sum_{m \in\bbZ} \ft{g}_{k,m} e^{|k|x_2} - \frac{1}{2|k|} \sum_{m\in\bbZ} \dfrac{|m|}{|m|+|k|} \ft{g}_{k,m} (e^{|k|x_2} + e^{-|k|x_2}) \\
&\quad + \sum_{m\in\bbZ} \dfrac{|m|}{2k} \ft{g}_{k,m} \int_0^{x_2} \big[e^{(|m|-k)x_2 + kx_2} - e^{(|m|+k)x_2-kx_2}\big] dx_2\,.
\end{align*}
Therefore,
\begin{align*}
ik \ft{Y}_k(0) &= i \text{\rm sgn}(k) \sum_{m \in\bbZ} \ft{g}_{k,m} - i \text{\rm sgn}(k) \sum_{m\in\bbZ} \dfrac{|m|}{|m|+|k|} \ft{g}_{k,m}
\end{align*}
which is (\ref{p1Y}).
\end{proof}

\subsection{The quadratic $h$-model}
From \eqref{phi_recursion},
\begin{equation}\label{phi0}
\Delta\varphi_0 = 0 \text{ \ in } \D \ \text{ and } \
\dfrac{\p \varphi_0}{\p \bN} = - \omega_0 \text{ \ on } \mathbb{S}^1\,,
\end{equation}
and
\begin{equation}\label{phi1}
\Delta\varphi_1 =  \p_{\!2}  \left(2 \p_{\!1} h_0 \partial_{1} \varphi_0 + \p_{\!1}^2 h_0 \ \varphi_0 \right) \text{ \ in }\D \  \text{ and }  \
\dfrac{\p \varphi_1}{\p \bN}= - \omega_1 +\p_{\!1} h_0 \p_{\!1} \varphi_{0} \text{ \ on } \mathbb{S}^1\,.
\end{equation}
We decompose  $\varphi_1 = \varphi_1^{\text{\tiny(a)}} + \varphi_1^{\text{\tiny(b)}}$, where $\varphi_1^{\text{\tiny(a)}}$ and $\varphi_1^{\text{\tiny(b)}}$ satisfy
\begin{subequations}\label{varphi1i_eq_ww}
\begin{alignat}{2}
\Delta\varphi_1^{\text{\tiny(a)}} &= \rf:= \p_{\!2} \left(2 \p_{\!1} h_0 \partial_{1} \varphi_0 + \p_{\!1}^2 h_0\  \varphi_0\right) \quad&&\text{in }\ \D\,,\\
\dfrac{\p \varphi_1^{\text{\tiny(a)}}}{\p \bN} &= \rg:= \p_{\!1} h_0 \p_{\!1} \varphi_{0} &&\text{on }\ \mathbb{S}^1\,,
\end{alignat}
\end{subequations}
and
\begin{subequations}\label{varphi1b_eq_ww}
\begin{alignat}{2}
\Delta\varphi_1^{\text{\tiny(b)}} &= 0 &&\text{in }\ \D\,,\\
\dfrac{\p \varphi_1^{\text{\tiny(b)}}}{\p \bN} &= - \omega_1 \quad&&\text{on }\ \mathbb{S}^1\,.
\end{alignat}
\end{subequations}
We note that  the solvability condition for \eqref{varphi1i_eq_ww} is satisfied since integration-by-parts shows that
\begin{align*}
\int_{\mathbb{S}^1\times \bbR^-}\p_{\!2} \left(2 \p_{\!1} h_0 \partial_{1} \varphi_0 + \p_{\!1}^2 h_0\  \varphi_0\right) dy & =
\int_{\mathbb{S}^1}     2 \p_{\!1} h_0 \partial_{1} \varphi_0 + \p_{\!1}^2 h_0\  \varphi_0 dx_1 =
\int_{\mathbb{S}^1} \p_{\!1} h_0 \p_{\!1} \varphi_{0} dx_1\,,
\end{align*}
and similarly the solvability condition for \eqref{varphi1b_eq_ww} is also satisfied:
$
\smallint{\mathbb{S}^1}{}\, \omega_1 dx_1 = 0$.
With the solvability conditions satisfied, the elliptic problems  (\ref{varphi1i_eq_ww}) and (\ref{varphi1b_eq_ww}) have  unique solutions in $H^1(\D)$ by the Lax-Milgram theorem.
Using the Hilbert transform,
\begin{equation} \label{p1_varphi01_ww}
\p_{\!1} \varphi_0 =  H \omega_0  \ \text{ and } \
\p_{\!1} \varphi_1^{\text{\tiny(b)}} = H \omega_1 \text{ \ on } \mathbb{S}^1\,.
\end{equation}
We can then apply Lemma \ref{lem:p1X} and conclude that
\begin{equation}\label{p1phi1i_ww}
\p_{\!1} \varphi_1^{\text{\rm\tiny(a)}}(x_1,0,t) = \p_{\!1} \big(\comm{h_0}{H} H \omega_0\big) = - H\big[(\p_{\!1}h_0)(H \omega_0) + h_0 \Lambda \omega_0 \big] - \p_{\!1} (h_0 \omega_0) \quad\text{on }\ \mathbb{S}^1\,.
\end{equation}

From the recursion for the tangential velocity and \eqref{hk_ww}, we have that
\begin{subequations}
\label{w0w1}
\begin{alignat}{2}
\partial_t h_{0} &= H \omega_0 &&\text{ on}\quad\mathbb{S}^1\,,\\
\partial_t\omega_{0} &= -g \p_{\!1} h_0 &&\text{ on}\quad\mathbb{S}^1\,,\\
\partial_t h_{1} &= H \omega_1- H\big[(\p_{\!1}h_0)(H \omega_0) + h_0 \Lambda \omega_0 \big] - \p_{\!1} (h_0 \omega_0) &&\text{ on}\quad\mathbb{S}^1\,,\\
\partial_t\omega_{1} &= -g \p_{\!1} h_1 + \frac{1}{2} \p_{\!1} \big(|\p_t h_0|^2 - |\omega_0|^2 \big) \qquad&&\text{ on}\quad\mathbb{S}^1\,.
\end{alignat}
\end{subequations}
We can write \eqref{w0w1} as system of wave equations,
\begin{subequations}
\label{h0h1}
\begin{align}
\p_{\!t}^2 h_0 + g \Lambda h_0 &=0 \\
\p_{\!t}^2 h_1 + g \Lambda h_1 &= \frac{1}{2} \Lambda  \big(|\p_{\!t} h_0|^2 - |H \p_{\!t} h_0|^2\big) \nonumber\\
&\quad - H\big[(\p_{\!1} \p_{\!t} h_0)(\p_{\!t} h_0) - g (\p_{\!1} h_0)(\Lambda h_0) + (\p_{\!t} h_0) (\p_{\!1} \p_{\!t} h_0) - g h_0 \Lambda \p_{\!1} h_0 \big] \nonumber\\
&\quad + \p_{\!1} \big[(\p_{\!t} h_0)(H \p_{\!t} h_0)\big] + g \p_{\!1} (h_0 \p_{\!1} h_0) \nonumber\\
&=  - \Lambda (|H \p_{\!t}h_0|^2) + g \Lambda (h_0 \Lambda h_0) + g \p_{\!1} (h_0 \p_{\!1} h_0) \,,
\end{align}
\end{subequations}
where, in the last equality, we have used the Tricomi identity
\begin{equation}\label{tricomi}
2H( f Hf) =  (Hf)^2 -   f^2 \,.
\end{equation}
The quadratic $h$-model follows from setting
\begin{equation}\label{hdefinition}
h(x_1,t)=  \epsilon h_0(x_1,t) + \epsilon ^2 h_1(x_1,t) \,,
\end{equation}
so that
\begin{align*}
\p_{\!t}^2 h + g \Lambda h &=    \epsilon ^2 \left(  - \Lambda (|H \p_{\!t}h_0|^2) + g \Lambda (h_0 \Lambda h_0) + g \p_{\!1} (h_0 \p_{\!1} h_0)\right) \\
&=  - \Lambda (|H \p_{\!t}h|^2) + g \Lambda (h \Lambda h) + g \p_{\!1} (h \p_{\!1} h)  + \O(\epsilon ^3)\,.
\end{align*}
Neglecting terms of order $O(\epsilon^3)$, the quadratic $h-$model reads
\begin{subequations}
\label{hmodel-quad}
\begin{align}
\p_{\!t}^2 h + g \Lambda h &=  - \Lambda (|H \p_{\!t}h|^2) - g\p_{\!1} (\comm{h}{H} \Lambda h) \,,\\
&=  2\p_{\!1}( \p_{\!t} h H \p_{\!t} h) - \Lambda \left( (\p_{\!t} h)^2\right) + g \Lambda (h \Lambda h) + g \p_{\!1} (h \p_{\!1} h)\,.
\end{align}
\end{subequations}

The quadratic $h$-model \eqref{hmodel-quad} modeling gravity water waves in deep water reduces to the ``Model'' equation obtained by Akers \& Milewski \cite{AkMi2010}, although in a very different way.   They first simplify the water waves problem
by making the assumption that the potential function at each recursion relation is set on the fixed domain with top boundary  given by
$x_2=0$ (rather than $x_2= h(x_1,t)$).  It is interesting to note that up to quadratic nonlinearity, this simplification produces the same $h$-model as we have obtained
by keeping the full water waves system in the asymptotics.  We note that
Akers \& Nicholls \cite{AkNi2010} later used a diffeomorphism (similar to our $\phi$) to fix the domain,  but only study the linear recursion for the traveling solitary wave
ansatz.

\begin{remark}
We observe that the quadratic $h$-model \eqref{hmodel-quad} is kept invariant by the scaling
$$
h_{\mu}(x_1,t)=\frac{1}{\mu^2}h(\mu^2 x_1,\mu t).
$$
This is the same scale invariance as for the full gravity water wave problem.
\end{remark}

\begin{remark}\label{remark}
Following a similar approach, for the case of an internal wave separating two perfect fluids with densities $\rho^+$ and $\rho^-$, we can derive the equation
\begin{equation}\label{waveeq}
\partial_t^2 h=Ag\Lambda h+A\Lambda(|H \partial_t h|^2)+A^2g \left(\Lambda(h\Lambda h)+ \partial_1 (h\p_{\!1} h)\right)
\end{equation}
where
$$
A= \frac{\rho^+ - \rho^-}{\rho^+ + \rho^-}
$$
is the Atwood number.
\end{remark}
A similar asymptotic model was derived in Granero-Belinch\'{o}n \& Shkoller \cite{GrSh2017} to study the two-fluid problem.

\subsection{The quadratic $h$-model with surface tension}\label{sectionST}For surface waves in the regime where the effects of both gravity and surface tension are similar in magnitude (wavelengths of order $L\approx\sqrt{73.5/981}cm$ or, equivalently, the Bond number $\frac{\lambda}{gL^2}\approx 1$), the previous recursion for the tangential velocity has to be changed. This is somewhat challenging for  general $k$ due to the denominator present in the expression for the  mean curvature;  however, for $k=0$ and $1$ this modification takes the following form:
\begin{subequations}
\label{w0w1surface}
\begin{alignat}{2}
\partial_t h_{0} &= H \omega_0 &&\text{ on}\quad\mathbb{S}^1\,,\\
\partial_t\omega_{0} &= -g \p_{\!1} h_0+\lambda \partial_1^3 h_0&&\text{ on}\quad\mathbb{S}^1\,,\\
\partial_t h_{1} &= H \omega_1- H\big[(\p_{\!1}h_0)(H \omega_0) + h_0 \Lambda \omega_0 \big] - \p_{\!1} (h_0 \omega_0) &&\text{ on}\quad\mathbb{S}^1\,,\\
\partial_t\omega_{1} &= -g \p_{\!1} h_1+\lambda \partial_1^3 h_1 + \frac{1}{2} \p_{\!1} \big(|\p_t h_0|^2 - |\omega_0|^2 \big) \quad&&\text{ on}\quad\mathbb{S}^1\,.
\end{alignat}
\end{subequations}
Equivalently, taking a time derivative, we have that
$$
\partial_t^2 h_{0} = -g \Lambda h_0-\lambda \Lambda^3 h_0
$$
\begin{align*}
\partial_t^2 h_{1} &= -g \Lambda h_1-\lambda \Lambda^3 h_1 + \frac{1}{2} \Lambda \big(|\p_t h_0|^2 - |H\p_t h_0|^2 \big)\\
&\quad- H\big[(\p_{\!1}\partial_t h_0)(\partial_t h_0)+(\p_{\!1}h_0)(H \partial_t \omega_0) + \partial_t h_0 \partial_1\partial_t h_0+h_0 \Lambda \partial_t \omega_0 \big] \\
&\quad - \p_{\!1} (\partial_th_0 \omega_0+h_0 \partial_t\omega_0)\\
&= -g \Lambda h_1-\lambda \Lambda^3 h_1 - \Lambda \big(|H\p_t h_0|^2 \big)- \Lambda\big[h_0 H \partial_t \omega_0 \big] - \p_{\!1} (h_0 \partial_t\omega_0)\\
&= -g \Lambda h_1-\lambda \Lambda^3 h_1 - \Lambda \big(|H\p_t h_0|^2 \big)+ \Lambda\big[h_0 \left(g \Lambda h_0+\lambda \Lambda^3 h_0\right) \big] - \p_{\!1} \left(h_0 \left(-g \p_{\!1} h_0+\lambda \partial_1^3 h_0\right)\right)\,.
\end{align*}

Then, a similar argument as before shows that, up to $\mathcal{O}(\epsilon^3)$, the quadratic $h-$model  \eqref{hmodel-quad} with surface tension modeling gravity-capillary waves in deep water is written as
\begin{equation}\label{hmodelsurface}
\partial_t^2 h=-g \Lambda h-\lambda \Lambda^3 h - \Lambda \big(|H\p_t h|^2 \big)+ \Lambda\big[h \left(g \Lambda h+\lambda \Lambda^3 h\right) \big] - \p_{\!1} \left(h \left(-g \p_{\!1} h+\lambda \partial_1^3 h\right)\right)\,.
\end{equation}

\subsection{The cubic $h$-model}  In order to derive the cubic h-model, we shall also need the equation that $\varphi_2$ satisfies; thus, in addition to \eqref{phi0} and \eqref{phi1}, we use
\eqref{phi_recursion} to find that
\begin{subequations}\label{phi2}
\begin{alignat}{2}
\Delta\varphi_2 &= \p_{\!2}\big[ 2 \p_{\!1} h_0 \partial_{12} \varphi_1 + 2 \p_{\!1} h_1 \partial_{1} \varphi_0 + \varphi_1 \p_{\!1}^2 h_0 + \varphi_0 \p_{\!1}^2 h_1 - (\p_{\!1} h_0)^2 \p_{\!2} \varphi_0  \big]\quad&&\text{in }\ \D \,,\\
\dfrac{\p \varphi_2}{\p \bN} &= - \omega_2 +\p_{\!1} h_0 \p_{\!1} \varphi_1 + \p_{\!1} h_1 \p_{\!1} \varphi_{0} - (\p_{\!1} h_0)^2 \p_{\!2} \varphi_0 &&\text{on }\ \mathbb{S}^1\,.
\end{alignat}
\end{subequations}
We decompose $\varphi_2$ as the sum
$\varphi_2 = \varphi_2^{\text{\tiny(a)}} + \varphi_2^{\text{\tiny(b)}} + \varphi_2^{\text{\tiny(c)}} + \varphi_2^{\text{\tiny(d)}}$, where
$\varphi_2^{\text{\tiny(a)}}$, $\varphi_2^{\text{\tiny(b)}}$, $\varphi_2^{\text{\tiny(c)}}$ and $\varphi_2^{\text{\tiny(d)}}$ satisfy
\begin{alignat*}{2}
\Delta\varphi_2^{\text{\tiny(a)}} &= 2 \p_{\!1} h_1 \partial_{12} \varphi_0 + \p_{\!1}^2 h_1 \p_{\!2} \varphi_0 \ \text{ in }\ \D
&& \qquad   \dfrac{\p \varphi_2^{\text{\tiny(a)}}}{\p \bN} = \p_{\!1} h_1 \p_{\!1} \varphi_{0}\ \text{ on } \  \mathbb{S}^1\,, \\
 \qquad  \Delta\varphi_2^{\text{\tiny(b)}} &= 2 \p_{\!1} h_0 \partial_{12} \varphi_1 + \p_{\!1}^2 h_0 \p_{\!2} \varphi_1 \ \text{ in } \ \D
&& \qquad\dfrac{\p \varphi_2^{\text{\tiny(b)}}}{\p \bN} = \p_{\!1} h_0 \p_{\!1} \varphi_1 \ \text{ on } \ \mathbb{S}^1\,,  \\
\Delta\varphi_2^{\text{\tiny(c)}} &= - (\p_{\!1} h_0)^2 \p_{\!2}^2 \varphi_0 \ \text{ in } \ \D
&&\qquad \dfrac{\p \varphi_2^{\text{\tiny(c)}}}{\p \bN} = - (\p_{\!1} h_0)^2 \p_{\!2} \varphi_0 \ \text{ on }  \ \mathbb{S}^1\,,\\
\Delta\varphi_2^{\text{\tiny(d)}} &= 0 \ \text{ in } \ \D
&& \qquad  \dfrac{\p \varphi_2^{\text{\tiny(d)}}}{\p \bN} = - \omega_2 \ \text{ on } \ \mathbb{S}^1\,.
\end{alignat*}
Note that $\p_{\!1} \varphi_2^{\text{\tiny(d)}} = H\omega_2$ on $\mathbb{S}^1$.\vspace{.1cm}

Solving  \eqref{phi0},  $\varphi_0(x_1,x_2,t) = -\sum\limits_{k\in\bbZ,k\ne 0} \dfrac{\ft{\omega_0}_k(t)}{|k|} e^{ikx_1 + |k| x_2}$, and by Lemma
\ref{lem:p1X},
\begin{equation}\label{p1phi2a}
\p_{\!1} \varphi_2^{\text{\tiny(a)}}
= \p_{\!1} (\comm{h_1}{H} H \omega_0)\,.
\end{equation}
Next, we write the solution to  \eqref{varphi1b_eq_ww} using the Fourier components $\ft{\varphi_1^{\text{\tiny(b)}}}(k,x_2,t) = -\dfrac{\ft{\omega_1}_k(t)}{|k|} e^{|k| x_2}$, and
using 
the variation-of-parameters solution to \eqref{varphi1i_eq_ww},
we see that
\begin{align*}
\ft{\varphi_1^{\text{\tiny(a)}}}(k,x_2,t) &= \dfrac{\text{\rm sgn}(k)}{k} \ft{\rg}(k,t) e^{|k|x_2} + \int_0^{x_2} \ft{\rf}(k,y_2,t) \dfrac{e^{k(x_2-y_2)} - e^{k(y_2-x_2)}}{2k} dy_2 \\
&= \frac{\ft{\smallexp{$(\p_{\!1}(h_0 H \omega_0))$}}_k(t)}{|k|} e^{|k|x_2} - \sum\limits_{\ell\in \bbZ} \ft{h_0}_{k-\ell}(t) \ft{\omega_0}_\ell(t) e^{|\ell|x_2} \,;
\end{align*}
thus,
\begin{align*}
\varphi_1(x_1,x_2,t) &= \sum_{k\in\bbZ} \dfrac{- \ft{\omega_1}_k(t) + \ft{\smallexp{$(\p_{\!1}(h_0 H \omega_0))$}}_k(t)}{|k|} e^{ikx_1 + |k|x_2} - \sum_{k,\ell\in \bbZ} \ft{h_0}_{k-\ell}(t) \ft{\omega_0}_\ell(t) e^{ikx_1 + |\ell|x_2}\,.
\end{align*}
It then follows from Lemma \ref{lem:p1X} with $\ft{P}_{k,m} = \delta_{km} \ddfrac{- \ft{\omega_1}_k(t) + \ft{(\p_{\!1}(h_0 H \omega_0))}_k(t)}{|k|} - \ft{h_0}_{k-m}(t) \ft{\omega_0}_m(t)$
 that
\begin{align}
\p_{\!1} \varphi_2^{\text{\tiny(b)}} &= \p_{\!1} \big(\comm{h_0}{H} H\omega_1\big) - \p_{\!1} \big(\comm{h_0}{H} \Lambda (h_0 H \omega_0)\big) + H \big[(\p_{\!1} h_0) \p_{\!1}(h_0 \omega_0)\big] \nonumber\\
&\quad + \sum_{k,\ell,m\in\bbZ} i \text{sgn}(k) \dfrac{|m| (\ell^2 - k^2)}{|m|+|k|} \ft{h_0}_{k-\ell} \ft{h_0}_{\ell-m} \ft{\omega_0}_m e^{ikx_1} \nonumber\\
&= \p_{\!1} \big(\comm{h_0}{H} H\omega_1\big) - \p_{\!1} \big(\comm{h_0}{H} \Lambda (h_0 H \omega_0)\big) + H \big[(\p_{\!1} h_0) \p_{\!1}(h_0 \omega_0)\big] \nonumber\\
&\quad - \sum_{k,\ell,m\in\bbZ} \text{sgn}(k) \dfrac{|m| (\ell+k)}{|m|+|k|} \ft{(\p_{\!1} h_0)}_{k-\ell} \ft{h_0}_{\ell-m} \ft{\omega_0}_m e^{ikx_1} \,. \label{p1phi2b}
\end{align}
Using Lemma \ref{lem:p1Y} with $g = - (\p_{\!1} h_0)^2(\p_{\!1}\varphi_0)$ (or equivalently with  $\ft{g}_{k,m} = \ft{(\p_{\!1}h_0)^2}_{k-m}\, \ft{\omega_0}_m$), we find that
\begin{align}
& (\p_{\!1} \varphi_2^{\text{\tiny(c)}})(x_1,0,t) \nonumber\\
&\quad = - H \big[(\p_{\!1}h_0)^2 \omega_0\big](t) - \sum_{k,\ell\in\bbZ, k,\ell\ne 0} i \text{\rm sgn}(k) \dfrac{|m|}{|m|+|k|} \ft{(\p_{\!1}h_0)^2}_{k-m}(t) \ft{\omega_0}_m(t) e^{ikx_1} \nonumber\\
&\quad = - H \big[(\p_{\!1}h_0)^2 \omega_0\big](t) - \sum_{k,\ell,m\in\bbZ} i \text{\rm sgn}(k) \dfrac{|m|}{|m|+|k|} \ft{(\p_{\!1}h_0)}_{k-\ell}(t) \ft{(\p_{\!1}h_0)}_{\ell-m}(t) \ft{\omega_0}_m(t) e^{ikx_1} \nonumber\\
&\quad = - H \big[(\p_{\!1}h_0)^2 \omega_0\big](t) + \sum_{k,\ell,m\in\bbZ} \text{\rm sgn}(k) \dfrac{|m|(\ell-m)}{|m|+|k|} \ft{(\p_{\!1} h_0)}_{k-\ell}(t) \ft{h_0}_{\ell-m}(t) \ft{\omega_0}_m(t) e^{ikx_1} \,. \label{p1phi2c}
\end{align}
Hence, combining (\ref{p1phi2a}), (\ref{p1phi2b}) and (\ref{p1phi2c}) and the fact that $\p_{\!1} \varphi_2^{\text{\tiny(d)}} = H\omega_2$ on $\mathbb{S}^1$,
\begin{align*}
& (\p_{\!1} \varphi_2)(x_1,0) = H \omega_2 + \p_{\!1} (\comm{h_1}{H} H \omega_0) + \p_{\!1} \big(\comm{h_0}{H} H\omega_1\big) - \p_{\!1} \big(\comm{h_0}{H} \Lambda (h_0 H \omega_0)\big) \\
&\ \ +\hspace{-1pt} H \big[(\p_{\!1} h_0) \p_{\!1}(h_0 \omega_0)\big] \hspace{-1pt}-\hspace{-1pt} H \big[(\p_{\!1}h_0)^2 \omega_0\big] \hspace{-1pt}-\hspace{-1.5pt} \sum_{k,\ell,m\in\bbZ} \!\text{\rm sgn}(k) \dfrac{|m|(m+k)}{|m|+|k|} \ft{(\p_{\!1} h_0)}_{k-\ell} \ft{h_0}_{\ell-m} \ft{\omega_0}_m e^{ikx_1} .
\end{align*}
Noting that for each fixed $k,m\in\bbZ$,
$$
\sum_{\ell\in\bbZ} \ft{\p_{\!1} h_0}_{k-\ell} \ft{h_0}_{\ell-m} = \ft{h_0 \p_{\!1}h_0}_{k-m} = \frac{1}{2} \ft{\p_{\!1} (h_0^2)}_{k-m} = \frac{i(k-m)}{2} \ft{(h_0^2)}_{k-m}\,,
$$
we have that
\begin{align*}
& \sum_{k,\ell,m\in\bbZ} \text{\rm sgn}(k) \dfrac{|m|(m+k)}{|m|+|k|} \ft{(\p_{\!1} h_0)}_{k-\ell} \ft{h_0}_{\ell-m}(t) \ft{\omega_0}_m e^{ikx_1} \\
&\qquad = \sum_{k,m\in\bbZ} i \text{\rm sgn}(k) \dfrac{|m|(m+k)(k-m)}{2(|m|+|k|)} \ft{(h_0^2)}_{k-m}(t) \ft{\omega_0}_m e^{ikx_1} \\
&\qquad = \frac{1}{2} \sum_{k,m\in\bbZ} i \text{\rm sgn}(k) |m|(|k|-|m|)\ft{(h_0^2)}_{k-m}(t) \ft{\omega_0}_m e^{ikx_1} = \dfrac{1}{2} \p_{\!1} (h_0^2 \Lambda \omega_0) + \dfrac{1}{2} H (h_0^2 \Lambda^2 \omega_0) \,.
\end{align*}
Therefore, we have that
\begin{subequations}
\label{dth2}
\begin{alignat}{2}
\p_{\!t} h_2 - H \omega_2 &= \p_{\!1} (\comm{h_1}{H} H \omega_0) + \p_{\!1} \big(\comm{h_0}{H} H\omega_1\big) - \p_{\!1} \big(\comm{h_0}{H} \Lambda (h_0 H \omega_0)\big) \nonumber\\
&\quad + H \big[(\p_{\!1} h_0) \p_{\!1}(h_0 \omega_0)\big] - H \big[(\p_{\!1}h_0)^2 \omega_0\big] - \dfrac{1}{2} \p_{\!1} (h_0^2 \Lambda \omega_0) - \dfrac{1}{2} H (h_0^2 \Lambda^2 \omega_0) \,,\\
\partial_t \omega_2 &= - g \p_{\!1} h_2 + \p_{\!1} \left[\partial_t h_{1}\partial_t h_{0}  - \omega_0\omega_{1}\right]+ \p_{\!1} \left[ \omega_0 \partial_t h_{0} \p_{\!1} h_{0} \right]\,,
\end{alignat}
\end{subequations}
We next time-differentiate (\ref{dth2}a) and substitute (\ref{dth2}b) to find that
\begin{align*}
\p_{\!t}^2 h_2 + g\Lambda h_2 &= \p_{\!1} (\comm{\p_{\!t} h_1}{H} H \omega_0) + \p_{\!1} (\comm{h_1}{H} H \p_{\!t} \omega_0) + \p_{\!1} \big(\comm{\p_{\!t} h_0}{H} H\omega_1\big) \\
&\quad +\p_{\!1} \big(\comm{h_0}{H} H \p_{\!t}\omega_1\big) - \p_{\!1} \big(\comm{\p_{\!t} h_0}{H} \Lambda (h_0 H \omega_0)\big) \\
&\quad - \p_{\!1} \big(\comm{h_0}{H} \Lambda (\p_{\!t} h_0 H \omega_0)\big) - \p_{\!1} \big(\comm{h_0}{H} \Lambda (h_0 H \p_{\!t} \omega_0)\big) \\
&\quad + H \big[(\p_{\!1} \p_{\!t} h_0) \p_{\!1}(h_0 \omega_0)\big] + H \big[(\p_{\!1} h_0) \p_{\!1}(\p_{\!t} h_0 \omega_0)\big]
 + H \big[(\p_{\!1} h_0) \p_{\!1}(h_0 \p_{\!t} \omega_0)\big] \nonumber\\
&\quad - 2 H \big[(\p_{\!1}h_0) (\p_{\!1} \p_{\!t} h_0) \omega_0\big] - H \big[(\p_{\!1}h_0)^2 (\p_{\!t} \omega_0)\big] - \p_{\!1} \big[h_0 (\p_{\!t} h_0) (\Lambda \omega_0) \big] \\
&\quad - \dfrac{1}{2} \p_{\!1} \big[h_0^2 (\Lambda \p_{\!t} \omega_0)\big] - H \big[h_0 (\p_{\!t} h_0)\Lambda^2 \omega_0\big] - \dfrac{1}{2} H (h_0^2 \Lambda^2 \p_{\!t} \omega_0)\\
&\quad +\Lambda \left[\partial_t h_{1}\partial_t h_{0}  - \omega_0\omega_{1}\right]+ \Lambda \left[ \omega_0 \partial_t h_{0} \p_{\!1} h_{0} \right] \,.
\end{align*}
Thanks to the identities in \eqref{w0w1}, we conclude that
$$
\omega_0=-H\p_{\!t} h_0\ \text{ and } \
\omega_1 = -H\p_{\!t} h_1 + \Lambda \big(\comm{h_0}{H}  \p_{\!t} h_0\big)\quad\text{on }\mathbb{S}^1.
$$

Thus,
\begin{align*}
\p_{\!t}^2 h_2 + g\Lambda h_2 &= \p_{\!1} (\comm{\p_{\!t} h_1}{H} \partial_t h_0) - g\p_{\!1} (\comm{h_1}{H} \Lambda h_0) + \p_{\!1} \big(\comm{\p_{\!t} h_0}{H} \left[\partial_t h_1-\p_{\!1} \big(\comm{h_0}{H} \partial_t h_0\big)\right]\big) \\
&\quad +\p_{\!1} \left(\comm{h_0}{H} \left[-g \Lambda h_1 + \frac{1}{2} \Lambda \big(|\p_t h_0|^2 - |H\partial_t h_0|^2 \big)\right]\right) - \p_{\!1} \big(\comm{\p_{\!t} h_0}{H} \Lambda (h_0 \partial_t h_0)\big) \\
&\quad - \p_{\!1} \big(\comm{h_0}{H} \Lambda ((\p_{\!t} h_0)^2)\big) + g\p_{\!1} \big(\comm{h_0}{H} \Lambda (h_0 \Lambda h_0)\big) \\
&\quad - H \big[(\p_{\!1} \p_{\!t} h_0) \p_{\!1}(h_0 H\p_{\!t} h_0)\big] - H \big[(\p_{\!1} h_0) \p_{\!1}(\p_{\!t} h_0 H\p_{\!t} h_0)\big]
 - gH \big[(\p_{\!1} h_0) \p_{\!1}(h_0 \partial_1 h_0)\big] \nonumber\\
&\quad + 2 H \big[(\p_{\!1}h_0) (\p_{\!1} \p_{\!t} h_0) H\p_{\!t} h_0\big] + gH \big[(\p_{\!1}h_0)^3\big] - \p_{\!1} \big[h_0 (\p_{\!t} h_0) (\partial_1 \p_{\!t} h_0) \big] \\
&\quad + \dfrac{g}{2} \p_{\!1} \big[h_0^2 (\Lambda \p_{1} h_0)\big] - H \big[h_0 (\p_{\!t} h_0)\Lambda \partial_1\p_{\!t} h_0\big] + \dfrac{g}{2} H (h_0^2 \Lambda^2 \partial_1h_0)\\
&\quad +\Lambda \left[\partial_t h_{1}\partial_t h_{0}  +H\p_{\!t} h_0\left(-H\p_{\!t} h_1 + \Lambda \big(\comm{h_0}{H}  \p_{\!t} h_0\big)\right)\right]- \Lambda \left[ H\p_{\!t} h_0 \partial_t h_{0} \p_{\!1} h_{0} \right]\,.
\end{align*}
Using Tricomi's identity \eqref{tricomi}, we can reduce the previous expression to
\begin{align*}
\p_{\!t}^2 h_2 + g\Lambda h_2 &= \p_{\!1} (\comm{\p_{\!t} h_1}{H} \partial_t h_0) - g\p_{\!1} (\comm{h_1}{H} \Lambda h_0) + \p_{\!1} \big(\comm{\p_{\!t} h_0}{H} \left[\partial_t h_1-\p_{\!1} \big(\comm{h_0}{H} \partial_t h_0\big)\right]\big) \\
&\quad -\p_{\!1} \left(\comm{h_0}{H} \left[g \Lambda h_1 - \partial_1 \big(\p_t h_0H\partial_t h_0 \big)\right]\right) - \p_{\!1} \big(\comm{\p_{\!t} h_0}{H} \Lambda (h_0 \partial_t h_0)\big) \\
&\quad - \p_{\!1} \big(\comm{h_0}{H} \Lambda ((\p_{\!t} h_0)^2)\big) + g\p_{\!1} \big(\comm{h_0}{H} \Lambda (h_0 \Lambda h_0)\big) \\
&\quad - H \big[(\p_{\!1} \p_{\!t} h_0) \p_{\!1}(h_0 H\p_{\!t} h_0)\big] - H \big[(\p_{\!1} h_0) \p_{\!1}(\p_{\!t} h_0 H\p_{\!t} h_0)\big]
 - gH \big[(\p_{\!1} h_0) \p_{\!1}(h_0 \partial_1 h_0)\big] \nonumber\\
&\quad + 2 H \big[(\p_{\!1}h_0) (\p_{\!1} \p_{\!t} h_0) H\p_{\!t} h_0\big] + gH \big[(\p_{\!1}h_0)^3\big] - \p_{\!1} \big[h_0 (\p_{\!t} h_0) (\partial_1 \p_{\!t} h_0) \big] \\
&\quad + \dfrac{g}{2} \p_{\!1} \big[h_0^2 (\Lambda \p_{1} h_0)\big] - H \big[h_0 (\p_{\!t} h_0)\Lambda \partial_1\p_{\!t} h_0\big] + \dfrac{g}{2} H (h_0^2 \Lambda^2 \partial_1h_0)\\
&\quad +\Lambda \left[\partial_t h_{1}\partial_t h_{0}  +H\p_{\!t} h_0\left(-H\p_{\!t} h_1 + \Lambda \big(\comm{h_0}{H}  \p_{\!t} h_0\big)\right)\right]- \Lambda \left[ H\p_{\!t} h_0 \partial_t h_{0} \p_{\!1} h_{0} \right]\,.
\end{align*}
The cubic $h$-model follows from setting
$$
h(x_1,t)=  \epsilon h_0(x_1,t) + \epsilon ^2 h_1(x_1,t)+\epsilon ^3 h_2(x_1,t) \,,
$$
so that
\begin{align*}
\p_{\!t}^2 h + g\Lambda h &= \epsilon^3\bigg{[}\p_{\!1} (\comm{\p_{\!t} h_1}{H} \partial_t h_0) - g\p_{\!1} (\comm{h_1}{H} \Lambda h_0) + \p_{\!1} \big(\comm{\p_{\!t} h_0}{H} \left[\partial_t h_1-\p_{\!1} \big(\comm{h_0}{H} \partial_t h_0\big)\right]\big) \\
&\quad -\p_{\!1} \left(\comm{h_0}{H} \left[g \Lambda h_1 - \partial_1 \big(\p_t h_0H\partial_t h_0 \big)\right]\right) - \p_{\!1} \big(\comm{\p_{\!t} h_0}{H} \Lambda (h_0 \partial_t h_0)\big) \\
&\quad - \p_{\!1} \big(\comm{h_0}{H} \Lambda ((\p_{\!t} h_0)^2)\big) + g\p_{\!1} \big(\comm{h_0}{H} \Lambda (h_0 \Lambda h_0)\big) \\
&\quad - H \big[(\p_{\!1} \p_{\!t} h_0) \p_{\!1}(h_0 H\p_{\!t} h_0)\big] -  H \big[(\p_{\!1} h_0) \p_{\!1}(\p_{\!t} h_0 H\p_{\!t} h_0)\big]
 - gH \big[(\p_{\!1} h_0) \p_{\!1}(h_0 \partial_1 h_0)\big] \nonumber\\
&\quad + 2 H \big[(\p_{\!1}h_0) (\p_{\!1} \p_{\!t} h_0) H\p_{\!t} h_0\big] + gH \big[(\p_{\!1}h_0)^3\big] - \p_{\!1} \big[h_0 (\p_{\!t} h_0) (\partial_1 \p_{\!t} h_0) \big] \\
&\quad + \dfrac{g}{2} \p_{\!1} \big[h_0^2 (\Lambda \p_{1} h_0)\big] - H \big[h_0 (\p_{\!t} h_0)\Lambda \partial_1\p_{\!t} h_0\big] + \dfrac{g}{2} H (h_0^2 \Lambda^2 \partial_1h_0)\\
&\quad +\Lambda \left[\partial_t h_{1}\partial_t h_{0}  +H\p_{\!t} h_0\left(-H\p_{\!t} h_1 + \Lambda \big(\comm{h_0}{H}  \p_{\!t} h_0\big)\right)\right]- \Lambda \left[ H\p_{\!t} h_0 \partial_t h_{0} \p_{\!1} h_{0} \right]\bigg{]}\\
&\quad + \epsilon^2\bigg{[}-\Lambda (|H \p_{\!t}h_0|^2) + g \Lambda (h_0 \Lambda h_0) + g \p_{\!1} (h_0 \p_{\!1} h_0)\bigg{]}\,.
\end{align*}
Then, we find that
\begin{align*}
\p_{\!t}^2 h + g\Lambda h &=\p_{\!1} (\comm{\epsilon^2\p_{\!t} h_1}{H} \epsilon\partial_t h_0) - g\p_{\!1} (\comm{h}{H} \epsilon\Lambda h_0) + \p_{\!1} \big(\comm{\epsilon\p_{\!t} h_0}{H} \left[\epsilon^2\partial_t h_1-\p_{\!1} \big(\comm{h}{H} \partial_t h\big)\right]\big) \\
&\quad - \p_{\!1} \left(\comm{\epsilon h_0}{H} \left[g \Lambda \epsilon^2 h_1 - \partial_1 \big(\p_t hH\partial_t h \big)\right]\right) - \p_{\!1} \big(\comm{\p_{\!t} h}{H} \Lambda (h \partial_t h)\big) \\
&\quad - \p_{\!1} \big(\comm{h}{H} \Lambda ((\p_{\!t} h)^2)\big) + g\p_{\!1} \big(\comm{h}{H} \Lambda (h \Lambda h)\big) \\
&\quad - H \big[(\p_{\!1} \p_{\!t} h) \p_{\!1}(h H\p_{\!t} h)\big] -  H \big[(\p_{\!1} h) \p_{\!1}(\p_{\!t} h H\p_{\!t} h)\big]
 - gH \big[(\p_{\!1} h) \p_{\!1}(h \partial_1 h)\big] \nonumber\\
&\quad + 2 H \big[(\p_{\!1}h) (\p_{\!1} \p_{\!t} h) H\p_{\!t} h\big] + gH \big[(\p_{\!1}h)^3\big] - \p_{\!1} \big[h (\p_{\!t} h) (\partial_1 \p_{\!t} h) \big] \\
&\quad + \dfrac{g}{2} \p_{\!1} \big[h^2 (\Lambda \p_{1} h)\big] - H \big[h (\p_{\!t} h)\Lambda \partial_1\p_{\!t} h\big] - \dfrac{g}{2} H (h^2 \partial_1^3h)\\
&\quad +\Lambda \left[\epsilon^2\partial_t h_{1}\epsilon \partial_t h_{0}  +H\epsilon\p_{\!t} h_0\left(-H\p_{\!t} h + \Lambda \big(\comm{h}{H}  \p_{\!t} h\big)\right)\right]- \Lambda \left[ H\p_{\!t} h \partial_t h \p_{\!1} h \right]+O(\epsilon^4)\,.
\end{align*}
We observe that
$$
- g\p_{\!1} (\comm{h}{H} \epsilon\Lambda h_0)-g\p_{\!1} \left(\comm{\epsilon h_0}{H} \Lambda \epsilon^2 h_1
\right)=-g\p_{\!1} \left(\comm{h}{H} \Lambda h\right)+O(\epsilon^4)\,;
$$
thus, we can simplify the equation as follows:
\begin{align*}
\p_{\!t}^2 h + g\Lambda h &=\p_{\!1} (\comm{\epsilon^2\p_{\!t} h_1}{H} \epsilon\partial_t h_0) -g\p_{\!1} \left(\comm{h}{H} \Lambda h\right) + \p_{\!1} \big(\comm{\epsilon\p_{\!t} h_0}{H} \left[\epsilon^2\partial_t h_1-\p_{\!1} \big(\comm{h}{H} \partial_t h\big)\right]\big) \\
&\quad +\p_{\!1} \left(\comm{h}{H} \left[\partial_1 \big(\p_t hH\partial_t h \big)\right]\right) - \p_{\!1} \big(\comm{\p_{\!t} h}{H} \Lambda (h \partial_t h)\big) \\
&\quad - \p_{\!1} \big(\comm{h}{H} \Lambda ((\p_{\!t} h)^2)\big) + g\p_{\!1} \big(\comm{h}{H} \Lambda (h \Lambda h)\big) \\
&\quad - H \big[(\p_{\!1} \p_{\!t} h) \p_{\!1}(h H\p_{\!t} h)\big] - H \big[(\p_{\!1} h) \p_{\!1}(\p_{\!t} h H\p_{\!t} h)\big]
 - gH \big[(\p_{\!1} h) \p_{\!1}(h \partial_1 h)\big] \nonumber\\
&\quad + 2 H \big[(\p_{\!1}h) (\p_{\!1} \p_{\!t} h) H\p_{\!t} h\big] + gH \big[(\p_{\!1}h)^3\big] - \p_{\!1} \big[h (\p_{\!t} h) (\partial_1 \p_{\!t} h) \big] \\
&\quad + \dfrac{g}{2} \p_{\!1} \big[h^2 (\Lambda \p_{1} h)\big] - H \big[h (\p_{\!t} h)\Lambda \partial_1\p_{\!t} h\big] - \dfrac{g}{2} H (h^2 \partial_1^3h)\\
&\quad +\Lambda \left[\epsilon^2\partial_t h_{1}\epsilon \partial_t h_{0}  +H\epsilon\p_{\!t} h_0\left(-H\p_{\!t} h + \Lambda \big(\comm{h}{H}  \p_{\!t} h\big)\right)\right]- \Lambda \left[ H\p_{\!t} h \partial_t h \p_{\!1} h \right]+O(\epsilon^4)\,.
\end{align*}
We also compute that
$$
\p_{\!1} (\comm{\epsilon^2\p_{\!t} h_1}{H} \epsilon\partial_t h_0)=\p_{\!1} (\comm{\epsilon^2\p_{\!t}h_1}{H} \partial_t h)  +O(\epsilon^4)\,,
$$
$$
\Lambda \left[\epsilon^2\partial_t h_{1}\epsilon \partial_t h_{0}\right]=\Lambda \left[\epsilon^2\partial_t h_1\partial_t h\right]+O(\epsilon^4)\,,
$$
$$
\p_{\!1} \big(\comm{\epsilon\p_{\!t} h_0}{H} \epsilon^2\partial_t h_1\big)=\p_{\!1} \big(\comm{\p_{\!t} h}{H} \epsilon^2\partial_t h_1\big)+O(\epsilon^4)\,,
$$
so that
\begin{multline}
\p_{\!1} (\comm{\epsilon^2\p_{\!t} h_1}{H} \epsilon\partial_t h_0)+\Lambda \left[\epsilon^2\partial_t h_{1}\epsilon \partial_t h_{0}\right]+\p_{\!1} \big(\comm{\epsilon\p_{\!t} h_0}{H} \epsilon^2\partial_t h_1\big)\\
=\p_{\!1} (\comm{\epsilon^2\p_{\!t}h_1}{H} \partial_t h)+\Lambda \left[\epsilon^2\partial_t h_1\partial_t h\right]+\p_{\!1} \big(\comm{\p_{\!t} h}{H} \epsilon^2\partial_t h_1\big)+O(\epsilon^4)\\
=\p_{\!1} (\epsilon^2\p_{\!t}h_1H\partial_t h)+\p_{\!1} \big(\comm{\p_{\!t} h}{H} \epsilon^2\partial_t h_1\big)+O(\epsilon^4)\\
=\p_{\!1} (\epsilon^2\p_{\!t}h_1H\partial_t h)+\p_{\!1} \big(\p_{\!t} hH\epsilon^2\partial_t h_1\big)-\Lambda \left[\epsilon^2\partial_t h_1\partial_t h\right]+O(\epsilon^4).\label{eqauxb}
\end{multline}
Using Tricomi's identity for two functions, we find that
\begin{equation}\label{eqauxa}
\p_{\!1} (\epsilon^2\p_{\!t}h_1H\partial_t h)+\p_{\!1} \big(\p_{\!t} hH\epsilon^2\partial_t h_1\big)=-\Lambda\left[H\partial_t h\epsilon^2H\partial_t h_1-\partial_t h\epsilon^2\partial_t h_1\right] \,.
\end{equation}
Grouping terms further using \eqref{eqauxb} and \eqref{eqauxa}, together with
$$
-\Lambda\left[H\epsilon\p_{\!t} h_0H\p_{\!t} h\right]=-\Lambda\left[\left(H\p_{\!t} h\right)^2\right]+\Lambda\left[H\epsilon^2\p_{\!t} h_1H\p_{\!t} h\right]+O(\epsilon^4)
$$
we obtain that
\begin{multline*}
\p_{\!1} (\epsilon^2\p_{\!t}h_1H\partial_t h)+\p_{\!1} \big(\p_{\!t} hH\epsilon^2\partial_t h_1\big)-\Lambda \left[\epsilon^2\partial_t h_1\partial_t h\right]-\Lambda\left[H\epsilon\p_{\!t} h_0H\p_{\!t} h\right]\\
=-\Lambda\left[H\partial_t h\epsilon^2H\partial_t h_1-\partial_t h\epsilon^2\partial_t h_1\right]-\Lambda \left[\epsilon^2\partial_t h_1\partial_t h\right]\\
-\Lambda\left[\left(H\p_{\!t} h\right)^2\right]+\Lambda\left[H\epsilon^2\p_{\!t} h_1H\p_{\!t} h\right]+O(\epsilon^4)\,.
\end{multline*}
The cubic $h-$model is then given by
\begin{align*}
\p_{\!t}^2 h + g\Lambda h &=-\Lambda \big[\left(H\p_{\!t} h\right)^2\big] -g\p_{\!1} \left(\comm{h}{H} \Lambda h\right) - \p_{\!1} \big(\comm{\p_{\!t} h}{H} \left[\p_{\!1} \big(\comm{h}{H} \partial_t h\big)\right]\big) \\
&\quad +\p_{\!1} \left(\comm{h}{H} \left[\partial_1 \big(\p_t hH\partial_t h \big)\right]\right) - \p_{\!1} \big(\comm{\p_{\!t} h}{H} \Lambda (h \partial_t h)\big) \\
&\quad - \p_{\!1} \big(\comm{h}{H} \Lambda ((\p_{\!t} h)^2)\big) + g\p_{\!1} \big(\comm{h}{H} \Lambda (h \Lambda h)\big) \\
&\quad - H \big[(\p_{\!1} \p_{\!t} h) \p_{\!1}(h H\p_{\!t} h)\big] - H \big[(\p_{\!1} h) \p_{\!1}(\p_{\!t} h H\p_{\!t} h)\big]
 - gH \big[(\p_{\!1} h) \p_{\!1}(h \partial_1 h)\big] \nonumber\\
&\quad + 2 H \big[(\p_{\!1}h) (\p_{\!1} \p_{\!t} h) H\p_{\!t} h\big] + gH \big[(\p_{\!1}h)^3\big] - \p_{\!1} \big[h (\p_{\!t} h) (\partial_1 \p_{\!t} h) \big] \\
&\quad + \dfrac{g}{2} \p_{\!1} \big[h^2 (\Lambda \p_{1} h)\big] - H \big[h (\p_{\!t} h)\Lambda \partial_1\p_{\!t} h\big] - \dfrac{g}{2} H (h^2 \partial_1^3h)\\
&\quad +\Lambda \left[H\p_{\!t} h\Lambda \big(\comm{h}{H}  \p_{\!t} h\big)\right]- \Lambda \left[ H\p_{\!t} h \partial_t h \p_{\!1} h \right]\,.
\end{align*}

Therefore,
\begin{align}
\p_{\!t}^2 h + g \Lambda h &= -\Lambda \big[(H\p_{\!t} h)^2\big] + g \p_{\!1} \big(h \p_{\!1} h \big) +g\Lambda \big(h \Lambda h \big)+ \mathcal{Q}(h)\,.
\label{cubic-h-model}
\end{align}

Making use of the commutator identities in Appendix \ref{appendix1}, the cubic nonlinearity $\mathcal{Q}(h)$ can be written as
\begin{align}
\mathcal{Q}(h) &= - \p_{\!1}\bigg{[} \comm{\p_{\!t} h}{H} \p_{\!1} \big(h H \p_{\!t} h \big)- \comm{h}{H} \p_{\!1} (\comm{\p_{\!t} h}{H} \p_{\!t} h)   - g \comm{h}{H} \Lambda (h \Lambda h) \nonumber\\
&\qquad\quad + H \big[(h \p_{\!t} h) (\Lambda \p_{\!t}h)\big] - \frac{g}{2}\comm{h^2}{H}\partial_1^2h+h \p_{\!t} h \partial_1 \p_{\!t} h  \nonumber\\
&\qquad\quad - H\big[(H \p_{\!t} h) \Lambda \big(h H \p_{\!t} h\big)\big] -H\big[(H \p_{\!t} h)  h (\p_{\!1}\p_{\!t} h)\big]\bigg{]}\,. \label{R(h)}
\end{align}

\section{Well-posedness of the $h$-models}\label{Catalansection}
In this section we develop a well-posedness theory for the quadratic and cubic $h-$models. In particular, we prove local well-posedness in the Wiener spaces \eqref{X_tau} for the quadratic $h-$model \eqref{hmodel-quad} and the cubic $h-$model \eqref{cubic-h-model}.

To prove these results, we use the linear recursion derived in \eqref{eq:hmodelfourier} and \eqref{forcingwave}. For the Navier-Stokes equation, the idea of using
 an asymptotic expansion together with the diffusive properties of the semigroup to prove well-posedness goes back to Oseen \cite{Os1912} and Knightly 
 \cite{Kn1966}. However, in our case, the $h-$model is hyperbolic and, consequently, there is no diffusive properties of the semigroup. Furthermore, as we 
 consider the case of periodic waves, the semigroup does not have dispersive effects either. Instead of relying upon smoothing properties of the semigroup, we use the structure of the nonlinearity. This structure allow us to write an inequality for the $X_{\tau}$-norms \eqref{X_tau} that resembles the recursion for the \emph{Catalan} numbers $\{\mathcal{C}_k\}_{k=0}^\infty$. These numbers can be defined recursively as
\begin{equation}\label{Catalan}
\mathcal{C}_k = \sum_{j=0}^{k-1} \mathcal{C}_j \mathcal{C}_{k-1-j}\,,\qquad \mathcal{C}_0 = 1\,.
\end{equation}
Then, we finalize the argument using well-known growth rates of the Catalan numbers. This idea of using the structure of the nonlinearity and the Catalan numbers is, to the best of our knowledge, new to the analysis of water waves models (an wave equations in general).

We remark that  a closely  related two-fluid asymptotic model  derived in \cite{GrSh2017} has been shown to be well-posed in Sobolev spaces when
the initial data satisfies a certain sign condition.   In particular, in the case that $ \partial_t h|_{t=0} <0$ for each point on the free-surface, that model is locally well-posed for arbitrary data 
and globally well-posed under certain size restrictions (see   Theorems 7.1 and 7.6  in \cite{GrSh2017}).\footnote{The condition $\p_t h <0|_{t=0}$ can occur
globally  with certain
in-flow boundary conditions.} 
  It is possible that the quadratic $h$-model is also well-posed in Sobolev spaces when this sign condition holds for the initial data, and we plan to investigate this
  in future work.

\subsection{Well-posedness theory for the quadratic $h$-model}
Using the ansatz \eqref{ansatz}, the quadratic $h-$model \eqref{hmodel-quad} can be written as
\begin{equation}\label{eq:u}
\partial_t^2 \widetilde{h}=-g\Lambda \widetilde{h}- \epsilon \Lambda(H\partial_t \widetilde{h})^2+g\epsilon\left(\bp(\widetilde{h}\bp \widetilde{h})+\Lambda(\widetilde{h}\Lambda \widetilde{h})\right),
\end{equation}
with initial conditions \eqref{initialdata2}.
We expand $\widetilde{h}$ as in \eqref{expansions} for functions $h_k: \mathbb{S}^1 \to \bbR$ to be determined. Substituting into \eqref{eq:u} we find that
\begin{align}\label{eq:uk}
\p^2_{\!t}h_{k}&=-g\Lambda h_k
+ \sum_{j=0}^{k-1}\big[
\Lambda (gh_{j}\Lambda h_{k-1-j}  - H\p_{\!t}h_{j}H\p_{\!t}h_{k-1-j})  + g\p_{\!1}( h_{j}\bp h_{k-1-j})
\big]
\end{align}
with initial conditions
\begin{align}\label{initialdata3}
h_0(x_1,0)&=\frac{\hinit(x_1)}{\epsilon}\,, \ \p_{\!t}h_0(x_1,0)=\frac{\htinit(x_1)}{\epsilon}\,, \nonumber\\
h_k(x_1,0)&= \p_{\!t}h_k(x_1,0)=0\;\;k\geq1 \,.
\end{align}
Using the Fourier series expansion,
 (\ref{eq:uk}) shows that each Fourier component satisfies the differential equation
\begin{align}\label{eq:hmodelfourier}
\p_{\!t}^2 \ft{h}_k(\ell,t) &= -g |\ell| \ft{h}_k(\ell,t) +f(\ell,t) \,,
\end{align}
where $f$ is given by
\begin{align}
f(\ell,t) &= \sum_{j=0}^{k-1}\sum_{m=-\infty}^{\infty}\bigg{[}-|\ell|\left(\left(-i\text{sgn}(\ell-m)\right)\p_{\!t}\ft{h}_{j}(\ell-m,t)\left(-i\text{sgn}(m)\right)\p_{\!t}\ft{h}_{k-1-j}(m,t)\right)\nonumber\\
&\quad+g\left( i\ell\ft{h}_{j}(\ell-m,t)im \ft{h}_{k-1-j}(m,t) +
|\ell|\ft{h}_{j}(\ell-m,t)|m|\ft{h}_{k-1-j}(m,t) \right) \bigg{]}\,.\label{forcingwave}
\end{align}

We note that  integration of \eqref{eq:uk} shows that
\begin{equation}\label{zero_average}
\int_{\mathbb{S}^1} h_k(x_1,t)=0 \ \text{ and } \ \int_{\mathbb{S}^1} \partial_th_k(x_1,t)ds=0\;\forall\,t\geq0\,.
\end{equation}


Solving the ODE \eqref{eq:hmodelfourier}  for $k=0$, we find that
$$
\ft{h}_0(\ell,t)=\ft{h}_0(\ell,0)\cos\left(\sqrt{g|\ell|}t\right)+\frac{\partial_t\ft{h}_{0}(\ell,0)}{\sqrt{g|\ell|}}\sin\left(\sqrt{g|\ell|}t\right).
$$
Similarly, the solution to  \eqref{eq:hmodelfourier} for $k>0$ is
\begin{align*}
\ft{h}_k(\ell,t) &= - \frac{\cos \sqrt{g|\ell|} t}{\sqrt{g |\ell|}} \int_0^t f(s) \sin\left( \sqrt{g|\ell|} s\right) \,ds + \frac{\sin \sqrt{g|\ell|} t}{\sqrt{g |\ell|}} \int_0^t f(s) \cos\left( \sqrt{g|\ell|} s\right) \,ds \\
&= \frac{1}{\sqrt{g |\ell|}} \int_0^t f(s) \sin\left( \sqrt{g|\ell|} (t-s)\right) \,ds \,,
\end{align*}
and hence
\begin{align}\label{recurrence}
\p_{\!t} \ft{h}_k(\ell,t) = \int_0^t f(s) \cos\left( \sqrt{g|\ell|} (t-s)\right) \,ds\,.
\end{align}
Using the expression \eqref{forcingwave}, we have that for $k\geq1$,  $\ft{h}_k$ and $\p_{\!t} \ft{h}_k$  verify the following recursion relations:
\begin{align}
& \ft{h}_k(\ell,t) \nonumber\\
&\ \ = \frac{1}{\sqrt{g |\ell|}}\int_0^t\bigg{[}\!-\sum_{j=0}^{k-1}\sum_{m=-\infty}^{\infty}\!|\ell|\left(\left(-i\text{sgn}(\ell-m)\right)\p_{\!t}\ft{h}_{j}(\ell-m,s)\left(-i\text{sgn}(m)\right)\p_{\!t}\ft{h}_{k-1-j}(m,s)\right)\nonumber\\
&\qquad\qquad\qquad +g\sum_{j=0}^{k-1}\sum_{m=-\infty}^{\infty} i\ell\ft{h}_{j}(\ell-m,s)im \ft{h}_{k-1-j}(m,s)\nonumber\\
&\qquad\qquad\qquad +g\sum_{j=0}^{k-1}\sum_{m=-\infty}^{\infty}|\ell|\ft{h}_{j}(\ell-m,s)|m|\ft{h}_{k-1-j}(m,s)\bigg{]} \sin\left( \sqrt{g|\ell|} (t-s)\right) \,ds\,,\label{eq:recurrence}
\end{align}
and
\begin{align}
& \p_{\!t} \ft{h}_k(\ell,t) \nonumber\\
&\quad = \int_0^t\bigg{[}-\sum_{j=0}^{k-1}\sum_{m=-\infty}^{\infty}|\ell|\left(\left(-i\text{sgn}(\ell-m)\right)\p_{\!t}\ft{h}_{j}(\ell-m,s)\left(-i\text{sgn}(m)\right)\p_{\!t}\ft{h}_{k-1-j}(m,s)\right)\nonumber\\
&\qquad\quad\ + g\sum_{j=0}^{k-1}\sum_{m=-\infty}^{\infty} i\ell\ft{h}_{j}(\ell-m,s)im \ft{h}_{k-1-j}(m,s)\nonumber\\
&\qquad\quad\ + g\sum_{j=0}^{k-1}\sum_{m=-\infty}^{\infty}|\ell|\ft{h}_{j}(\ell-m,s)|m|\ft{h}_{k-1-j}(m,s)\bigg{]} \cos\left( \sqrt{g|\ell|} (t-s)\right) \,ds\,.\label{eq:recurrence2}
\end{align}

We now establish our existence theory for the quadratic h-model, and simultaneously prove that the Stokes expansion converges.
\begin{theorem}\label{theorem1} Let $\epsilon>0$, $g > 0$ and the initial data in \eqref{initialdata} $(\hinit, \htinit)$ be given. Assume that there exists $1<D<\infty$ such that
$$
\hinit(x_1)=\sum_{|\ell|\leq D}\fhinit(\ell)e^{i\ell x_1}\,,
$$
$$
\htinit(x_1)=\sum_{|\ell|\leq D,\ell\neq0}\fhtinit(\ell)e^{i\ell x_1}\,.
$$
Then there exists a unique analytic solution $h(x_1,t)\in C([0,T];X_{0.5})$ to {\rm(\ref{hmodel-quad})} for $t$ in the time interval $[0,T]$ with
$$
T=\frac{1}{e \epsilon 4C(\|\Lambda^{0.5}\hinit\|_{X_0},\|\htinit\|_{X_0})} \,.
$$
Equivalently, the Stokes expansion for the quadratic $h-$model \eqref{hmodel-quad} converges for arbitrary $\epsilon>0$  and $T>0$
taken sufficiently small.
\end{theorem}
\begin{proof}
\textbf{Existence.}
We fix $R\in \bbZ^+$ such that
$$
\frac{D(R+1)}{1+R^2}\leq 1.
$$
Given $\epsilon>0$, we seek  solutions $h$ of {\rm(\ref{hmodel-quad})} having the form
\begin{equation}\label{series}
h(x_1,t)=\sum_{k=0}^\infty\epsilon^{k+1}h_k(x_1,t)\text{ \ and \ }\partial_t h(x_1,t)=\sum_{k=0}^\infty\epsilon^{k+1}\partial_t h_k(x_1,t).
\end{equation}
The series in \eqref{series} are respectively bounded by
\begin{equation}\label{major}
\sup_{0\leq t\leq T}\sum_{k=0}^\infty\epsilon^{k+1}\|h_k(t)\|_{X_1}\text{ \ and \ }\sup_{0\leq t\leq T}\sum_{k=0}^\infty\epsilon^{k+1}\|\partial_t h_k(t)\|_{X_1}.
\end{equation}
Thus, by proving the boundedness of \eqref{major}, we obtain the absolute convergence of \eqref{series} and, in particular, the existence of solutions to {\rm(\ref{hmodel-quad})}.

To obtain the required estimates, we first consider the truncated series (for $0<k\leq R$)
$$
\sum_{j=0}^R\epsilon^{j+1}h_j(x_1,t).
$$
Using \eqref{eq:recurrence}, we have that
\begin{align*}
& \|h_k(t)\|_{X_{R+1-k}} \\
&\qquad\leq \sum_{\ell=-\infty}^{\infty}\frac{e^{(R+1-k)|\ell|}}{\sqrt{g |\ell|}}\int_0^t\bigg{[}\sum_{j=0}^{k-1}\sum_{m=-\infty}^{\infty}|\ell||\p_{\!t}\ft{h}_{j}(\ell-m,s)||\p_{\!t}\ft{h}_{k-1-j}(m,s)|\nonumber\\
&\qquad\qquad\qquad\qquad\qquad\qquad\! +2g\sum_{j=0}^{k-1}\sum_{m=-\infty}^{\infty} |\ell||\ft{h}_{j}(\ell-m,s)||m| |\ft{h}_{k-1-j}(m,s)|\bigg{]}\,ds\\
&\qquad \leq \frac{1}{\sqrt{g}}\int_0^t\bigg{[}\sum_{j=0}^{k-1}\|\p_{\!t}h_{j}(s)\|_{X_{R+2-k}}\|\p_{\!t}h_{k-1-j}(s)\|_{X_{R+2-k}}\\
&\qquad\qquad\qquad +8g\sum_{j=0}^{k-1}\|h_{j}(s)\|_{X_{R+1.5-k}}\|h_{k-1-j}(s)\|_{X_{R+2-k}}\bigg{]}\,ds\\
&\qquad\leq \frac{1}{\sqrt{g}}\int_0^t\bigg{[}\sum_{j=0}^{k-1}\|\p_{\!t}h_{j}(s)\|_{X_{R+2-k}}\|\p_{\!t}h_{k-1-j}(s)\|_{X_{R+2-k}}\\
&\qquad\qquad\qquad +8g\sum_{j=0}^{k-1}\|h_{j}(s)\|_{X_{R+2-k}}\|h_{k-1-j}(s)\|_{X_{R+2-k}}\bigg{]}\,ds\,,
\end{align*}
where we have used Tonelli's theorem together with the fact that
$\ft{h_k}(0,t)=0$, which follows from \eqref{zero_average},
and the important inequality
\begin{equation}\label{eq:aux}
|\ell|\leq ce^{\frac{|\ell|}{c}}\leq ce^{\frac{|\ell-m|+|m|}{c}}\;\forall\,c\in \bbZ^+\,.
\end{equation}
Using \eqref{eq:recurrence2}, we can find a similar bound $\partial_t h_k(t)$:
\begin{align*}
\|\partial_t h_k(t)\|_{X_{R+1-k}} &\leq \frac{1}{\sqrt{g}}\int_0^t\bigg{[}\sum_{j=0}^{k-1}\|\p_{\!t}h_{j}(s)\|_{X_{R+2-k}}\|\p_{\!t}h_{k-1-j}(s)\|_{X_{R+2-k}}\\
&\qquad\qquad +8g\sum_{j=0}^{k-1}\|h_{j}(s)\|_{X_{R+2-k}}\|h_{k-1-j}(s)\|_{X_{R+2-k}}\bigg{]}\,ds\,.
\end{align*}
Since $R+2-k\leq R+2-k+j =R+1-(k-1-j)$, it follows that
$$
\|u_{k-1-j}(s)\|_{X_{R+2-k}}\leq \|u_{k-1-j}(s)\|_{X_{R+1-(k-1-j)}}.
$$
Similarly, if $j\leq k-1$, then $R+2-k=R+1-(k-1)\leq R+1-j$ and
$$
\|u_{j}(s)\|_{X_{R+2-k}}\leq \|u_{j}(s)\|_{X_{R+1-j}}.
$$
Then, we have that
\begin{align}
& \|h_k(t)\|_{X_{R+1-k}}+\|\p_{\!t} h_k(t)\|_{X_{R+1-k}}\nonumber\\
&\qquad\leq \frac{2}{\sqrt{g}}\int_0^t\bigg{[}\sum_{j=0}^{k-1}\|\p_{\!t}h_{j}(s)\|_{X_{R+2-k}}\|\p_{\!t}h_{k-1-j}(s)\|_{X_{R+2-k}}
\nonumber\\
&\qquad\qquad\qquad\ +16g\sum_{j=0}^{k-1}\|h_{j}(s)\|_{X_{R+2-k}}\|h_{k-1-j}(s)\|_{X_{R+2-k}}\bigg{]}\,ds\nonumber\\
&\qquad\leq \frac{2}{\sqrt{g}}\int_0^t\bigg{[}\sum_{j=0}^{k-1}\|\p_{\!t}h_{j}(s)\|_{X_{R+1-j}}\|\p_{\!t}h_{k-1-j}(s)\|_{X_{R+1-(k-1-j)}}\nonumber\\
&\qquad\qquad\qquad\ +16g\sum_{j=0}^{k-1}\|h_{j}(s)\|_{X_{R+1-j}}\|h_{k-1-j}(s)\|_{X_{R+1-(k-1-j)}}\bigg{]}\,ds\nonumber\\
&\qquad\leq \max\left\{\frac{2}{\sqrt{g}},16g\right\}\int_0^t\bigg{[}\sum_{j=0}^{k-1}\|\p_{\!t}h_{j}(s)\|_{X_{R+1-j}}\|\p_{\!t}h_{k-1-j}(s)\|_{X_{R+1-(k-1-j)}}\nonumber\\
&\qquad\qquad\qquad\qquad\qquad\quad\ \ +\sum_{j=0}^{k-1}\|h_{j}(s)\|_{X_{R+1-j}}\|h_{k-1-j}(s)\|_{X_{R+1-(k-1-j)}}\bigg{]}\,ds\,.\label{eq:aux2}
\end{align}
We define
$$
\mathscr{A}_k(t)=\max\left\{\frac{2}{\sqrt{g}},16g\right\}e^{-\frac{k+1}{1+\sgn(k)R^2}D(R+1)}\left[\|h_k(t)\|_{X_{R+1-k}}+\|\p_{\!t} h_k(t)\|_{X_{R+1-k}}\right].
$$
Next we show that for $0\le j\le k-1$,
\begin{equation}\label{proof1_ineq}
\frac{k+1}{1+\sgn(k)R^2}\leq \frac{\left(k-1-j+1\right)}{1+\sgn(k-1-j)R^2}+\frac{(j+1)}{1+\sgn(j)R^2}\,.
\end{equation}
Note that (\ref{proof1_ineq}) clearly holds for $0 < j < k-1$.
When $j=0$,
\begin{align*}
\frac{k+1}{1+\sgn(k)R^2}&=\frac{\left(k-1-j+1\right)}{1+R^2}+\frac{(j+1)}{1+R^2}\\
&\leq \frac{\left(k-1-j+1\right)}{1+R^2}+(j+1)\\
&\leq \frac{\left(k-1-j+1\right)}{1+\sgn(k-1-j)R^2}+\frac{(j+1)}{1+\sgn(j)R^2}\;.
\end{align*}
Similarly, when $j=k-1$,
\begin{align*}
\frac{k+1}{1+\sgn(k)R^2}&=\frac{\left(k-1-j+1\right)}{1+R^2}+\frac{(j+1)}{1+R^2}\\
&\leq \left(k-1-j+1\right)+\frac{(j+1)}{1+\sgn(j)R^2}\\
&\leq \frac{\left(k-1-j+1\right)}{1+\sgn(k-1-j)R^2}+\frac{(j+1)}{1+\sgn(j)R^2}\;.
\end{align*}
Thus, (\ref{proof1_ineq}) holds for $0\leq j\leq k-1$, and this further implies that
$$
e^{-D(R+1)\frac{k+1}{1+\sgn(k)R^2}}\leq e^{-D(R+1)\frac{k-1-j+1}{1+\sgn(k-1-j)R^2}}e^{-D(R+1)\frac{j+1}{1+\sgn(j)R^2}}\,.
$$
We then obtain that the previous recursion for $\|h_k(t)\|_{X_{R+1-k}}+\|\p_{\!t} h_k(t)\|_{X_{R+1-k}}$ can equivalently be stated as
$$
\mathscr{A}_k(t)\leq \int_0^t \sum_{j=0}^{k-1} \mathscr{A}_{k-1-j}(s)\mathscr{A}_{j}(s) \,ds.
$$
We observe that
\begin{align}
\mathscr{A}_0(t)&=\max\left\{\frac{2}{\sqrt{g}},16g\right\}e^{-D(R+1)}\left[\|h_0(t)\|_{X_{R+1}}+\|\p_{\!t} h_0(t)\|_{X_{R+1}}\right]\nonumber\\
&\leq \max\left\{\frac{2}{\sqrt{g}},16g\right\}e^{-D(R+1)}\sum_{|\ell|\leq D}e^{(R+1)|\ell|}\left(|\ft{h}_0(\ell,0)|+\left|\frac{\partial_t\ft{h}_{0}(\ell,0)}{\sqrt{g|\ell|}}\right|\right)\nonumber\\
&\leq C(\|\Lambda^{0.5}\hinit\|_{X_0},\|\htinit\|_{X_0}).\label{eq:aux3}
\end{align}
Then, we want to prove by induction that
\begin{equation}\label{rs1}
\mathscr{A}_k\leq  \mathcal{C}_kt^k\left[C(\|\Lambda^{0.5}\hinit\|_{X_0},\|\htinit\|_{X_0})\right]^{k+1} \,,
\end{equation}
where $\mathcal{C}_k$ are the Catalan numbers \eqref{Catalan}. Remarkably, the Catalan numbers $\mathcal{C}_k = \O(k^{-\frac{3}{2}} 4^k)$ as $k\to \infty$  \cite[page 136]{St2015}.

Having already established that \eqref{rs1} holds for $k=0$, we proceed with the induction step. For $1\leq k,$ we have that
\begin{align*}
\mathscr{A}_k(s)&\leq\int_0^t \sum_{j=0}^{k-1} \mathscr{A}_{k-1-j}(s)\mathscr{A}_{j}(s) \,ds\\
&\leq \left[C(\|\Lambda^{0.5}\hinit\|_{X_0},\|\htinit\|_{X_0})\right]^{k+1}\int_0^t \sum_{j=0}^{k-1}  \mathcal{C}_{k-1-j}s^{k-1-j} \mathcal{C}_{j}s^{j} \,ds\\
&\leq \left[C(\|\Lambda^{0.5}\hinit\|_{X_0},\|\htinit\|_{X_0})\right]^{k+1}\mathcal{C}_k\int_0^t s^{k-1} \,ds\\
&\leq \left[C(\|\Lambda^{0.5}\hinit\|_{X_0},\|\htinit\|_{X_0})\right]^{k+1}\mathcal{C}_k\frac{t^k}{k}.
\end{align*}
Thus, using the asymptotic growth of the Catalan numbers, we have that
\begin{align*}
\| h_k(t)\|_{X_{1}}+\|\p_{\!t} h_k(t)\|_{X_{1}}&\leq\|h_k(t)\|_{X_{R+1-k}}+\|\p_{\!t} h_k(t)\|_{X_{R+1-k}}\\
&\leq e^{(k+1)\frac{D(R+1)}{1+R^2}}t^k4^k \left[C(\|\Lambda^{0.5}\hinit\|_{X_0},\|\htinit\|_{X_0})\right]^{k+1}.
\end{align*}
Analogously,
$$
\|h_0(t)\|_{X_{1}}+\|\p_{\!t} h_0(t)\|_{X_{1}}\leq \widetilde{C}(\|\Lambda^{0.5}\hinit\|_{X_0},\|\htinit\|_{X_0}).
$$
We define the series
\begin{align*}
I^1_R&= \p_{\!t} h_0(x_1,t) + \epsilon \p_{\!t}h_1(x_1,t) +\epsilon^2 \p_{\!t}h_2(x_1,t) + \cdots + \epsilon^R \p_{\!t}h_R(x_1,t)\,, \\
I^2_R&= h_0(x_1,t) + \epsilon h_1(x_1,t) +\epsilon^2 h_2(x_1,t) + \cdots + \epsilon^R h_R(x_1,t)\,.
\end{align*}
Then
\begin{align*}
\|I^1_R\|_{X_1}&\leq \widetilde{C}(\|\Lambda^{0.5}\hinit\|_{X_0},\|\htinit\|_{X_0})\\
&\quad+e^{\frac{D(R+1)}{1+R^2}}C(\|\Lambda^{0.5}\hinit\|_{X_0},\|\htinit\|_{X_0})\sum_{k=1}^R\left(e^{\frac{D(R+1)}{1+R^2}}\epsilon t 4 C(\|\Lambda^{0.5}\hinit\|_{X_0},\|\htinit\|_{X_0})\right)^k\\
&\leq \widetilde{C}(\|\Lambda^{0.5}\hinit\|_{X_0},\|\htinit\|_{X_0})\\
&\quad+e C(\|\Lambda^{0.5}\hinit\|_{X_0},\|\htinit\|_{X_0})\sum_{k=1}^R\left(e\epsilon t 4 C(\|\Lambda^{0.5}\hinit\|_{X_0},\|\htinit\|_{X_0})\right)^k.
\end{align*}
Similarly,
\begin{align*}
\|I^2_R\|_{X_1}&\leq \widetilde{C}(\|\Lambda^{0.5}\hinit\|_{X_0},\|\htinit\|_{X_0})\\
&\quad+eC(\|\Lambda^{0.5}\hinit\|_{X_0},\|\htinit\|_{X_0})\sum_{k=1}^R\left(e\epsilon t 4 C(\|\Lambda^{0.5}\hinit\|_{X_0},\|\htinit\|_{X_0})\right)^k.
\end{align*}
We conclude that if
$$
t<\frac{1}{e \epsilon 4C(\|\Lambda^{0.5}\hinit\|_{X_0},\|\htinit\|_{X_0})},
$$
then we can take the limit in $R$ and we compute that
$$
\p_{\!t}h(x_1,t)=I^1_\infty\text{ \ and \ } h(x_1,t)=I^2_\infty.
$$

Our estimates lead to
$$
h,\p_{\!t}h\in L^\infty(0,T;X_{1}).
$$
Moreover, using the Cauchy product of power series, we have that
\begin{align*}
\ft{h}(\ell,t) &= \ft{h}_0(\ell,0)\cos\left(\sqrt{g|\ell|}t\right)+\partial_t\ft{h}_{0}(\ell,0)\sin\left(\sqrt{g|\ell|}t\right)+\frac{1}{\sqrt{g |\ell|}} \int_0^t \mathcal{N}(\ell,s) \sin\left( \sqrt{g|\ell|} (t-s)\right) \,ds \,,\\
\p_{\!t} \ft{h}(\ell,t) &= -\ft{h}_0(\ell,0)\sin\left(\sqrt{g|\ell|}t\right)+\partial_t\ft{h}_{0}(\ell,0)\cos\left(\sqrt{g|\ell|}t\right)+\int_0^t \mathcal{N}(\ell,s) \cos\left( \sqrt{g|\ell|} (t-s)\right) \,ds\,,
\end{align*}
where
\begin{align*}
\mathcal{N}(\ell,t) &= \sum_{m=-\infty}^{\infty}\bigg{[}-|\ell|\left(\left(-i\text{sgn}(\ell-m)\right)\p_{\!t}\ft{h}(\ell-m,t)\left(-i\text{sgn}(m)\right)\p_{\!t}\ft{h}(m,t)\right)\nonumber\\
&\qquad\qquad +g\left( i\ell\ft{h}(\ell-m,t)im \ft{h}(m,t) +
|\ell|\ft{h}(\ell-m,t)|m|\ft{h}(m,t) \right) \bigg{]}\,.
\end{align*}
Since $h$ and $\p_{\!t}h$ are analytic functions in space, using the previous expression, we obtain that $h$ and $\p_{\!t}h$ satisfy
$$
h,\p_{\!t}h\in C([0,T],X_{0.5}).
$$
In particular, they are continuous functions in time, $h,\p_{\!t}h\in C([0,T]\times \mathbb{S}^1)$.

\textbf{Uniqueness.} Let us assume that there exist two solutions $h^{(1)},h^{(2)}\in C([0,T],X_{0.5})$ emanating from the same initial data. Then, the difference
$$
z=h^{(1)}-h^{(2)}
$$
satisfies
\begin{align*}
\ft{z}(\ell,t) &= \frac{1}{\sqrt{g |\ell|}} \int_0^t \mathcal{M}(\ell,s) \sin\left( \sqrt{g|\ell|} (t-s)\right) \,ds \,,\\
\p_{\!t} \ft{z}(\ell,t) &= \int_0^t \mathcal{M}(\ell,s) \cos\left( \sqrt{g|\ell|} (t-s)\right) \,ds\,,
\end{align*}
with
\begin{align*}
\mathcal{M}(\ell,t) &= \sum_{m=-\infty}^{\infty}\bigg{[}-|\ell|\left(\left(-i\text{sgn}(\ell-m)\right)\p_{\!t}\ft{z}(\ell-m,t)\left(-i\text{sgn}(m)\right)\p_{\!t}\ft{h}^{(1)}(m,t)\right)\nonumber\\
&\qquad\qquad-|\ell|\left(\left(-i\text{sgn}(\ell-m)\right)\p_{\!t}\ft{h}^{(2)}(\ell-m,t)\left(-i\text{sgn}(m)\right)\p_{\!t}\ft{z}(m,t)\right)\nonumber\\
&\qquad\qquad+g\left( i\ell\ft{z}(\ell-m,t)im \ft{h}^{(1)}(m,t) +
|\ell|\ft{z}(\ell-m,t)|m|\ft{h}^{(1)}(m,t) \right)\nonumber\\
&\qquad\qquad+g\left( i\ell\ft{h}^{(2)}(\ell-m,t)im \ft{z}(m,t) +
|\ell|\ft{h}^{(2)}(\ell-m,t)|m|\ft{z}(m,t) \right) \bigg{]}\,.
\end{align*}
Then, following the same argument as in the previous section, we expand $h^{(j)}_k$, $j=1,2$ as in \eqref{recurrence} and find that $h^{(j)}_k$, $j=1,2$ satisfy the cascade of linear problems \eqref{eq:hmodelfourier} and \eqref{forcingwave}. Equivalently, we have that
$$
z_k=h^{(1)}_k-h^{(2)}_k,
$$
satisfies
\begin{align*}
\ft{z}_k(\ell,t) &= \frac{1}{\sqrt{g |\ell|}} \int_0^t f_k(\ell,s) \sin\left( \sqrt{g|\ell|} (t-s)\right) \,ds \,,\\
\p_{\!t} \ft{z}_k(\ell,t) &= \int_0^t f_k(\ell,s) \cos\left( \sqrt{g|\ell|} (t-s)\right) \,ds\,,
\end{align*}
with
\begin{align*}
f_k(\ell,t) &= \sum_{j=0}^{k-1}\sum_{m=-\infty}^{\infty}\bigg{[}-|\ell|\left(\left(-i\text{sgn}(\ell-m)\right)\p_{\!t}\ft{z}_j(\ell-m,t)\left(-i\text{sgn}(m)\right)\p_{\!t}\ft{h}^{(1)}_{k-1-j}(m,t)\right)\nonumber\\
&\qquad\qquad\quad\ -|\ell|\left(\left(-i\text{sgn}(\ell-m)\right)\p_{\!t}\ft{h}^{(2)}_{k-1-j}(\ell-m,t)\left(-i\text{sgn}(m)\right)\p_{\!t}\ft{z}_j(m,t)\right)\nonumber\\
&\qquad\qquad\quad\ +g\left( i\ell\ft{z}_j(\ell-m,t)im \ft{h}^{(1)}_{k-1-j}(m,t) +
|\ell|\ft{z}_j(\ell-m,t)|m|\ft{h}^{(1)}_{k-1-j}(m,t) \right)\nonumber\\
&\qquad\qquad\quad\ +g\left( i\ell\ft{h}^{(2)}_{k-1-j}(\ell-m,t)im \ft{z}_j(m,t) +
|\ell|\ft{h}^{(2)}_{k-1-j}(\ell-m,t)|m|\ft{z}_j(m,t) \right) \bigg{]}\,.
\end{align*}
As before, we consider $R\in \bbZ^+$ such that $\dfrac{D(R+1)}{1+R^2}\leq 1,$ and define
$$
\mathscr{A}^{(1)}_k(t)=\max\left\{\frac{2}{\sqrt{g}},16g\right\}e^{-\frac{k+1}{1+\sgn(k)R^2}D(R+1)}\left[\|h_k^{(1)}(t)\|_{X_{R+1-k}}+\|\p_{\!t} h_k^{(1)}(t)\|_{X_{R+1-k}}\right],
$$
$$
\mathscr{A}^{(2)}_k(t)=\max\left\{\frac{2}{\sqrt{g}},16g\right\}e^{-\frac{k+1}{1+\sgn(k)R^2}D(R+1)}\left[\|h_k^{(2)}(t)\|_{X_{R+1-k}}+\|\p_{\!t} h_k^{(2)}(t)\|_{X_{R+1-k}}\right],
$$
$$
\mathscr{B}_k(t)=\max\left\{\frac{2}{\sqrt{g}},16g\right\}e^{-\frac{k+1}{1+\sgn(k)R^2}D(R+1)}\big[\|z_k(t)\|_{X_{R+1-k}}+\|\p_{\!t} z_k(t)\|_{X_{R+1-k}}\big].
$$
Following the arguments in the previous section, we find that
\begin{align*}
\mathscr{A}^{(j)}_k&\leq  \mathcal{C}_kt^k\left[C(\|\Lambda^{0.5}\hinit\|_{X_0},\|\htinit\|_{X_0})\right]^{k+1},\;j=1,2\\
\mathscr{B}_k(t)&\leq \int_0^t \sum_{j=0}^{k-1} \left(\mathscr{A}^{(1)}_{k-1-j}(s)+\mathscr{A}^{(2)}_{k-1-j}(s)\right)\mathscr{B}_{j}(s) \,ds,\\
\mathscr{B}_0(t)&=0.
\end{align*}
Due to the previous inequalities, we prove that $\mathscr{B}_k(t)=0$ using induction and we conclude the uniqueness.
\end{proof}

\begin{remark}
Let us emphasize that the estimate is independent of $D$. Thus, an appropriate passage to the limit allows for analytic initial data whose Fourier series has unbounded support. The argument to prove this generalization is straightforward and we leave it for the interested reader.
\end{remark}

\subsection{Well-posedness theory for the cubic $h$-model}
As we have seen, the cubic $h-$model can be written as the following nonlinear wave equation \eqref{cubic-h-model}
$$
\partial_t^2 h=-g\Lambda h-\Lambda \big[(H\p_{\!t} h)^2\big] + g \p_{\!1} \big(h \p_{\!1} h \big) +g\Lambda \big(h \Lambda h \big)+ \mathcal{Q}(h)\,,
$$
where $\mathcal{Q}(h)$ is defined in \eqref{R(h)}. Using the ansatz \eqref{ansatz}, the cubic $h-$model \eqref{cubic-h-model} can be written as
\begin{equation}\label{eq:u2}
\partial_t^2 \widetilde{h}=-g\Lambda \widetilde{h}- \epsilon \Lambda(H\partial_t \widetilde{h})^2+g\epsilon\left(\bp(\widetilde{h}\bp \widetilde{h})+\Lambda(\widetilde{h}\Lambda \widetilde{h})\right)+\epsilon^2 \mathcal{Q}(\widetilde{h}),
\end{equation}
with initial conditions \eqref{initialdata2}.
We again  consider the expansion
$\widetilde{h}(x_1,t) = h_0(x_1,t) + \epsilon h_1(x_1,t) + \epsilon^2 h_2(x_1,t) + \cdots \,. $
The quadratic nonlinearity follows as in \eqref{eq:uk}.
It thus suffices to expand the cubic nonlinearity.
We define
\begin{align}
\Q(h_r,h_{j-r},h_{k-2-j})&=- \p_{\!1}\bigg{[} \comm{\p_{\!t} h_r}{H} \p_{\!1} \big(h_{j-r} H \p_{\!t} h_{k-2-j} \big)- \comm{h_r}{H} \p_{\!1} (\comm{\p_{\!t} h_{j-r}}{H} \p_{\!t} h_{k-2-j})  \nonumber\\
&\quad - g \comm{h_r}{H} \Lambda (h_{j-r} \Lambda h_{k-2-j}) + H \big[(h_r \p_{\!t} h_{j-r}) (\Lambda \p_{\!t}h_{k-2-j})\big] \nonumber\\
&\quad- \frac{g}{2}\comm{h_rh_{j-r}}{H}\partial_1^2h_{k-2-j}+h_r \p_{\!t} h_{j-r} \partial_1 \p_{\!t} h_{k-2-j}  \nonumber\\
&\quad  -H\big[(H \p_{\!t} h_r) \Lambda \big(h_{j-r} H \p_{\!t} h_{k-2-j}\big)\big]\nonumber\\
&\quad -H\big[(H \p_{\!t} h_r)  h_{j-r} (\p_{\!1}\p_{\!t} h_{k-2-j})\big]\bigg{]}
\,.\label{eq:Q}
\end{align}

Comparing powers of $\epsilon$, we find that
\begin{align}\label{eq:uk2}
\p^2_{\!t}h_{k}&=-g\Lambda h_k
+ \sum_{j=0}^{k-1}\left(
\Lambda \left[gh_{j}\Lambda h_{k-1-j}  - H\p_{\!t}h_{j}H\p_{\!t}h_{k-1-j}\right]  + g\p_{\!1} \left[ h_{j}\bp h_{k-1-j}\right]
\right)\nonumber\\
&\quad +\sum_{j=0}^{k-2}\sum_{r=0}^{j}\Q(h_r,h_{j-r},h_{k-2-j})\,,
\end{align}
with initial conditions \eqref{initialdata3}


Our starting point is the linear recursion \eqref{eq:uk2} and \eqref{eq:Q}. Similarly, the solution to \eqref{eq:uk2} for $k>0$ verifies
\begin{align}\label{recurrencecubic}
\ft{h}_k(\ell,t) &= \frac{1}{\sqrt{g |\ell|}} \int_0^t f(\ell,s) \sin\left( \sqrt{g|\ell|} (t-s)\right) \,ds \,,
\end{align}
and
\begin{align}\label{recurrencecubic2}
\p_{\!t} \ft{h}_k(\ell,t) = \int_0^t f(\ell,s) \cos\left( \sqrt{g|\ell|} (t-s)\right) \,ds\,,
\end{align}
where  $f$ is the Fourier transform of the (linear) forcing
\begin{align}
\widecheck{f} &= \sum_{j=0}^{k-1}\left(
\Lambda \left[h_{j}\Lambda h_{k-1-j}  - H\p_{\!t}h_{j}H\p_{\!t}h_{k-1-j}\right]  + g\p_{\!1} \left[ h_{j}\bp h_{k-1-j}\right]
\right)\nonumber\\
&\quad +\sum_{j=0}^{k-2}\sum_{r=0}^{j}\Q(h_r,h_{j-r},h_{k-2-j})\,,\label{eq:forcingwavecubic}
\end{align}
where $\Q$ is given by \eqref{eq:Q}.

\begin{theorem}\label{theorem4}
Let $\epsilon>0$, $g > 0$ and the initial data in \eqref{initialdata} $(\hinit, \htinit)$ be given. Assume that there exists $1<D<\infty$ such that
$$
\hinit(x_1)=\sum_{|\ell|\leq D}\fhinit(\ell)e^{i\ell x_1}\,,
$$
$$
\htinit(x_1)=\sum_{|\ell|\leq D,\ell\neq0}\fhtinit(\ell)e^{i\ell x_1}\,.
$$
Then there exists a unique analytic solution $$h(x_1,t)\in C([0,T];X^{0.5})$$ to {\rm(\ref{cubic-h-model})} for $t$ in the time interval $[0,T]$ with
$$
T=\frac{1}{e \epsilon 4C(\|\Lambda^{0.5}\hinit\|_{X_0},\|\htinit\|_{X_0})} \,.
$$
Equivalently, the Stokes expansion for the cubic $h-$model \eqref{cubic-h-model} converges for arbitrary $\epsilon>0$  and $T>0$ taken sufficiently small.
\end{theorem}

\begin{proof}
The proof of this Theorem is similar to the proof of Theorem \ref{theorem1}. As before, we fix $R\in \bbZ^+$ such that
$$
\frac{D(R+1)}{1+R^2}\leq 1,
$$
and consider $0<k\leq R$. We need to estimate $\|\Q(h_r,h_{j-r},h_{k-2-j})\|_{X_{R+1-k}}$. Using the previous ideas in the proof of Theorem \ref{theorem1} together with \eqref{eq:aux} and the trivial identity
$$
|\ell|\leq |\ell-m|+|m|\leq |\ell-m|+|m-n|+|n|,
$$
we have that
\begin{align*}
\|\Q(h_r,h_{j-r},h_{k-2-j})\|_{X_{R+1-k}}&\leq c_1(g)\bigg{[}\|\p_{\!t} h_r\|_{X_{R+2-k}}\|h_{j-r}\|_{X_{R+2-k}}\|\p_{\!t}h_{k-2-j}\|_{X_{R+2-k}}\\
&\quad+\|h_r\|_{X_{R+2-k}}\|h_{j-r}\|_{X_{R+2-k}}\|h_{k-2-j}\|_{X_{R+2-k}}\\
&\quad+\|h_r\|_{X_{R+2-k}}\|\p_{\!t}h_{j-r}\|_{X_{R+2-k}}\|\p_{\!t}h_{k-2-j}\|_{X_{R+2-k}}\bigg{]}\,.
\end{align*}
As before, we have that,  for $r\leq j\leq k-2$
$$
R+2-k=R-(k-2)\leq R+1-r,
$$
$$
R+2-k\leq R+1-j+r=R+1-(j-r),
$$
$$
R+2-k\leq R+1-(k-2-j).
$$
Thus, we can estimate
\begin{align*}
\|\Q(h_r,h_{j-r},h_{k-2-j})\|_{X_{R+1-k}}&\leq c_1(g)\bigg{[}\|\p_{\!t} h_r\|_{X_{R+1-r}}\|h_{j-r}\|_{X_{R+1-(j-r)}}\|\p_{\!t}h_{k-2-j}\|_{X_{R+1-(k-2-j)}}\\
&\quad+\|h_r\|_{X_{R+1-r}}\|h_{j-r}\|_{X_{R+1-(j-r)}}\|h_{k-2-j}\|_{X_{R+1-(k-2-j)}}\\
&\quad+\| h_r\|_{X_{R+1-r}}\|\p_{\!t}h_{j-r}\|_{X_{R+1-(j-r)}}\|\p_{\!t}h_{k-2-j}\|_{X_{R+1-(k-2-j)}}\bigg{]}\,.
\end{align*}
Recalling \eqref{eq:aux2}, we find that
\begin{align*}
& \|h_k(t)\|_{X_{R+1-k}}+\|\p_{\!t} h_k(t)\|_{X_{R+1-k}}\nonumber\\
&\qquad\leq c_2(g)\int_0^t\bigg{[}\sum_{j=0}^{k-1}\bigg{[}\|\p_{\!t}h_{j}(s)\|_{X_{R+1-j}}\|\p_{\!t}h_{k-1-j}(s)\|_{X_{R+1-(k-1-j)}}\nonumber\\
&\qquad+\|h_{j}(s)\|_{X_{R+1-j}}\|h_{k-1-j}(s)\|_{X_{R+1-(k-1-j)}}\bigg{]}\\
&\qquad+\sum_{j=0}^{k-2}\sum_{r=0}^{j}\bigg{[}\|\p_{\!t} h_r\|_{X_{R+1-r}}\|h_{j-r}\|_{X_{R+1-(j-r)}}\|\p_{\!t}h_{k-2-j}\|_{X_{R+1-(k-2-j)}}\\
&\qquad+\|h_r\|_{X_{R+1-r}}\|h_{j-r}\|_{X_{R+1-(j-r)}}\|h_{k-2-j}\|_{X_{R+1-(k-2-j)}}\\
&\qquad+\| h_r\|_{X_{R+1-r}}\|\p_{\!t}h_{j-r}\|_{X_{R+1-(j-r)}}\|\p_{\!t}h_{k-2-j}\|_{X_{R+1-(k-2-j)}}\bigg{]}\bigg{]}\,ds\,.
\end{align*}
We define
$$
\mathscr{A}_k(t)=\max\left\{c_2(g),1\right\}e^{-\frac{k+1}{1+\sgn(k)R^2}D(R+1)}\left[\|h_k(t)\|_{X_{R+1-k}}+\|\p_{\!t} h_k(t)\|_{X_{R+1-k}}\right].
$$
Then, we can conclude that
$$
\mathscr{A}_k(t)\leq \int_0^t \sum_{j=0}^{k-1} \mathscr{A}_{k-1-j}(s)\mathscr{A}_{j}(s)+\sum_{j=0}^{k-2}\sum_{r=0}^{j} \mathscr{A}_{k-2-j}(s)\mathscr{A}_{j-r}(s)\mathscr{A}_{r}(s) \,ds.
$$
We assume that $t<1$. Recalling \eqref{eq:aux3}, and the fact that for the Catalan numbers \eqref{Catalan} we have that
\begin{align*}
\sum_{j=0}^{k-2}\sum_{r=0}^{j}\mathcal{C}_r\mathcal{C}_{j-r}\mathcal{C}_{k-2-j}&=\sum_{j=0}^{k-2}\mathcal{C}_{j+1}\mathcal{C}_{k-1-(j+1)}\\
&=\sum_{n=1}^{k-1}\mathcal{C}_{n}\mathcal{C}_{k-1-n}\\
&\leq\sum_{n=1}^{k-1}\mathcal{C}_{n}\mathcal{C}_{k-1-n}+\mathcal{C}_{0}\mathcal{C}_{k-1}\\
&\leq \mathcal{C}_{k},
\end{align*}
we can prove by induction that
$$
\mathscr{A}_k\leq  2^k\mathcal{C}_kt^{k+1}\left[C(\|\Lambda^{0.5}\hinit\|_{X_0},\|\htinit\|_{X_0})\right]^{k+1}.
$$
Using this bound, we can conclude the existence and uniqueness as in Theorem \ref{theorem1}.
\end{proof}

\section{The Craig-Sulem \emph{WW2} model}\label{sec:CSWW2}
Zakharov  \cite{zakharov1968stability} formulated the water waves problem as the following system of one-dimensional nonlinear and nonlocal equations:
\begin{subequations}\label{Zak1phase}
\begin{alignat}{2}
\partial_t h&=G(h)\Psi\\
\partial_t \Psi&= -g h-\frac{1}{2}|\bp \Psi|^2+\frac{1}{2}\frac{(\bp  h\bp\Psi+G(h)\Psi)^2}{1+|\bp h|^2},
\end{alignat}
\end{subequations}
where $h(x_1,t)$ is the free surface, $\Psi(x_1,t)$ is the trace of the velocity potential $\u=\nabla\phi$ on the free surface
$$
\Psi(x_1,t)=\phi(x_1,h(x_1,t),t),
$$
and $G(h)$ is the Dirichlet-Neumann operator
\begin{equation}\label{D2N}
G(h)\Psi(x_1,t)=\frac{\partial\phi}{\partial_{x_2}}\bigg{|}_{(x_1,h(x_1,t),t)}-\bp h(x_1,t)\frac{\partial\phi}{\partial_{x_1}}\bigg{|}_{(x_1,h(x_1,t),t)}.
\end{equation}

As a way to numerically simulate the evolution of water waves when surface tension is neglected, Craig and Sulem \cite{CrSu1993}  gave
a  power series expansion for the Dirichlet-to-Neumann operator \eqref{D2N} as\footnote{This type of expansion for the Dirichlet-to-Neumann operator was first used in electromagnetism  by Milder \cite{milder1991improved} and Milder \& Sharp \cite{milder1992improved}.}
\begin{equation}\label{seriesD2N}
G(h)=\sum_{j=0}^\infty G_j(h),
\end{equation}
with
$$
G_0=\Lambda,
$$
$$
G_j(h)=-\Lambda^{j-1}\partial_1\frac{h^j}{j!}\partial_1-\sum_{i=0}^{j-1}\Lambda^{j-i}\frac{h^{j-i}}{(j-i)!}G_i(h).
$$

By keeping terms up to certain order in the previous expansion \eqref{seriesD2N} and starting from the Zakharov formulation \eqref{Zak1phase}, Craig and Sulem obtained a hierarchy of new truncated series models of the water waves problem. For instance, when we keep the terms up to second order, $G_0$ and $G_1$, we obtain the WW2 (\emph{water waves 2}) system
\begin{subequations}\label{CraigSulem}
\begin{alignat}{2}
\partial_t h&=\Lambda \Psi-\bp\left(\comm{H}{h}\Lambda\Psi\right)\\
\partial_t \Psi&= -g h+\frac{1}{2}\left((\Lambda\Psi)^2-(\bp \Psi)^2\right).
\end{alignat}
\end{subequations}

We define
\begin{equation}\label{omegadefinition}
\omega(x_1,t)=  \epsilon \omega_0(x_1,t) + \epsilon ^2 \omega_1(x_1,t) \,.
\end{equation}
Similarly, defining $h,\omega$ as in \eqref{hdefinition} and \eqref{omegadefinition}, respectively, and using \eqref{w0w1} we have that
$$
\frac{1}{2} \p_{\!1} \big(|\p_t h_0|^2 - |\omega_0|^2 \big)=\frac{1}{2} \p_{\!1} \big(|H\omega_0|^2 - |\omega_0|^2 \big).
$$
Then, neglecting terms of order $O(\epsilon^3)$ and considering surface tension effects, from \eqref{w0w1} we obtain the following coupled transport equations
\begin{subequations}\label{hmodel-quadtransport2}
\begin{align}
\p_{\!t} h &=  H\omega + \partial_1\left(\comm{h}{H}H\omega\right)\,,\\
\p_{\!t} \omega &=  -g\partial_1 h+\lambda \partial_{1}^3 h+ \Lambda \big(\omega H \omega \big).
\end{align}
\end{subequations}
These equations are the WW2 system obtained by Craig \& Sulem writen in the variable $\omega=\partial_1\Psi.$ Thus, our method is also able to recover the WW2 system. Similarly, following our previous ideas we can prove the following result for the WW2 Craig-Sulem system:
\begin{theorem}\label{theorem1b} Let $g > 0$ and the initial data $(\hinit, \omega_{\operatorname{init}})$ be given. Assume that there exists $1<D<\infty$ such that
$$
\hinit(x_1)=\sum_{|\ell|\leq D}\fhinit(\ell)e^{i\ell x_1}\,,
$$
$$
\omega_{\operatorname{init}}(x_1)=\sum_{|\ell|\leq D,\ell\neq0}\ft{\omega}_{\operatorname{init}}(\ell)e^{i\ell x_1}\,.
$$
Then there exists a unique analytic solution $(h(x_1,t),\omega(x_1,t))$ to (\ref{hmodel-quadtransport2}) for $t$ in the time interval $[0,T]$ with
$$
T=\frac{1}{\epsilon C(\hinit,\omega_{\operatorname{init}})} \,.
$$
\end{theorem}

In the following we are going to write the WW2 Craig-Sulem model as a wave equation. For an arbitrary function $f$, we define the operator
$$
\mathscr{T}f=\partial_1\comm{H}{h}f.
$$
The following inequalities hold
\begin{align*}
\|\mathscr{T}f\|_{L^2}&\leq C\|\partial_1 h\|_{\dot{H}^1}\|f\|_{L^2},\\
\|\mathscr{T}^2f\|_{L^2}&\leq C\|\partial_1 h\|_{\dot{H}^1}\|\mathscr{T}f\|_{L^2}\leq (C\|\partial_1 h\|_{\dot{H}^1})^2\|f\|_{L^2},\\
\|\mathscr{T}^kf\|_{L^2}&\leq C\|\partial_1 h\|_{\dot{H}^1}\|\mathscr{T}^{k-1}f\|_{L^2}\leq...\leq (C\|\partial_1 h\|_{\dot{H}^1})^k\|f\|_{L^2}.
\end{align*}
We define the following Neumann series
$$
\mathscr{N}=\sum_{k=0}^\infty \mathscr{T}^k.
$$
Then, if $\|\partial_1 h\|_{\dot{H}^1}C<1$ we have that
$$
\|\mathscr{N}f\|_{L^2}\leq \|f\|_{L^2}\sum_{k=0}^\infty (C\|\partial_1 h\|_{L^2})^k\leq \widetilde{C}(\|\partial_1 h\|_{L^2})\|f\|_{L^2},
$$
so, denoting by where $I$ the identity operator, we have that $I-\mathscr{T}$ is invertible and
$$
(I-\mathscr{T})^{-1}=\mathscr{N}.
$$
We observe that (\ref{hmodel-quadtransport2}a) is equivalent to
$$
(I-\mathscr{T})^{-1}\partial_t h= H\omega
$$
Using the previous operators, we find the following equivalent formulation of the Craig-Sulem WW2 model as a nonlinear wave equation:
\begin{equation}\label{ww2}
\partial_t^2h= -g\Lambda h-\lambda\Lambda^3 h+\bp(\partial_thH\partial_th)-\bp\comm{H}{\partial_th}\partial_th+g\bp\comm{H}{h}\Lambda h+\mathcal{P}
\end{equation}
where the cubic and higher nonlinearities are contained in
\begin{align}
\mathcal{P}&=\bp\left(H\partial_th\mathscr{M}\partial_t h\right)+
\bp\left(H\mathscr{M}\partial_th\left(\partial_th+\mathscr{M}\partial_th\right)\right)
-\bp\left(\comm{H}{\partial_th}\left(\mathscr{M}h_t\right)\right)\nonumber\\
&\quad+\bp\left(\comm{H}{h}\left(\bp\left(\left(-H\partial_th-H\mathscr{M}\partial_th\right) \left(\partial_th+\mathscr{M}\partial_th\right)\right)\right)\right)\,,\label{Q}
\end{align}
and the operator $\mathscr{M}$ is defined as
$$
\mathscr{M}=\sum_{k=1}^\infty \bp\comm{H}{h}^k=\mathscr{N}+I\,.
$$

In particular, we observe that, when the cubic and higher nonlinearities in $\mathcal{P}$ are neglected, the Craig-Sulem \emph{WW2} model reduces to the quadratic $h-$model \eqref{hmodel-quad} (or the, so-called, ``Model" by Akers and Milewski \cite{AkMi2010}).

\section{Estimating the difference between the $h-$models and the solution of the full water waves problem}\label{sec:error}
In this section we   estimate the error of solutions of the $h$-models to solutions of the full water waves system.

Let $(\hinit, \htinit)$ be a $\mathcal{O}(\epsilon)$ initial data and consider its corresponding local solution to the full water waves 
problem $(h^{ww},\omega^{ww})$ in $C([0,T];X^1)$. As we described in the introduction, the well-posedness of the  water waves problem is well-known (see the works by Ovsjannikov \cite{ovsjannikov1974shallow} and Shinbrot \cite{shinbrot1976initial} for the case with analytic initial data), and that solutions exists for a lifespan $T=\mathcal{O}(\epsilon^{-1})$. We have the following 

\begin{theorem}\label{theorem5} Let $\epsilon>0$, $g > 0$ and the initial data $(\hinit, \htinit)$ be given. Assume that there exists $1<D<\infty$ such that
$$
\hinit(x_1)=\epsilon\sum_{|\ell|\leq D}\fhinit(\ell)e^{i\ell x_1}\,,
$$
$$
\htinit(x_1)=\epsilon\sum_{|\ell|\leq D,\ell\neq0}\fhtinit(\ell)e^{i\ell x_1}\,.
$$
Denote by $(h^{ww},\omega^{ww})$ the local solution in $C([0,T(\epsilon)];X^1)$ of the full water waves problem starting from the initial data $(\hinit, \htinit)$ and 
let $h^{qm}$ denote the solution to the quadratic $h-$model {\rm(\ref{hmodel-quad})}. Then, as long as both solutions exist, 
$$
\|h^{ww}-h^{qm}\|_{C([0,T];X^{0.5})}\leq \mathcal{O}(\epsilon^{3}).
$$
\end{theorem}
\begin{proof}
From \cite{ovsjannikov1974shallow} and\cite{shinbrot1976initial}, there exists analytic solutions to the full water waves problem; hence, we write the
solution $h^{ww}$ as an asymptotic series
We have that
\begin{equation}\label{rss1}
h^{ww}(x_1,t)=\epsilon\sum_{j=0}^\infty \epsilon^j h_j^{ww}(x_1,t).
\end{equation} 
It follows that each term $h_j$ 
evolves according to \eqref{phi_recursion}, \eqref{hk_ww} and \eqref{wk_ww}. 
We have  also shown that
\begin{equation}\label{rss2}
h^{qm}(x_1,t)=\epsilon\sum_{j=0}^\infty \epsilon^j h_j^{qm}(x_1,t),
\end{equation}  
with $h_j^{qm}$ evolving according to \eqref{eq:uk}.  

It follows from \eqref{rss1} that
$$
\sup_{0\leq t\leq T}\|h^{ww}-\epsilon\sum_{j=0}^1 \epsilon^j h_j^{ww}\|_{X^{0.5}}\leq \mathcal{O}(\epsilon^{3}) \,,
$$
and from \eqref{rss2},
$$
\sup_{0\leq t\leq T}\|h^{qm}-\epsilon\sum_{j=0}^1 \epsilon^j h_j^{qm}\|_{X^{0.5}}\leq \mathcal{O}(\epsilon^{3}) \,.
$$
We have to estimate
$$
\|\sum_{j=0}^1 \epsilon^j h_j^{ww}-\sum_{j=0}^1 \epsilon^j h_j^{qm}\|_{X^{0.5}}.
$$
From \eqref{eq:uk} and \eqref{w0w1}, we  have that
$$
h_0^{ww}-h_0^{qm}=0 \,,
$$
and from  (\ref{h0h1}b) and \eqref{eq:uk}, we also have that
$$
h_1^{ww}-h_1^{qm}=0 \,.
$$
Thus, the terms in each series only begin to deviate at $O( \epsilon ^3)$, which establishes the result.
\end{proof}

Analogously, we have that
\begin{theorem}\label{theorem6} Let $\epsilon>0$, $g > 0$ and the initial data $(\hinit, \htinit)$ be given. Assume that there exists $1<D<\infty$ such that
$$
\hinit(x_1)=\epsilon\sum_{|\ell|\leq D}\fhinit(\ell)e^{i\ell x_1}\,,
$$
$$
\htinit(x_1)=\epsilon\sum_{|\ell|\leq D,\ell\neq0}\fhtinit(\ell)e^{i\ell x_1}\,.
$$
Denote by $(h^{ww},\omega^{ww})$ the local solution in $C([0,T(\epsilon)];X^1)$ of the full water waves problem starting from the initial data $(\hinit, \htinit)$ and 
let $h^{cm}$ denote the solution to the cubic $h-$model \eqref{cubic-h-model}. Then, as long as both solutions exist, 
$$
\|h^{ww}-h^{cm}\|_{C([0,T],X^{0.5})}\leq \mathcal{O}(\epsilon^{4}).
$$
\end{theorem}
\begin{proof}
The proof follows  as in Theorem \ref{theorem5} by noting that for the cubic $h$-model
$$
h_2^{ww}-h_2^{cm}=0\,,
$$
and hence the deviation in the series representations of the two solutions occurs at $O( \epsilon ^4)$.
\end{proof}

\section{Numerical comparison of water waves and the $h$-model}\label{sec:numerics}

In this section we compute solutions of the quadratic and cubic
$h$-models and compare them to numerical solutions of the Euler
equations. We find that the linear, quadratic and cubic $h$-models
converge at the expected rates as $\eps\rightarrow0$, and show
regimes where the quadratic model captures the essential features of
the wave beyond the linear regime, and where the cubic model captures
features beyond the quadratic regime. We also observe that the
quadratic model can form corner singularities in finite time, while
the cubic model can evolve to an unstable state where high-frequency
Fourier modes of the solution start growing rapidly. This only causes
problems for large-amplitude waves on excessively fine grids.

\subsection{Solving the Euler equations}

To evolve the full water wave equations, we use the spectrally
accurate boundary integral method developed by Wilkening
\cite{wilkening:standing-waves-2011} and Wilkening and Yu
\cite{wilkening-yu:shooting-2012} for computing standing water waves.
While a conformal mapping approach
\cite{dyachenko:96a,dyachenko:96b,milewski:10} is usually easier to
implement, the result would have to be re-parametrized to be
equally-spaced in $x$ in order to compare with the $h$-model. This is
not particularly difficult, but the boundary integral method is more
natural in this setting. We write the Euler equations in the form
\begin{equation}
  \label{eq:euler}
  \begin{aligned}
    h_t &= \phi_y - h_x\phi_x, \\
    \varphi_t &= P\left[ \phi_yh_t - \frac12\phi_x^2 - \frac12\phi_y^2 - gh +
      \frac\lambda\rho\pa_x\left(\frac{h_x}{\sqrt{1+h_x^2}}\right)\right],
  \end{aligned}
\end{equation}
where $\varphi(x,t) = \phi(x,h(x,t),t)$ is the restriction of the
velocity potential to the free surface, $\lambda$ is the surface
tension parameter (set to zero in this section), and $P$ is the
projection onto zero mean in $L^2(0,2\pi)$. Only $h(x,t)$ and
$\varphi(x,t)$ are evolved in time since $\phi(x,y,t)$ can be computed
from $\varphi(x,t)$ using (\ref{eq:cauchy}) below. The velocity
components $u=\phi_x$, $v=\phi_y$ on the free surface are computed
from $\varphi$ as follows. We identify $\mbb{R}^2$ with $\mbb{C}$ and
attempt to represent the complex velocity potential
$\Phi(z)=\phi(z)+i\psi(z)$ as a Cauchy integral
\begin{equation}\label{eq:cauchy}
  \Phi(z) = \frac1{2\pi i}\int_0^{2\pi} \frac{\zeta'(\alpha)}2\cot\frac{\zeta(\alpha)-z}2
  \mu(\alpha)\,d\alpha, \qquad
  \zeta(\alpha) = \alpha + ih(\alpha), \;\; 0\le\alpha<2\pi,
\end{equation}
where $\mu(\alpha)$ is real-valued and we have suppressed $t$ in the
notation. Here we used $\alpha=x$ to parametrize the horizontal
component of the free surface, but the formulas in this section
generalize to allow for mesh refinement or overturning waves if one
writes $\zeta(\alpha)=\xi(\alpha)+i h(\alpha)$.  The cotangent kernel
comes from summing the Cauchy kernel over periodic images
\begin{equation}
  \frac12\cot\frac z2 = PV\sum_{k\in\mbb Z} \frac1{z+2\pi k}, \qquad\quad
  (PV = \text{ principal value}).
\end{equation}
Letting $z$ approach $\zeta(\alpha)$ from below and using the Plemelj
formula \cite{muskhelishvili:sing} gives
\begin{equation}
  \label{eq:Phi1}
  \Phi(\zeta(\alpha)^-) = -\frac12\mu(\alpha)+\frac i2H\mu(\alpha) + \frac1{2\pi i}\int_0^{2\pi}
  K(\alpha,\beta)\mu(\beta)\,d\beta,
\end{equation}
where
\begin{equation}
  \label{eq:K:ab}
  K(\alpha,\beta) =
  \frac{\zeta'(\beta)}2\cot\frac{\zeta(\beta)-\zeta(\alpha)}2 -
  \frac12\cot\frac{\beta-\alpha}2.
\end{equation}
The second term of $K$ is included to cancel the singularity of the
first term, which makes $K(\alpha,\beta)$ continuous at
$\alpha=\beta$, with $K(\alpha,\alpha) =
\zeta''(\alpha)/[2\zeta'(\alpha)]$. In fact, the components of $K$ are
real analytic, periodic functions of $\alpha$ and $\beta$ on $\mbb
R/2\pi\mbb Z$ if $\zeta(\alpha)$ (i.e.~$h(\alpha)$) is real-analytic
and periodic.  Including this term in $K(\alpha,\beta)$ is accounted
for in (\ref{eq:Phi1}) by the Hilbert transform term, using
\begin{equation}
  H f(\alpha)=\frac1\pi PV\!\int_0^{2\pi}\frac12\cot\frac{\alpha-\beta}2f(\beta)\,d\beta.
\end{equation}
The real part of (\ref{eq:Phi1}) gives a second-kind Fredholm integral equation
\cite{folland:pde} that can be solved for $\mu$ given $\varphi$,
\begin{equation}\label{eq:mu}
  -\frac12\mu(\alpha) + \frac1{2\pi}\int_0^{2\pi}\im\{K(\alpha,\beta)\}\mu(\beta)\,d\beta
  = \varphi(\alpha), \qquad 0\le\alpha<2\pi.
\end{equation}
Differentiating (\ref{eq:cauchy}), integrating by parts, and using a
standard argument for principal value integrals to handle the
interchange of $\alpha$ and $\beta$ in the kernel, one may show
\cite{wilkening-yu:shooting-2012} that
\begin{equation}
  \label{eq:Phi:z1}
  \zeta'(\alpha)\Phi_z(\zeta(\alpha)^-) = 
  -\frac12\mu'(\alpha)+\frac i2H\mu'(\alpha) - \frac1{2\pi i}\int_0^{2\pi}
  K(\beta,\alpha)\mu'(\beta)\,d\beta.
\end{equation}
Since $\phi_x-i\phi_y = \phi_x+i\psi_x = \Phi_z$, (\ref{eq:Phi:z1})
gives an explicit formula for $\phi_x$ and $\phi_y$ on the free
surface once $\mu(\alpha)$ is known from (\ref{eq:mu}). Equations
(\ref{eq:mu}) and (\ref{eq:Phi:z1}) are easily discretized with
spectral accuracy using the trapezoidal rule on a uniformly spaced grid
\begin{equation}
  \alpha_j = 2\pi j/M, \qquad j=0,\dots,M-1
\end{equation}
to compute integrals, and the Fourier transform to compute derivatives
and the Hilbert transform (with symbol $\hat H_k =
  -i\opn{sgn}(k)$). For example, (\ref{eq:mu}) becomes
\begin{equation}
  -\frac12\mu_i + \frac1M\sum_{j=0}^{M-1}\im\{K(\alpha_i,\alpha_j)\}\mu_j = \varphi_i, \qquad
  i=0,\dots,M-1,
\end{equation}
where we recall that $K(\alpha_i,\alpha_i)=\zeta''(\alpha_i)/[2\zeta'(\alpha_i)]$.
We timestep (\ref{eq:euler}) using an 8th order Runge-Kutta method due to
Dormand and Prince \cite{hairer:I,dormand:81}.  We also need $h_t$ in the comparison
to the $h$-model, but this formula is part of the right-hand side of (\ref{eq:euler}).

\subsection{Timestepping the $h$-model}

Next we describe an effective method of timestepping the $h$-model
(linear, quadratic or cubic). First, we write it as a first-order
system of the form $u_t=Lu+N(u,t)$, which for the cubic case is
\begin{equation}\label{eq:hmodel:etd}
  \underbrace{\der{}{t}\begin{pmatrix} h \\ h_t \end{pmatrix}}_{\jd u_t} =
  \underbrace{\begin{pmatrix} 0 & P \\ -g\Lambda & 0 \end{pmatrix}
  \begin{pmatrix} h \\ h_t \end{pmatrix}}_{\jd Lu} +
  \underbrace{\begin{pmatrix} P_0(h_t) \\ -\Lambda\big[(Hh_t)^2\big]
      + g\pa_x\big(hh_x\big) + g\Lambda\big(h\Lambda h\big)
    + Q(h) \end{pmatrix}}_{\jd N(u,t)}.
\end{equation}
For the quadratic model, we drop $Q(h)$, and for the linear model, the entire
second component of $N$ is set to zero.  Here
\begin{equation}
  Pf(x) = f(x) - P_0f, \qquad P_0f = \frac1{2\pi}\int_0^{2\pi}f(x)\,dx
\end{equation}
are the orthogonal projections onto zero mean, and onto the constant
functions, respectively. Though $P_0(h_t)$ is linear, it is convenient
to move it from $L$ to $N$ to avoid a Jordan block in the
diagonalization of $L$ (see below). We use the spectral exponential
time differencing scheme of Chen and Wilkening
\cite{chen-wilkening:setd}, which is an arbitrary-order,
fully-implicit variant of the popular fourth-order ETD scheme of Cox
and Matthews \cite{cox-matthews:etd-2002, kassam-trefethen:etd-2006},
to solve the stiff system (\ref{eq:hmodel:etd}). To evolve the
solution over a timestep, which, for simplicity, we take to be from
$t=0$ to $t=h$, we solve the Duhamel integral equation
\begin{equation}
  u(t) = e^{Lt}u_0 + \int_0^t e^{(t-\tau)L} N\big(\tau,u(\tau)\big)\,d\tau
\end{equation}
by collocation using a Chebyshev-Lobatto grid. In more detail, let
\begin{equation}
  t_j = c_j h, \qquad c_j = \frac{1-\cos(\pi j/\nu)}2, \qquad (j=0,\dots,\nu).
\end{equation}
Given $u_0$, we look for $u_1,\dots,u_\nu$ such that
\begin{equation}
  u_r = e^{t_rL}u_0 + \int_0^{t_r}e^{(t_r-\tau)L}\sum_{j=0}^\nu N(t_j,u_j)\ell_j(\tau/h)\,d\tau,
  \qquad (r=1,\dots,\nu),
\end{equation}
where $l_j(s)=\prod_{k\ne j}\frac{s-c_k}{c_j-c_k}$ are
the Lagrange polynomials for the Chebyshev-Lobatto grid on $[0,1]$.
The change of variables $\tau=hs$, $d\tau=h\,ds$ then gives
\begin{equation}\label{eq:setd:eqs}
  u_r = e^{c_rhL}u_0 + h\sum_{j=0}^\nu \left(\int_0^{c_r}e^{(c_r-s)hL}\ell_j(s)\,ds\right) N(c_jh,u_j),
  \qquad (r=1,\dots,\nu),
\end{equation}
which is a nonlinear system of equations that can be solved
efficiently using a Newton-Krylov solver; see
\cite{chen-wilkening:setd} for details. The algorithm in
\cite{chen-wilkening:setd} is designed so the user only has to supply
routines to apply $U$, $S$ and $U^{-1}$ to arbitrary vectors, where
$L=USU^{-1}$. Internally, when the Newton-Krylov solver needs to apply
$e^{c_rhL}$ and $\int_0^{c_r} e^{(c_r-s)hL}\ell_j(s)\,ds$ to a
sequence of vectors, it does so by asking the user to apply only $U$,
$S$ and $U^{-1}$. This makes implementing the method on new problems
straightforward as long as $L$ can be diagonalized efficiently.

In our case, $L$ is diagonalized by the Fourier transform, as we now
explain. Let $\mc{F}$ be the ``r2c'' version of the Fast Fourier
Transform, which maps
\begin{equation}
  V \ni \mat{ u_0 \\ u_1 \\ \vdots \\ u_{M-1} } \;\overset{\mc F}\longmapsto \;
  \mat{ \hat u_0 + i\hat u_{M/2} \\ \hat u_1 \\ \vdots \\ \hat u_{M/2-1} } \in \hat V, \qquad
  \hat u_k = \frac1M\sum_{j=0}^{M-1} u_j e^{-2\pi ijk/M}.
\end{equation}
Here we assume $M$ is even, and we note that $\hat u_0$ and $\hat u_{M/2}$ are real
since $e^{-2\pi ijk/M}\in\{1,-1\}$ when $k=0$ or $k=M/2$. The ``missing'' Fourier modes
are known implicitly from $\hat u_{-k} = \overline{\hat u_k}$. The mapping $\mc F$ is
an isometry of real vector spaces if we endow $V$ and $\hat V$ with the inner products
\begin{equation}
  \big\la u, v\big\rangle = \frac1M\sum_{j=0}^{M-1}u_jv_j, \qquad
  \big\la \hat u,\hat v\big\rangle = \hat u_0\hat v_0 + \hat u_{M/2}\hat v_{M/2} +
  \sum_{k=1}^{M/2-1} 2\opn{Re}\{ \overline{\hat u_k}\hat v_k \}.
\end{equation}
To diagonalize $L$, we note that both $\Lambda$ and $P$ in
(\ref{eq:hmodel:etd}) kill constant functions, and we define the
finite-dimensional truncations of $\Lambda$ and $P$ to
also kill the Nyquist mode $u_j=(-1)^j$. Thus
\begin{equation}
L = \mat{ \mc F^{-1} \\ & \mc F^{-1} }
\mat{ 0 & E \\ -gK & 0 }\mat{ \mc F \\ & \mc F },
\end{equation}
where $E = \diag[0,1,\dots,1]$, $K = \diag[0,1,2,\dots,M/2-1]$, and
multiplying a vector in $\hat V$ by $E$ or $K$ via complex arithmetic
is still linear when $\hat V$ is regarded as a real vector space.
The inner matrix can be diagonalized into $2\times2$ blocks by a
permutation matrix
\begin{equation}\label{eq:diag:EgK}
  \mat{ 0 & E \\ -gK & 0 } =
  \mat{ & E_e & \\ & E_o & }
  \mat{ A_0 \\ & \ddots \\ & & A_{M/2-1} }
  \mat{ E_e^T & E_o^T }, \;\;
  \begin{array}{c}
    A_0=\mat{ 0 & 0 \\ 0 & 0 }, \\[10pt]
    A_{k\ge1}=\mat{ 0 & 1 \\ -gk & 0 },
    \end{array}
\end{equation}
%
where  $E_{e,ij} = \delta_{2i,j}$, $E_{o,ij} =
\delta_{2i+1,j}$ for $0\le i<M/2$, $0\le j<M$.  Left-multiplication
  by $E_e$ or $E_o$ selects the even-index or odd-index rows,
  respectively; right-multiplication by $E_e^T$ or $E_o^T$ selects
  even or odd-index columns; and applying $(E_e^T,E_o^T)$ to $[\hat h;\hat
    h_t]$ interlaces the components of $\hat h\in\hat V$ and $\hat h_t\in\hat V$,
  so that $\pa_t\hat h_k$ follows $\hat h_k$. Finally,
$A_0$ is already diagonal while $A_k=Q_kS_kQ_k^{-1}$ with
\begin{equation}\label{eq:diag:cx}
   Q_k = \mat{ 1 & 1 \\ i\sqrt{gk} & -i\sqrt{gk} }, \quad
  S_k = \mat{ i\sqrt{gk} \\ & -i\sqrt{gk} }, \quad
  Q_k^{-1} = \frac12\mat{1 & -i/\sqrt{gk} \\ 1 & i/\sqrt{gk} }.
\end{equation}
The complex numbers in $E$, $K$, $E_e$, $E_o$, $A_k$, $Q_k$,
$Q_k^{-1}$ and $S_k$ actually represent real $2\times2$
matrices with the identification
\begin{equation}\label{eq:diag:real}
  \alpha+i\beta \longleftrightarrow
  \mat{ \alpha & -\beta \\ \beta & \alpha }. \qquad \text{Example:} \quad
  S_k = \mat{ 0 & \!-\sqrt{gk} \\ \sqrt{gk} & 0 \\ & & 0 & \sqrt{gk} \\ & & -\sqrt{gk} & 0 }.
\end{equation}
Treating the entries of $\hat V$ as complex numbers rather than
flattening $\hat V$ to $\bbR^M$ by interlacing real and imaginary
parts is convenient, but gets confusing in the last step when
complex eigenvalues arise. The final step of diagonalizing $L$ (had we
  flattened $\hat V$) would be to diagonalize the real matrix $S_k$ in
(\ref{eq:diag:real}), which would lead to a pair of double eigenvalues
$\pm i\sqrt{gk}$. But applying any power series to $S_k$ in
(\ref{eq:diag:cx}) and then flattening will give the same result as
applying the power series directly to $S_k$ in
(\ref{eq:diag:real}). In particular, $e^{c_rhS_k}$ and
$e^{(c_r-s)hS_k}$, which are needed to compute $e^{c_rhL}$ and
$e^{(c_r-s)hL}$ in (\ref{eq:setd:eqs}), can be computed either
way. This justifies not flattening $\hat V$, and cuts the number of
eigenvalues that are explicitly dealt with in half --- each
double-eigenvalue in (\ref{eq:diag:real}) appears only once in
(\ref{eq:diag:cx}).

Note that moving $P_0(h_t)$ over to $N(u,t)$ in (\ref{eq:hmodel:etd})
was necessary to avoid a Jordan block in $A_0$ in (\ref{eq:diag:EgK}).
We also remark that normally one wants to include the highest-order
differential operators in $L$, but in our case they are nonlinear and
depend on time, so this was not possible. However, the method still
does not suffer from severe CFL constraints since fully implicit
Runge-Kutta schemes based on Lobatto quadrature are $L$-stable
\cite{hairer:I}. The above method reduces to such a scheme when $L=0$,
and we would not expect instabilities to arise by separating the
linear part of the operator into a Duhamel-based formulation.

\subsection{Comparison of water waves and the $h$-model}

As a first test, we consider the family of solutions $h(x,t)=\eps\htil(x,t)$,
$\varphi(x,t) = \eps \phitil(x,t)$ with
\begin{equation}
  \label{eq:ex1}
  \text{Example 1:} \qquad
  \htil(x,0) = \frac15\sin x + \frac1{10}\sin(2x) + \frac15\sin(3x), \qquad
  \phitil(x,0) = 0.
\end{equation}
The maximum slope of the initial wave profile $h(x,0)$ occurs at the
origin, and is equal to~$\eps$. The wave starts at rest and evolves
under the influence of gravity.  The solution of the full Euler
equations for $\eps=5/3$ for $0\le t\le 0.625$ is shown in
Figure~\ref{fig:ex1}($a$), along with the spatial Fourier mode
amplitudes (panel $b$) of $h(x,t)$ at the times shown in panel
($a$). Only positive index Fourier modes are shows since
$c_{-k}=\overline{c_k}$. A 3072-point spatial grid was used, with 720
uniform timesteps of the DOPRI8 Runge-Kutta method
\cite{hairer:I,dormand:81}.  Every 18th step was recorded (at $t=k/64$,
  $0\le k\le 40$).  At $t=0.625=40/64$, a jet begins to form in each
of the troughs, with the lowest trough containing the strongest jet.

\begin{figure}[p]
  \begin{center}
\includegraphics[width=.99\linewidth]{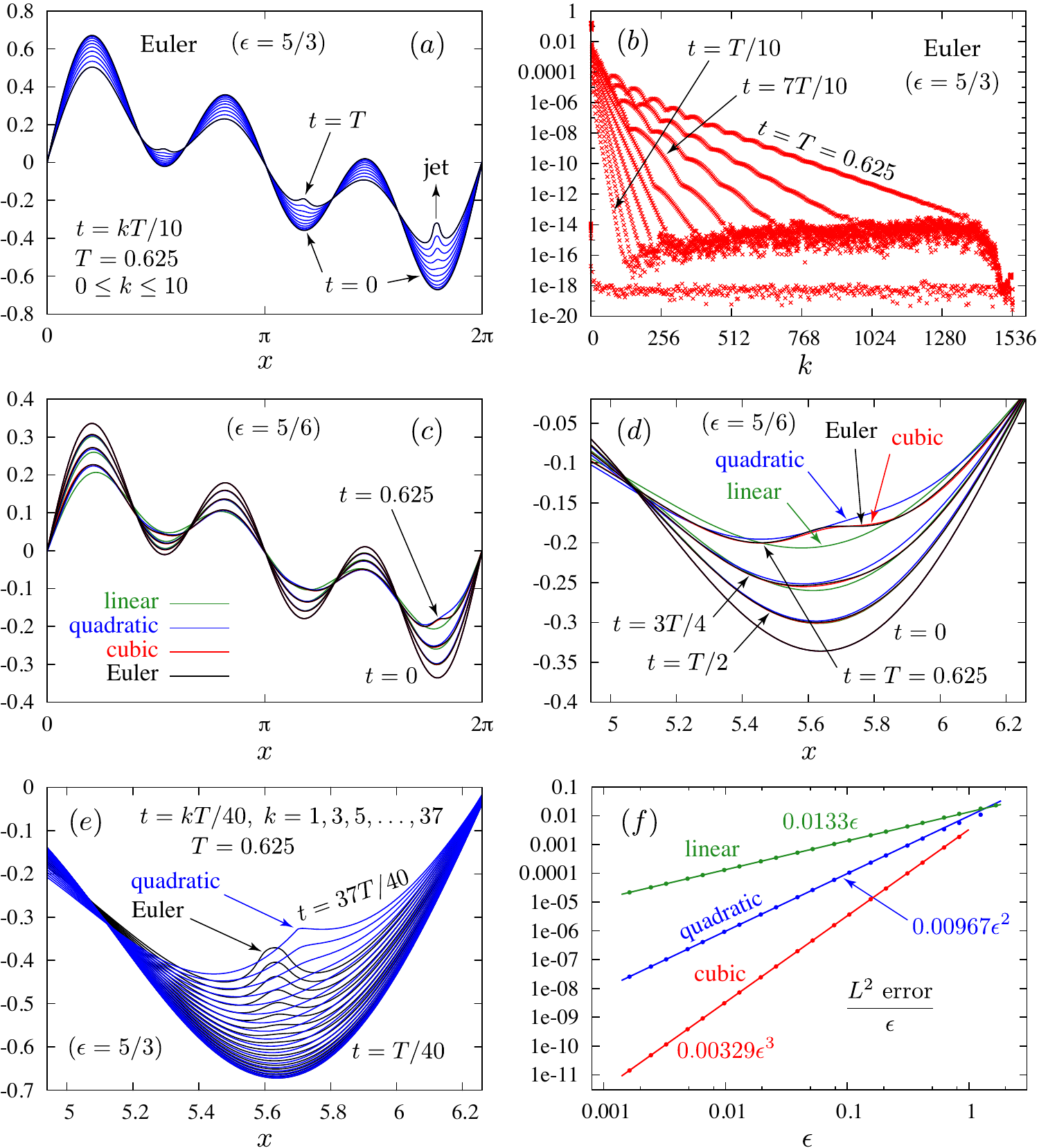}
\caption{\label{fig:ex1} Comparison of the linear, quadratic and cubic
  $h$-models to the full Euler equations for Example 1. ($a$)
  Large-amplitude Euler solution with $\eps=5/3$.  ($b$) Amplitude of
  Fourier modes for this solution. ($c$ and $d$) The cubic model
  remains close to the true solution for $\eps=5/6$ and $0\le
  t\le0.625$, whereas the linear and quadratic models miss key
  features. ($e$) The quadratic model forms a corner in the $\eps=5/3$
  case. ($f$) Relative $L^2$-errors versus $\eps$.}
\end{center}
\end{figure}

Panels $(c)$ and $(d)$ of Figure~\ref{fig:ex1} compare the solutions
of the linear, quadratic and cubic $h$-models with that of the full
Euler equations with $\eps=5/6$ in (\ref{eq:ex1}) at
$t\in\{0,\frac12T,\frac34T,T\}$, where $T=40/64$. At this amplitude,
both the linear and quadratic models miss the bulge in the lowest
trough as the jet begins to form, whereas the cubic model captures it
closely. If $\eps$ is doubled to $\eps=5/3$ (as in panel $a$), niether
the cubic nor quadratic models can be evolved all the way to
$t=40/64$. Panel ($e$) of Figure~\ref{fig:ex1} shows that the
quadratic model (on a 3072-point grid) appears to form a corner
singularity around $t=37/64$, and is far from the corresponding Euler
solution at this time. For the cubic model (evolved on a 1024-point
  grid), high-frequency Fourier modes begin to grow at $t=12/64$. By
$t=30/64$, roundoff errors in these high-frequency modes have been
amplified to be comparable in size to the leading modes.  The solution
completely blows up shortly afterwards, with values on the grid
jumping from $O(1)$ at $30/64$ to $O(10^{250})$ at 31/64. Increasing
the number of timesteps by a factor of 1000 did not change the time at
which the instability begins or the growth rate of the modes, so this
is not likely a CFL issue. However, increasing the spatial grid size
does affect the blow-up time since higher-frequency modes grow
faster. We omit a figure showing this for Example 1 as similar
behavior is observed in Example 3 below. Panel ($f$) of
Figure~\ref{fig:ex1} shows the $L^2$ error of the linear, quadratic
and cubic models at $t=40/64$ versus~$\eps$, where the $L^2$ errors
have been scaled by $\eps^{-1}$ to account for the decreasing norm
of the exact solution. As expected, these errors decay as
$O(\eps^k)$, where $k=1$ for the linear $h$-model, $k=2$ for the
quadratic $h$-model, and $k=3$ for the cubic $h$-model.

The second example we consider consists of an initial bulge over a
flat surface evolving from rest. More specifically, we consider the
family of functions
\begin{equation}
  \text{Example 2:}\qquad \tilde h_n(x,0) = \frac2n\left(1+\frac1{n^2}\right)^{\frac{n^2-1}2}
  \left[\sin^{n^2+1}\frac x2 - \frac{\Gamma((n^2/2)+1)}{\sqrt\pi\,\Gamma((n^2/2)+(3/2))}\right],
\end{equation}
where $n\in\{1,3,5,7,\dots\}$. The constants were chosen so that
$\tilde h_n(x,0)$ has zero mean and maximum slope $\pm1$, occuring
where $\tan(x/2)=\pm n$. The cases $n=7$ and $n=25$ are shown in panel
($a$) of Figure~\ref{fig:ex2}.  Panel ($b$) shows that the $L^2$ error
at $t=6$, scaled by $\eps^{-1}$, decays at the expected order as
$\eps\rightarrow0$ for $n=7$ and $n=25$. For $n=7$, the best-fit lines
shown are $0.0326\eps$, $0.0156\eps^2$ and $0.00476\eps^3$. For
$n=25$, they are $0.00542\eps$, $0.00337\eps^2$ and $0.000524\eps^3$,
which are smaller than in the $n=7$ case. This is not surprising as
the $L^2$-norm of the underlying wave is also smaller when $n=25$.

Panels ($c,d,e$) of Figure~\ref{fig:ex2} compare solutions of the
Euler equations with those of the linear, quadratic and cubic
$h$-models with $\eps=0.429$ and $n=25$ over $0\le t\le 1$ (panel
  $c$), $1\le t\le 2$ (panel $d$), and $5\le t\le 6$ (panel $e$).  The
linear model already deviates substantially from the exact solution by
$t=0.4$ (panel $c$), when the initial bulge is still accelerating
downward.  The quadratic and cubic models remain close to the Euler
solution throughout the evolution to $t=6$ (panel $e$), correctly
damping out the wave near the origin and propagating the correct
number of ripplies outward in both directions. The quadratic model
develops a sharper crest at $t=1$ (panel $d$) than the Euler solution,
which also sharpens somewhat at this time. For both equations, the
wave becomes smoother again.  This can be seen in panel ($f$), where
the Fourier mode amplitudes decay more slowly at $t=1$ than at $t=0$
or $t=6$. The minimum decay rate for both equations happens near
$t=1$.  The quadratic model has roughly 6 times as many active modes
as the Euler solution at $t=1$ due to the excessive sharpening at the
crest observed in panel ($d$). At later times (e.g.~panel~$e$), the
quadratic model retains remnants of the overly sharp crest that formed
at $t=1$, with smaller-scale features visibly deviating from the exact
solution (though the overall wave profiles are similar.) The cubic
model is nearly indistinguishable from the Euler model at the
resolution of the graphs in panels ($d$) and ($e$).  It has about
twice as many active Fourier modes as the Euler solution at $t=1$ and
$t=6$, as shown in panel ($g$).  The Euler modes are the same in
panels ($f$) and ($g$), and all three equations have the same Fourier
coefficients at $t=0$ in these plots.

\begin{figure}[p]
  \begin{center}
\includegraphics[width=.95\linewidth]{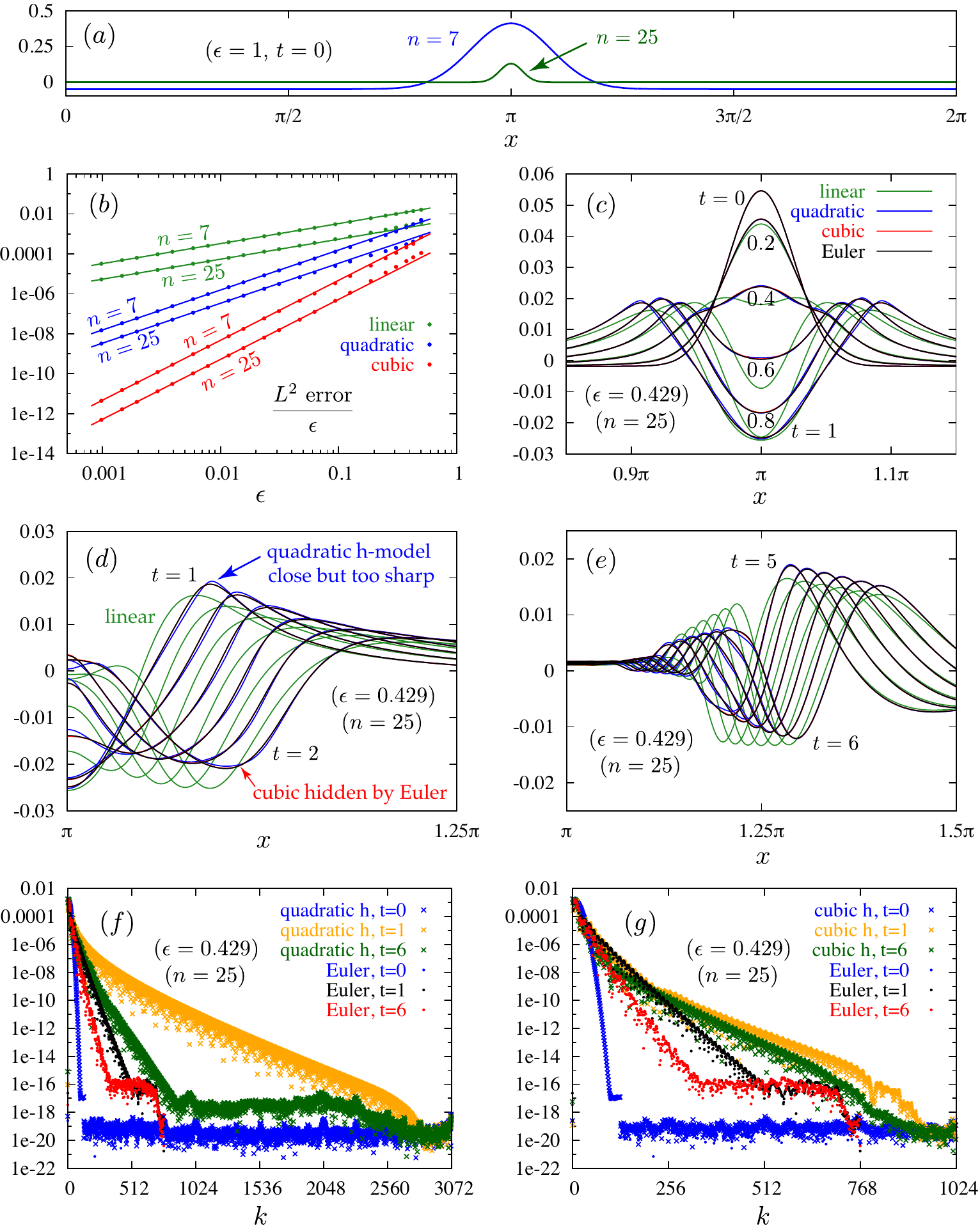}
\caption{\label{fig:ex2} Comparison of the linear, quadratic and cubic
  $h$-models to the full Euler equations for Example 2.  ($a$) Initial
  wave profiles before scaling by $\eps$. ($b$) Relative $L^2$-errors
  versus~$\eps$. ($c,d,e$) Solutions with $n=25$ and $\eps=0.429$ at
  the times shown. ($f,g$) Comparison of Fourier mode amplitudes at
  $t=0,1,6$. The sharper wave crest in the solution of the quadratic model
  at $t=1$ leads to slower mode decay than the Euler or cubic solutions.}
\end{center}
\end{figure}

Our third example consists of a family of standing water waves
computed using the overdetermined shooting method described in
\cite{wilkening:standing-waves-2011,wilkening-yu:shooting-2012}. Unlike
the previous two examples, the waves in this family are not related by
a simple scaling of the initial condition via $h(x,0)=\eps\tilde
h(x,0)$, $\varphi(x,t)=\eps\tilde\varphi(x,t)$.
%
%
In the previous examples, we chose $\tilde h(x,0)$ to have maximum slope 1 so
that $\eps$ was the maximum slope of $h(x,0)$.  For standing waves, we
match this latter property:
\begin{equation}\label{eq:stand:eps}
  \text{Example 3 (standing waves):}\qquad \eps = \text{
    maximum slope of $h(x,0)$.}
\end{equation}
Here we assume the fluid is initially at rest.
As before, we choose the length-scale so that the spatial period
is $2\pi$ after non-dimensionalization. Let $T$ (which depends on
  $\eps$) be half the temporal period of the standing wave so that
the wave comes to rest when $t\in T\bbZ$. At even multiples of $T$,
the wave crests are assumed to be located at $x\in2\pi\bbZ$, and
at odd multiples they are located at $x\in(\pi+2\pi\bbZ)$.

Characterizing the amplitude of the wave by its maximum slope is
useful for comparing with Examples 1 and 2, but it is not the most
convenient for actually computing standing waves. In the numerical
algorithm of Wilkening and Yu\cite{wilkening-yu:shooting-2012}, a
Fourier coefficient of the initial condition was used as the
bifurcation parameter. For low-amplitude waves, the initial amplitude
of the fundamental mode is a natural choice.  Yet another choice is
half the maximum crest to trough height ($\frac12$CTH). In all three
cases (slope, mode amplitude, or $\frac12$CTH), assuming $g=1$,
\begin{equation}\label{eq:eps:stand}
  h(x,t) = \eps\cos x\cos t + O(\eps^2).
\end{equation}
Building on previous work \cite{rayleigh:1876,penney:52,tadjbakhsh},
Schwartz and Whitney \cite{schwartz:81} developed a recursive
algorithm to compute the power series expansion for standing water
waves of this type in conformal variables, and computed the first 25
terms. Their choice of amplitude was $\frac12$CTH, which we denote by
$\veps$. While the full representation of the wave profile and
velocity potential is too complicated to reproduce here, we can report
the leading terms of the period and maximum slope:
\begin{equation}
  \frac T{\pi}= 1+\frac12\veps^2-\frac7{256}\veps^4+O(\veps^6), \qquad
  \eps = \veps + \frac12\veps^3+O(\veps^5).
\end{equation}
Amick and Toland proved that the terms in the Schwartz/Whitney
expansion are uniquely determined to all orders \cite{amick:87}, but
the question of whether the series has a positive radius of
convergence remains open.  Recent work using Nash-Moser theory has
been able to establish existence on a Cantor set of the bifurcation
parameter close to zero-amplitude \cite{iooss:05}.  Regardless of the
eventual convergence or divergence of the series, truncating the
series yields a family of initial conditions (over a range
  $0\le\veps\le\veps_\text{max}\approx0.06$ when 25 terms are
  retained) that return to their starting configurations to within
machine precision when evolved under the Euler equations. The shooting
method in \cite{wilkening-yu:shooting-2012} gives solutions that agree
with the Schwartz and Whitney series to all 16 digits at small
amplitude, but is not limited to such a narrow range of $\veps$ to
find solutions that are time-periodic to machine precision.

\begin{figure}[p]
  \begin{center}
\includegraphics[width=.99\linewidth]{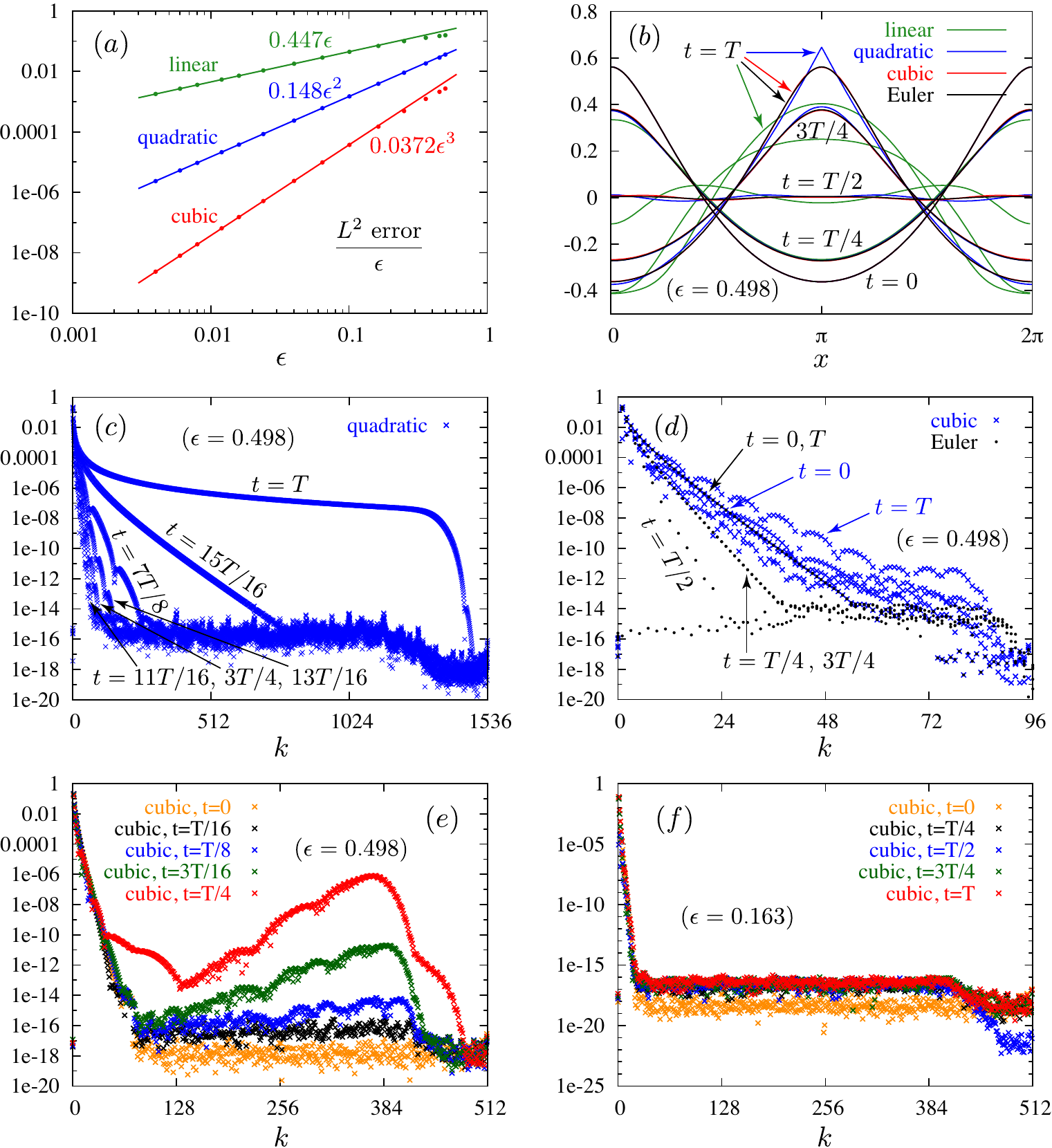}
\caption{\label{fig:ex3} Comparison of the linear, quadratic and cubic
  $h$-models to the full Euler equations for Example 3. ($a$) Relative
  $L^2$-errors versus $\eps$. ($b$) Snapshots of the solutions at
  $t=0,T/4,T/2,3T/4,T$ for $\eps=0.498$. The quadratic model nearly
  forms a corner at $t=T$. ($c$) The formation of a corner causes
  high-frequency Fourier mode amplitudes to grow as $t\rightarrow T$
  in the quadratic model. ($d$) The cubic model and Euler equations
  remain well-resoved with 192 Fourier modes. (Only 96 are shown since
    $c_{-k}=\overline{c_k}$.)  ($e,f$) Over short times or small
  amplitude, the cubic model remains well-posed; however, for
  $\eps=0.498$, with 1024 modes, the cubic model loses stability for
  $t>T/16$. The solution completely blows up shortly after $t=T/4$.}
\end{center}
\end{figure}

For each standing wave computed by the shooting method, we find the
maximum slope via Newton's method to determine $\eps$. We then evolve
the $h$-models using the initial conditions of the standing wave and
compare them to the Euler solution at $t=T$. Panel (a) of
Figure~\ref{fig:ex3} shows that the relative errors in the linear,
quadratic and cubic models decay at the expected rates.
Panel (b) shows snapshots of the solutions of the $h$-models and the
Euler equations for the $\eps=0.498$ wave at $t=\frac14T$, $\frac12T$,
$\frac34T$ and $t=T$. This choice of $\eps$ was the largest (among the
  waves we computed) in which the solution of the quadratic $h$-model
remains regular for $0\le t\le T$. We see in panel (b) that the
quadratic model nearly forms a corner at $t=T$, which also leads to
slow decay of its Fourier modes in panel (c) as $t$ approaches
$T$. The Euler solution returns to a spatial phase shift (by $\pi$) of
its initial condition to 14 digits. Its Fourier modes decay to
$10^{-14}$ by $k=60$. We used 192 gridpoints in the computation.  We
also used 192 grid points to evolve the cubic model. The solution
remains well-resolved in Fourier space over this time (panel d), and
remains nearly indistinguishable from the Euler solution at the
resolution of panel (b) over $0\le t\le T$. We also note in panel (b)
that the solution of the quadratic model remains close to the Euler
solution until $t=3T/4$, but the linear model already deviates
substantially near $x=0$ and $x=2\pi$ at $t=T/4$. It remains accurate
in the trough at least until $t=T/4$, but is completely wrong
throughout the domain by $t=T/2$.

At this large amplitude ($\eps=0.498$), the cubic model relies on
Fourier truncation to remain well-posed. In panel (e) of
Figure~\ref{fig:ex3}, we increase the number of gridpoints from 192 to
1024 with the same initial data as in panel (d), and find that
high-frequency modes begin to grow rapidly shortly after
$t=T/16$. This picture is independent of the number of timesteps taken
--- increasing the number of timesteps by a factor of 1000 led to a
similar picture (not shown), except that applying the filter 1000
times as often led to slight suppression of the mode amplitudes in the
range $300\le k\le 512$. Thus, around $t=T/16$, the solution of the
cubic model appears to evolve to a state where the PDE ceases to be
well-posed. By contrast, the solution of the quadratic model does not
show signs of instability regardless of the grid size until $t=T$ ---
the growth in mode amplitudes in panel (c) is due to formation of a
geometric singularity rather than ill-posedness.  In panel (f), we see
that the cubic model remains well-posed over the whole interval $0\le
t\le T$ for a smaller amplitude wave ($\eps=0.163$). Here again we
used 1024 gridpoints, even though 48 would have been sufficient to
fully resolve the solution spectrally. We observed similar behavior
in Examples 1 and 2, where large-amplitude waves were found to form
corners at their crests in the quadratic model, and caused the
solution to leave the realm of well-posedness for the cubic
model. (The sharpening feature in Figure~\ref{fig:ex2}d for the
  quadratic model forms a corner at larger amplitude).

\begin{figure}[p]
  \begin{center}
\includegraphics[width=.99\linewidth]{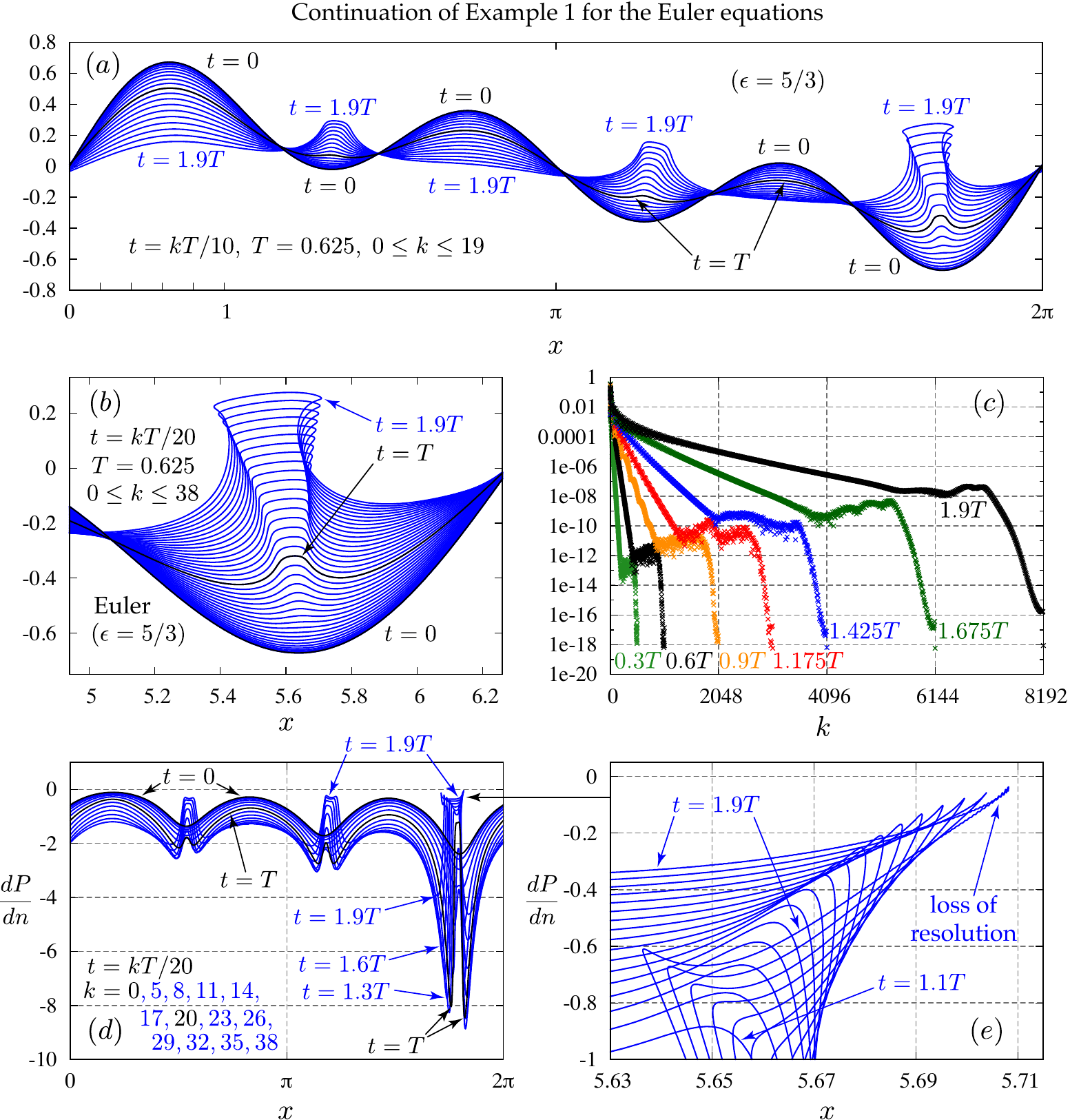}
\caption{\label{fig:ex1:overturn} ($a$) Evolution of the $\eps=5/3$
  wave with initial condition (\ref{eq:ex1}) to $t=1.9T$, where
  $T=0.625$ is the final time in Figure~\ref{fig:ex1}. The aspect
  ratio is 1:1. ($b$) A closer look at the jet forming in the lowest
  trough, which overturns on the left around $t=1.335T$ and on the
  right around $t=1.343T$. ($c$) Amplitude of the Fourier modes at the
  times shown. The spatial grid was refined 6 times in the course of
  the evolution. ($d,e$) The normal derivative of pressure versus $x$
  as time evolves.  The computed value of $dP/dn$ runs out of
  precision  but approaches zero near the tip of
  the overturning wave at $t=1.9T$, indicating corner formation.}
\end{center}
\end{figure}


In Figure~\ref{fig:ex1:overturn}, we return to the solution of the
full Euler equations for Example 1 with $\eps=5/3$. Here we switch to
an angle-arclength parametrization of the free surface
\cite{hls:94,hls:97,hls:01,ambrose-wilkening:vortex-sheet-2014}, which
allows for overturning waves. We continue to define $T=0.625$ and
evolve to $t=1.9T$. Panel ($a$) shows that the jets that were
beginning to form in the troughs at $t=T$ grow in height to become the
tallest points on the free surface at $t=1.9T$. The jet from the
lowest trough overtakes that of the middle trough around $t=1.4562T$,
and is on track to overtake that of the highest trough around
$t=1.9451T$, where we extrapolated from the last 4 timesteps.  Panel
($b$) shows a close-up of the jet from the lowest trough, which widens
and flattens out as it decelerates, causing the wave to overturn on
both sides of the jet. The overturn times are $t=1.335T$ on the left
and $1.343T$ on the right. Panel ($c$) shows the aplitude of the
Fourier modes as the solution evolves. The grid was refined 6 times,
from 1024 gridpoints at the beginning to 16384 at the end. Roundoff
errors become larger as the grid is refined due to increased
cancellation in the formula (\ref{eq:K:ab}) for $K(\alpha,\beta)$.
Each grid refinement also leads to some growth in high-frequency modes
that were being suppressed by the filter on the coarser mesh and
suddenly are not.  (We use the 36th order filter of Hou and Li
  \cite{hou:li:07}).

Once the wave overturns, there are 3 possible outcomes. It could
return to being single-valued, which sometimes happens after a vortex
sheet with surface tension overturns
\cite{ambrose-wilkening:vortex-sheet-2010,ambrose-wilkening:vortex-sheet-2014},
but seems unlikely here as there is no physical mechanism to slow down
the overturning wave.  It could self-intersect in a splash singularity
\cite{castro2012splash,coutand2014finite}. Or it could form a corner
at the tip of the overturning wave, similar to the way the quadratic
$h$-model tends to form singularities. This would coincide with
$dP/dn$ approaching zero at the corner, so that the Rayleigh-Taylor condition
$dP/dn<0$ ceases to hold.  Panels ($d,e$) of Figure~\ref{fig:ex3} show
$dP/dn$ plotted parametrically versus~$x$ at various times. We see
that indeed, $dP/dn$ appears to be increasing to 0 at the tip of each
overturning wave. However, computing $dP/dn$ involves taking a
derivative of the solution, and we were not able to maintain enough
digits of accuracy in double-precision to definitively say that
$dP/dn$ reaches zero. The Rayleigh-Taylor condition plays a key role in the local well-posedness
of the water wave problem in the absence of
surface tension (see the references given in the introduction). 
Further
investigation will be pursued in future work, where we will provide
details of the method for tracking overturning waves and computing
$dP/dn$. Our main point in this example is to show that the type of
breakdown we observe in the quadratic $h$-model, where the
solution forms a geometric singularity, may occur in the Euler equations as well.

\section*{Acknowledgements}
JW was supported by  NSF
 DMS-1716560 and by the Department of Energy, Office of
Science, Applied Scientific Computing Research, under award number
DE-AC02-05CH11231. RGB was partially funded by University of Cantabria and the Departement of Mathematics, Statistics and Computation.
SS was supported by NSF DMS-1301380,  the Department of Energy, Advanced Simulation and Computing (ASC) Program, and
by  DTRA HDTRA11810022.

\appendix
\section{Basic commutator identities}\label{appendix1}
In deriving the cubic $h$-model, we make use of the following identities:
\begin{align*}
\Lambda\big[(H \p_{\!t} h) \Lambda \big(\comm{h}{H} \p_{\!t} h\big)\big] - \Lambda \big[(H \p_{\!t} h) (\p_{\!1} h) (\p_{\!t} h)\big]
&=\Lambda\big[(H \p_{\!t} h) \Lambda \big(h H \p_{\!t} h\big)\big] + \Lambda \big[(H \p_{\!t} h)  h (\p_{\!1}\p_{\!t} h)\big] \,, \\
- \p_{\!1} \big(\comm{\p_{\!t} h}{H} \p_{\!1} \big(\comm{h}{H} \p_{\!t} h \big)\big)- \p_{\!1} \big(\comm{\p_{\!t} h}{H} \Lambda (h \p_{\!t} h)\big)
&=- \p_{\!1} \big(\comm{\p_{\!t} h}{H} \p_{\!1} \big(h H \p_{\!t} h \big)\big)\,, \\
\p_{\!1} \big(\comm{h}{H} \p_{\!1} (\p_{\!t} h H \p_{\!t} h) \big)  - \p_{\!1} \big[\comm{h}{H} \Lambda (\p_{\!t} h)^2 \big]
&=\p_{\!1} \big(\comm{h}{H} \p_{\!1} (\comm{\p_{\!t} h}{H} \p_{\!t} h) \big)\,, \\
-H \big[(\p_{\!1} h) h \partial_1^2 h\big]+\dfrac{1}{2} \p_{\!1} \big[h^2 (\Lambda \p_{1} h)\big]  - \dfrac{1}{2} H (h^2 \partial_1^3h)&=\frac{1}{2}\partial_1\comm{h^2}{H}\partial_1^2h,
\end{align*}
and
\begin{align*}
& H \big[(\p_{\!1} \p_{\!t} h) \p_{\!1}(h H \p_{\!t} h)\big] + H \big[(\p_{\!1} h) \p_{\!1}\big((\p_{\!t} h) (H \p_{\!t} h)\big)\big] - 2 H \big[(\p_{\!1}h) (\p_{\!1} \p_{\!t} h) (H \p_{\!t} h)\big] \\
&\quad + H \big[h (\p_{\!t} h) \p_{\!1} \Lambda \p_{\!t} h\big] \\
&\qquad = H \big[(\p_{\!1} \p_{\!t} h) (\p_{\!1} h) (H \p_{\!t} h)\big] + H \big[h (\p_{\!1} \p_{\!t} h) (\Lambda \p_{\!t} h)\big] + H \big[(\p_{\!1} h) (\p_{\!1} \p_{\!t} h) (H \p_{\!t} h)\big] \\
&\qquad\quad + H \big[(\p_{\!1} h) (\p_{\!t} h) (\Lambda \p_{\!t} h)\big] - 2 H \big[(\p_{\!1}h) (\p_{\!1} \p_{\!t} h) (H \p_{\!t} h)\big] + H \big[h (\p_{\!t} h) \p_{\!1} \Lambda \p_{\!t} h\big] \\
&\qquad = H \big[\p_{\!1} \big((h \p_{\!t} h) (\Lambda \p_{\!t} h) \big)\big] \\
&\qquad = \Lambda \big[(h \p_{\!t} h) (\Lambda \p_{\!t}h)\big]\,.
\end{align*}


\bibliographystyle{plain}

\end{document}